\documentclass[10pt]{article}

\usepackage{customize}

\title{A Singular Woodbury and Pseudo-Determinant Matrix Identities and Application to Gaussian Process Regression}

\author[1,2,3]{Siavash Ameli\protect\thanks{Email address: \href{mailto:sameli@berkeley.edu}{\protect\nolinkurl{sameli@berkeley.edu}}}~}
\author[1]{Shawn C. Shadden\protect\thanks{Email address: \href{mailto:shadden@berkeley.edu}{\protect\nolinkurl{shadden@berkeley.edu}}}}
\affil[1]{\small\textit{Mechanical Engineering, University of California, Berkeley, CA 94720, USA}}
\affil[2]{\small\textit{Department of Statistics, University of California, Berkeley, CA 94720, USA}}
\affil[3]{\small\textit{International Computer Science Institute, Berkeley, CA 94704, USA}}
\date{}
\newcommand{\spc}[1]{\mathcal{#1}}

\newcommand{\image}{\operatorname{im}}
\newcommand{\coimage}{\operatorname{coim}}
\newcommand{\coker}{\operatorname{coker}}
\newcommand{\spn}{\operatorname{span}}
\newcommand{\pdet}{\operatorname{pdet}}
\newcommand{\logdet}{\operatorname{logdet}}

\newcommand{\diag}{\operatorname{diag}}
\newcommand{\ind}{\operatorname{ind}}
\newcommand{\bigO}{\mathcal{O}}
\newcommand{\tri}{\(\triangleright\)\ }
\newcommand{\indep}{\perp \!\!\! \perp}  % Creates independent symbol

\makeatletter
\renewcommand{\theAlgoLine}{%
\@arabic{\numexpr\value{algocf}\relax}.\@arabic{\numexpr\value{AlgoLine}\relax}}
\makeatother

\crefalias{AlgoLine}{line}%
\makeatletter
\let\cref@old@stepcounter\stepcounter
\def\stepcounter#1{%
  \cref@old@stepcounter{#1}%
  \cref@constructprefix{#1}{\cref@result}%
  \@ifundefined{cref@#1@alias}%
    {\def\@tempa{#1}}%
    {\def\@tempa{\csname cref@#1@alias\endcsname}}%
  \protected@edef\cref@currentlabel{%
    [\@tempa][\arabic{#1}][\cref@result]%
    \csname p@#1\endcsname\csname the#1\endcsname}}
\makeatother

\makeatletter
\renewcommand*{\@algocf@pre@ruled}{\hrule height\heavyrulewidth depth0pt \kern\belowrulesep}
\renewcommand*{\algocf@caption@ruled}{\box\algocf@capbox\kern\aboverulesep\hrule height\lightrulewidth\kern\belowrulesep}
\renewcommand*{\@algocf@post@ruled}{\kern\aboverulesep\hrule height\heavyrulewidth\relax}
\makeatother

\def\CC{{C\nolinebreak[4]\hspace{-.05em}\raisebox{.4ex}{\tiny\bf ++}}}

\DeclareFixedFont{\ttb}{T1}{txtt}{m}{n}{9} % for bold
\DeclareFixedFont{\ttm}{T1}{txtt}{m}{n}{9}  % for normal

\definecolor{deepblue}{rgb}{0, 0, 0.5}
\definecolor{deepred}{rgb}{0.6, 0, 0}
\definecolor{deepgreen}{rgb}{0.4, 0.4, 0.4}
\definecolor{cmnt}{HTML}{778899}

\lstdefinestyle{myStyle}{
    language=Python,
    basicstyle=\small\ttfamily,
    morekeywords={self},              % Add keywords here
    keywordstyle=\small\color{deepblue},
    emph={InterpolateTraceinv, detkit},          % Custom highlighting
    emphstyle=\small\ttfamily\bfseries\color{deepred},    % Custom highlighting style
    commentstyle=\small\color{cmnt}\ttfamily\itshape,
    stringstyle=\color{deepgreen},
    frame=tb,                         % Any extra options here
    showstringspaces=false,
    escapeinside={;\#}{\#;},           % add LaTeX within the comment lines
    escapebegin=\color{cmnt},        % color of the escape
    numberstyle = \tiny\color{black},
    numbers=left,
    firstnumber=1,
    stepnumber=5,
    columns=flexible,
}

\ifpdf
\hypersetup{
  pdftitle={A Singular Woodbury Matrix Identity and Application to Gaussian Process Regression},
  pdfauthor={S. Ameli}
}
\fi

\begin{document}

\maketitle

\begin{abstract}
    We study a matrix that arises from a singular form of the Woodbury matrix identity. We present generalized inverse and pseudo-determinant identities for this matrix, which have direct applications for Gaussian process regression, specifically its likelihood representation and precision matrix. We extend the definition of the precision matrix to the Bott-Duffin inverse of the covariance matrix, preserving properties related to conditional independence, conditional precision, and marginal precision. We also provide an efficient algorithm and numerical analysis for the presented determinant identities and demonstrate their advantages under specific conditions relevant to computing log-determinant terms in likelihood functions of Gaussian process regression.
\end{abstract}


% \MSC[2020] 15A10 \sep 15-04 \sep 62G08
% \begin{MSCcodes}
% 15A10, 15-04, 62G08
% \end{MSCcodes}

\section{Introduction}
The Woodbury matrix identity and the matrix determinant lemma are the fundamental relations respectively for the inverse and determinant of the sum of two matrices. Consider the matrix
\begin{equation}
    \tens{N} = \tens{A} + \tens{X} \tens{B} \tens{Y}^{\ast}, \label{eq:N-with-B}
\end{equation}
where \(\tens{A}\) and \(\tens{B}\) are invertible matrices and \(\tens{X}\) and \(\tens{Y}\) are matrices of conformable size. We assume all matrices are defined over the complex field and \((\cdot)^{\ast}\) denotes conjugate transpose. The Woodbury matrix identity \citep[p. 427]{HARVILLE-1997} represents the inverse of the above matrix, \(\tens{M} \coloneqq \tens{N}^{-1}\), whenever it exists, by
\begin{equation}
    \tens{M} = \tens{A}^{-1} - \tens{A}^{-1} \tens{X} \left(\tens{Y}^{\ast} \tens{A}^{-1} \tens{X} + \tens{B}^{-1} \right)^{-1} \tens{Y}^{\ast} \tens{A}^{-1}. \label{eq:M-with-B}
\end{equation}
Also, the matrix determinant lemma represents the determinant of \(\tens{N}\) by
\begin{equation}
    \det(\tens{N}) = \det(\tens{A}) \det(\tens{B}) \det(\tens{Y}^{\ast} \tens{A}^{-1} \tens{X} + \tens{B}^{-1}). \label{eq:det-lemma}
    % \vert \tens{N} \vert = \vert \tens{A} \vert \, \vert \tens{B} \vert \, \vert \tens{Y}^{\ast} \tens{A}^{-1} \tens{X} + \tens{B}^{-1} \vert. \label{eq:det-lemma}
\end{equation}
The analytical importance and computational advantage of the above identities are well-known, with a wide range of applications in statistics, partial differential equations, optimization, asymptotic analysis, and networks, to name a few \citep{HAGER-1989}.

Several extensions to the above relations exist. Among notable works, \citet{HENDERSON-1981} derived inversion identities when \(\tens{A}\) is singular, or when \(\tens{B}\) is singular or rectangular. An alternative identity was derived by \citet{RIEDEL-1992} for rank-augmented matrices where the rank of \(\tens{N}\) is larger than of \(\tens{A}\). The relation between the generalized inverse of \(\tens{N}\) and those of \(\tens{A}\) and \(\tens{B}\) was explored by \citet{FILL-1999} and \citet{GROB-1999}. Other notable generalizations of the Woodbury identity were studied by \citet{TIAN-2005, ARIAS-2015} under a rank additivity condition and by \citet{DENG-2011} in relation to Moore-Penrose and Drazin inverses.

A special case of \eqref{eq:M-with-B} is when \(\tens{B}^{-1} = \tens{0}\), which is relevant in a variety of applications, such as in machine learning using Gaussian process regression (see \Cref{sec:app} and \citep{AMELI-2020, AMELI-2022-a}). In such a case, \(\tens{M}\) is usually rank-deficient and the matrix \(\tens{N}\) is undefined, rendering Woodbury-like identities and the determinant lemma inapplicable. Thus, the work herein develops Woodbury-like relations when \(\tens{B}^{-1} = \tens{0}\). Given that only \(\tens{M}\) is well-defined in this scenario, our approach shifts the focus to searching for identities for \(\tens{M}\) instead of \(\tens{N}\). Moreover, we consider that \(\tens{A}\), \(\tens{X}\), and \(\tens{Y}\) could be rectangular and potentially have rank deficiency. Concretely, we consider
\begin{equation}
    \tens{M} \coloneqq \tens{A}^{\dagger} - \tens{A}^{\dagger} \tens{X} \left(\tens{Y}^{\ast} \tens{A}^{\dagger} \tens{X}\right)^{\dagger} \tens{Y}^{\ast} \tens{A}^{\dagger}, \label{eq:M}
\end{equation}
where \((\cdot)^{\dagger}\) denotes the Moore-Penrose pseudo-inverse as defined in \Cref{sec:notation}. Our contributions are as follows.
\begin{itemize}
    \item We show \(\tens{M}\) is the generalized \(\{2\}\)-inverse (or outer inverse) of \(\tens{A}\), and we obtain its pseudo-inverse. We also derive an expression for \(\tens{M}\), which reduces to the Bott-Duffin inverse of \(\tens{A}\) whenever \(\tens{X}\) and \(\tens{Y}\) have the same column space.
    \item We obtain the pseudo-determinant of \(\tens{M}\)  assuming \(\tens{A}\) is an equal-principal matrix. Such relation also leads to a useful pseudo-determinant identity for \(\tens{A}\).
\end{itemize}
We further present practical applications of the derived identities to Gaussian process regression. Namely,
\begin{itemize}
    \item We show that the likelihood function of a form of Gaussian process regression can be recognized by the normal distribution, and we define its precision matrix by the Bott-Duffin inverse of its covariance matrix. We show such an extended definition retains the existing properties of the precision matrix related to conditional independence and the conditional and marginal precisions of the partitioned data.
    \item The presented matrix inversion identities enable us to derive an asymptotic analysis for the Gaussian process regression with a mixed linear model.
    \item We analyze the computational complexity of the pseudo-determinant identities and empirically demonstrate their numerical advantages on certain matrices.
\end{itemize}
The presented algorithms were implemented by the python package \texttt{detkit} \citep{AMELI-2022-d}, which is a determinant toolkit that can be used to reproduce the numerical results of this paper. The results of this work were also employed by \texttt{glearn}, a high-performance python package for machine learning using Gaussian process regression \citep{AMELI-2022-b}.

The paper is organized as follows. In \Cref{sec:notation}, we provide a brief overview of the generalized inverses of matrices. Our main results are presented in \Cref{sec:main}. Applications to the Gaussian process regression are discussed in \Cref{sec:app}. We present a numerical analysis in \Cref{sec:num}. \Cref{sec:conclusion} concludes the paper. Supporting proofs of the main results can be found in \Cref{sec:pf}. The necessary formulations for Gaussian process regression are outlined in \Cref{sec:gpr}. The dataset used in our numerical analysis is described in \Cref{sec:dataset}.

% =================================
% Generalized Inverse, a Background
% =================================

% \section{Generalized Inverse, a Background}
\section{Preliminaries} \label{sec:notation}

We indicate by \(\mathcal{M}_{n, m}(\mathbb{C})\) the space of all \(n \times m\) matrices with entries over the field \(\mathbb{C}\). For the matrix \(\tens{A} \in \mathcal{M}_{n, m}(\mathbb{C})\), we denote \(\image(\tens{A})\), \(\ker(\tens{A})\), \(\coimage(\tens{A}) = \mathbb{C}^m / \ker(\tens{A})\), and \(\coker(\tens{A}) = \mathbb{C}^n / \image(\tens{A})\) respectively by its image (range), kernel (null space), coimage, and cokernel. Due to the isomorphism of the quotient map \(\mathbb{C}^n / \image(\tens{A}) \cong \image{(\tens{A})}^{\perp}\) where \((\cdot)^{\perp}\) denotes orthogonal complement, we identify \(\coker(\tens{A})\) with \(\image({\tens{A}})^{\perp}\). Similarly, \(\ker(\tens{A}) = \coimage(\tens{A})^{\perp}\). Note that \(\ker(\tens{A}) = \coker(\tens{A}^{\ast})\), \(\image(\tens{A}) = \coimage(\tens{A}^{\ast})\), and \(\dim(\image(\tens{A})) = \dim(\coimage(\tens{A})) = \rank(\tens{A})\).

For a subspace \(\spc{X} \subseteq \mathbb{C}^n\), we say a square matrix \(\tens{A}\) is \(\spc{X}\)-zero if \(\tens{A} \spc{X} \cap \spc{X}^{\perp} = \{\tens{0}\}\) \citep[Definition 1]{CHEN-2003}. We also say a Hermitian matrix \(\tens{A}\) is \(\spc{X}\)-PD if the restriction of \(\vect{x} \mapsto \vect{x}^{\ast} \tens{A} \vect{x}\) on \(\spc{X} \setminus \{\vect{0}\}\) is positive-definite \citep[Definition 1]{YONGLIN-1990}. Also, the matrix \(\tens{P}_{\spc{R}, \spc{N}}\) with the complementary subspaces \(\spc{R}, \spc{N} \subseteq \mathbb{C}^n\) denotes the oblique projector onto \(\spc{R}\) along \(\spc{N}\), that is, \(\image(\tens{P}_{\spc{R}, \spc{N}}) = \spc{R}\) and \(\ker(\tens{P}_{\spc{R}, \spc{N}}) = \spc{N}\). In particular, \(\tens{P}_{\spc{R}}\) denotes the orthogonal projection matrix \(\tens{P}_{\spc{R}, \spc{R}^{\perp}}\), which can be constructed by \(\tens{P}_{\spc{R}} = \tens{R} \tens{R}^{\dagger}\) where \(\image(\tens{R}) = \spc{R}\). We note that \(\tens{P}_{\spc{N}, \spc{R}} = \tens{I} - \tens{P}_{\spc{R}, \spc{N}}\) is the complement projection to \(\tens{P}_{\spc{R}, \spc{N}}\) where \(\tens{I}\) is the identity matrix. In particular, \(\tens{P}_{\spc{R}^{\perp}} = \tens{I} - \tens{P}_{\spc{R}}\). Also, \(\tens{P}_{\spc{R}, \spc{N}}^{\ast} = \tens{P}_{\spc{N}^{\perp}, \spc{R}^{\perp}}\). In particular, \(\tens{P}_{\spc{R}}\) is Hermitian.

We briefly overview the generalized inverses of matrices that we use in the subsequent development, and we refer the reader to \citep{ISRAEL-2003, STANIMIROVIC-2017, WANG-2018} for further details. For a given rectangular matrix \(\tens{A}\), consider the following equations in \(\tens{Z}\) as
\begin{equation}
    (1)~ \tens{A} \tens{Z} \tens{A} = \tens{A},\quad
    (2)~ \tens{Z} \tens{A} \tens{Z} = \tens{Z}, \quad
    (3)~ \left( \tens{A} \tens{Z} \right)^{\ast} = \tens{A} \tens{Z}, \quad
    (4)~ \left( \tens{Z} \tens{A} \right)^{\ast} = \tens{Z} \tens{A},
    \label{eq:penrose}
\end{equation}
which are attributed to \citet{PENROSE-1955}. The Moore-Penrose inverse of \(\tens{A}\) is defined by the unique solution \(\tens{Z}\) that satisfies the four conditions in \eqref{eq:penrose}. More generally, a non-unique solution \(\tens{Z}\) that satisfies only partial conditions \((i), (j), \dots, (k)\) among the Penrose conditions \((1)\)--\((4)\) in \eqref{eq:penrose} is called the \(\{i, j, \dots, k\}\)-inverse of \(\tens{A}\), and denoted by \(\tens{A}^{(i, j, \dots, k)}\). Namely, the \(\{2\}\)-inverse, \(\tens{A}^{(2)}\), which satisfies the second condition of \eqref{eq:penrose} is known as the outer inverse of \(\tens{A}\). The outer inverse has several applications in statistics \citep{HSUAN-1985, GETSON-1988} and is of particular interest to this manuscript. The outer inverse of a matrix \(\tens{A}\) is not unique; however, it can be uniquely determined by prescribing its image \(\spc{R} \coloneqq \image (\tens{A}^{(2)})\) and kernel \(\spc{N} \coloneqq \ker (\tens{A}^{(2)})\) if and only if
\begin{equation}
    \tens{A} \spc{R} \oplus \spc{N} = \mathbb{C}^n, \label{eq:unique-XY}
\end{equation}
\citep[p. 72, Theorem 14]{ISRAEL-2003}; in such case the outer inverse is denoted by \(\tens{A}^{(2)}_{\spc{R}, \spc{N}}\). The outer inverse unifies the representation of several known types of generalized inverses \citep{WEI-1998, CHEN-2000}. For instance, since \(\image(\tens{A}^{\dagger}) = \image(\tens{A}^{\ast})\) and \(\ker(\tens{A}^{\dagger}) = \ker(\tens{A}^{\ast})\), \(\tens{A}^{\dagger}\) is a form of the outer inverse by
% For instance, the Moore-Penrose inverse \(\tens{A}^{\dagger}\) and group inverse \(\tens{A}^{\sharp}\) are special forms of the outer inverse, respectively by
\begin{equation}
    \tens{A}^{\dagger} = \tens{A}^{(2)}_{\image(\tens{A}^{\ast}), \ker(\tens{A}^{\ast})}.
    % \quad
    % \tens{A}^{\sharp} = \tens{A}^{(2)}_{\image(\tens{A}), \ker(\tens{A})}.
\end{equation}
% When \(\tens{A}\) is non-singular, \(\tens{A}^{\dagger}\) is the ordinary inverse \(\tens{A}^{-1}\).

Another type of generalized inverse was introduced by \citet{BOTT-1953}, which is a constrained inverse defined by the restriction of a square matrix \(\tens{A}\) on a subspace \(\spc{X}\) as
% and generalized Bott-Duffin inverse \citep{YONGLIN-1990, XUE-2002} 
    % which can be found \eg in \cite[p. 404, Note 2.2]{WEI-2001} or , and for the particular case of \( \hat{\spc{X}} = \hat{\spc{Y}} \), in \cite[Corollary 1]{XUE-2002}, \cite[Theorem 1]{CHEN-2003}, and \cite{DENG-2007}.
% \begin{subequations}
\begin{equation}
    \tens{A}^{(-1)}_{(\spc{X})} \coloneqq \tens{P}_{\spc{X}} ( \tens{P}_{\spc{X}^{\perp}} + \tens{A} \tens{P}_{\spc{X}} )^{-1}, \label{eq:def-bott}
    % \tens{A}^{(\dagger)}_{(\spc{X})} &\coloneqq \tens{P}_{\spc{X}} ( \tens{P}_{\spc{X}^{\perp}} + \tens{A} \tens{P}_{\spc{X}} )^{\dagger}, \label{eq:def-gen-bott}
\end{equation}
% \end{subequations}
provided that \( \tens{P}_{\spc{X}^{\perp}} + \tens{A} \tens{P}_{\spc{X}} \) is non-singular for \eqref{eq:def-bott} to exist (for the singular case, see generalization of \eqref{eq:def-bott} by \citet{YONGLIN-1990}). The Bott-Duffin inverse is also a form of outer inverse by
\begin{equation}
    \tens{A}^{(-1)}_{(\spc{X})} = \tens{A}^{(2)}_{\spc{X}, \spc{X}^{\perp}}.
    % \tens{A}^{(\dagger)}_{\spc{X}} = \tens{A}^{(2)}_{\spc{S}, \spc{S}^{\perp}},
    \label{eq:outer-bott}
\end{equation}
% where \(\spc{S} \coloneqq \image(\tens{P}_{\spc{X}} \tens{A})\).

In this paper, we often assume matrices are equal-principal (EP) (also known as range-Hermitian). We say a square matrix \(\tens{A}\) is EP if \(\image(\tens{A}) = \image(\tens{A}^{\ast})\), or equivalently, \(\image(\tens{A}) \perp \ker(\tens{A})\). %, hence, EP is also known as range perpendicular to null space (RPN) \citep[p. 408]{MEYER-2001}.
A necessary and sufficient condition for \(\tens{A}\) to be EP is that it commutes with its pseudo-inverse, \ie
\begin{equation*}
    \tens{A} \tens{A}^{\dagger} = \tens{A}^{\dagger} \tens{A}.
\end{equation*}
EP matrices are an extension of normal matrices; a normal matrix commutes with its Hermitian conjugate and is unitarily similar to a diagonal matrix, whereas an EP matrix is unitarily similar to a core-nilpotent matrix. Thus, a normal matrix is EP. Furthermore, the index of an EP matrix is \(1\), where the index is defined by \(\ind(\tens{A}) \coloneqq \min\{k ~|~ \mathrm{rank}\big(\tens{A}^{k+1}\big) = \mathrm{rank}\big(\tens{A}^{k}\big)\}\). A comprehensive list of properties of EP matrices can be found in \citep{TIAN-2011}.

% ============
% Main Results
% ============

\section{Main Results} \label{sec:main}

Throughout this work, \(\tens{M}\) refers to the matrix defined in \eqref{eq:M}, unless otherwise stated. Suppose \(\tens{A} \in \mathcal{M}_{n,m}(\mathbb{C})\), \(\tens{X} \in \mathcal{M}_{n,p}(\mathbb{C})\), and \(\tens{Y} \in \mathcal{M}_{m, q}(\mathbb{C})\). Let \(\spc{X} \coloneqq \image(\tens{X})\) and \(\spc{Y} \coloneqq \image(\tens{Y})\) denote the subspaces spanned by the column spaces of \(\tens{X}\) and \(\tens{Y}\), respectively. Define \(\hat{\spc{X}} \coloneqq \ker(\tens{M})\) and \(\hat{\spc{Y}} \coloneqq \coker(\tens{M})\), and let \(\hat{\tens{X}}\) and \(\hat{\tens{Y}}\) denote matrices so that \(\hat{\spc{X}} = \image(\hat{\tens{X}})\) and \(\hat{\spc{Y}} = \image(\hat{\tens{Y}})\). Note that \(\hat{\spc{X}}^{\perp} = \coimage(\tens{M})\), \(\hat{\spc{Y}}^{\perp} = \image(\tens{M})\), and \(\dim(\hat{\spc{X}}^{\perp}) = \dim(\hat{\spc{Y}}^{\perp})\).

In \Cref{sec:special}, we present a particular but common condition that allows expressing \(\hat{\spc{X}}\) or \(\hat{\spc{Y}}\) in terms of the known spaces \(\spc{X}\) and \(\spc{Y}\). We obtain the generalized inverse and pseudo-determinant identities for \(\tens{M}\) in \Cref{sec:identity} and \Cref{sec:det}, respectively. The proofs of \Cref{sec:main} can be found in \Cref{sec:pf}.

% ========================
% Kernel and Cokernel of M
% ========================

\subsection{Kernel and Cokernel of \texorpdfstring{\(\tens{M}\)}{M}} \label{sec:special}

Here we show that either \(\hat{\spc{X}}\) or \(\hat{\spc{Y}}\) can be readily obtained from \(\spc{X}\) or \(\spc{Y}\), respectively if
\begin{subequations}
\begin{align}
        &\tens{A}^{\dagger \ast} \spc{Y} + \spc{X}^{\perp} = \mathbb{C}^n, \label{eq:A-yx-1} \\
        &\spc{Y}^{\perp} + \tens{A}^{\dagger} \spc{X} = \mathbb{C}^m, \label{eq:A-yx-2}
\end{align}
\end{subequations}
where \(\tens{A}^{\dagger \ast}\) indicates \((\tens{A}^{\dagger})^{\ast}\). The above conditions can also be expressed by other forms as follows.

\begin{lemma} \label{lem:F}
    Define \(\tens{F} \coloneqq \tens{Y}^{\ast} \tens{A}^{\dagger} \tens{X}\). The following conditions are equivalent:
    \begin{center}
    \begin{enumerate*}[label=(\alph*),ref=(\alph*),itemjoin=\qquad]
        \item \label{cond:a} \(\tens{A}^{\dagger \ast} \spc{Y} + \spc{X}^{\perp} = \mathbb{C}^n,\)
        \item \label{cond:b} \((\tens{A}^{\dagger \ast} \spc{Y})^{\perp} \cap \spc{X} = \{\vect{0}\},\)
        \item \label{cond:c} \(\ker(\tens{F}) = \ker(\tens{X})\).
    \end{enumerate*}
    \end{center}
    Similarly, the following conditions are equivalent:
    \begin{center}
    \begin{enumerate*}[label=(\alph*),ref=(\alph*),resume,itemjoin=\qquad]
        \setcounter{enumi}{3}
        \item \label{cond:d} \(\spc{Y}^{\perp} + \tens{A}^{\dagger} \spc{X} = \mathbb{C}^m,\)
        \item \label{cond:e} \(\spc{Y} \cap (\tens{A}^{\dagger} \spc{X})^{\perp} = \{\vect{0}\},\)
        \item \label{cond:f} \(\image(\tens{F}) = \image(\tens{Y}^{\ast})\).
    \end{enumerate*}
    \end{center}
    % \begin{enumerate}[label=(\alph*),ref=(\alph*)]
    %     \item \label{cond:a} \(\tens{A}^{\dagger \ast} \spc{Y} + \spc{X}^{\perp} = \mathbb{C}^n\).
    %     \item \label{cond:b} \((\tens{A}^{\dagger \ast} \spc{Y})^{\perp} \cap \spc{X} = \{\vect{0}\}\).
    %     \item \label{cond:c} \(\ker(\tens{F}) = \ker(\tens{X})\).
    % \end{enumerate}
    % Similarly, the following conditions are equivalent:
    % \begin{enumerate}[label=(\alph*),ref=(\alph*),resume]
    %     \item \label{cond:d} \(\spc{Y}^{\perp} + \tens{A}^{\dagger} \spc{X} = \mathbb{C}^m\).
    %     \item \label{cond:e} \(\spc{Y} \cap (\tens{A}^{\dagger} \spc{X})^{\perp} = \{\vect{0}\}\).
    %     \item \label{cond:f} \(\image(\tens{F}) = \image(\tens{Y}^{\ast})\).
    % \end{enumerate}
    Furthermore, \ref{cond:a} to \ref{cond:f} are equivalent if and only if \(\dim(\spc{X}) = \dim(\spc{Y})\).
\end{lemma}

% \begin{remark}
%     It is intriguing to speculate whether the condition \ref{cond:a} might also be equivalent to \(\tens{A} \spc{Y} \oplus \spc{X}^{\perp} = \mathbb{C}^{n}\), but a counter example is \(\tens{X} = \tens{Y} = [2, 1]^{\ast}\) and \(\tens{A} = \operatorname{diag}(\frac{1}{4}, -1)\).
% \end{remark}

In practice, \eqref{eq:A-yx-1} and \eqref{eq:A-yx-2} can be verified by evaluating the conditions \ref{cond:c} and \ref{cond:f} in the above, respectively. A consequence of the above conditions is given below.

\begin{proposition} \label{prop:M-proj}
    The conditions \eqref{eq:A-yx-1} and \eqref{eq:A-yx-2} respectively imply
    \begin{subequations}
    \begin{align}
        \tens{M} &= \tens{A}^{\dagger} \tens{P}_{(\tens{A}^{\dagger \ast} \spc{Y})^{\perp}, \spc{X}}, \label{eq:M-P1} \\
        \tens{M} &= \tens{P}_{\spc{Y}^{\perp}, \tens{A}^{\dagger} \spc{X}} \tens{A}^{\dagger}. \label{eq:M-P2}
    \end{align}
    \end{subequations}
    % If \eqref{eq:A-yx-1} holds, then
    % \begin{subequations}
    % \begin{align}
    %     \tens{M} &= \tens{A}^{\dagger} \tens{P}_{(\tens{A}^{\dagger \ast} \spc{Y})^{\perp}, \spc{X}}. \label{eq:M-P1}\\
    %     \intertext{Also, if \eqref{eq:A-yx-2} holds, then}
    %     \tens{M} &= \tens{P}_{\spc{Y}^{\perp}, \tens{A}^{\dagger} \spc{X}} \tens{A}^{\dagger}. \label{eq:M-P2}
    % \end{align}
    % \end{subequations}
\end{proposition}

The matrix \(\tens{M}\) given by the forms of \eqref{eq:M-P1} and \eqref{eq:M-P2} allows us to express its kernel and cokernel in terms of \(\spc{X}\), \(\spc{Y}\), and the kernel and cokernel of \(\tens{A}\), as follows.

\begin{theorem} \label{thm:M-ker-coker}
    The conditions \eqref{eq:A-yx-1} and \eqref{eq:A-yx-2} respectively imply
    \begin{subequations}
    \begin{align}
        &\hat{\spc{X}} = \spc{X} \oplus ((\tens{A}^{\dagger \ast} \spc{Y})^{\perp} \cap \coker(\tens{A})), \label{eq:hat-X} \\
        &\hat{\spc{Y}} = \spc{Y} \oplus ((\tens{A}^{\dagger}\spc{X})^{\perp} \cap \ker(\tens{A})). \label{eq:hat-Y}
    \end{align}
    \end{subequations}
    % If \eqref{eq:A-yx-1} holds, then
    % \begin{subequations}
    % \begin{align}
    %     &\hat{\spc{X}} = \spc{X} \oplus ((\tens{A}^{\dagger \ast} \spc{Y})^{\perp} \cap \coker(\tens{A})). \label{eq:hat-X} \\
    %     \intertext{Also, if \eqref{eq:A-yx-2} holds, then}
    %     &\hat{\spc{Y}} = \spc{Y} \oplus ((\tens{A}^{\dagger}\spc{X})^{\perp} \cap \ker(\tens{A})). \label{eq:hat-Y}
    % \end{align}
    % \end{subequations}
\end{theorem}

\begin{corollary} \label{cor:XeqY}
    If \(\tens{A} \in \mathcal{M}_{n, n}(\mathbb{C})\) is non-singular and either of \eqref{eq:A-yx-1} or \eqref{eq:A-yx-2} holds, then \(\hat{\spc{X}} = \spc{X}\) and \(\hat{\spc{Y}} = \spc{Y}\), and we may set \(\hat{\tens{X}} = \tens{X}\) and \(\hat{\tens{Y}} = \tens{Y}\).
\end{corollary}

\Cref{cor:XeqY} justifies our notation for the kernel and cokernel of \(\tens{M}\) as \(\hat{\spc{X}}\) and \(\hat{\spc{Y}}\). Namely, whenever the hypothesis of \Cref{cor:XeqY} is satisfied, we can obtain the corresponding relations by omitting the hat symbol on matrices and subspaces from the subsequent expressions. We note our subsequent development does not necessitate the above hypotheses, however, such conditions are common. A practical application (see \Cref{sec:app}) that fulfills \eqref{eq:A-yx-1} and \eqref{eq:A-yx-2} is given below.

\begin{proposition} \label{prop:A-XPD}
    Both \eqref{eq:A-yx-1} and \eqref{eq:A-yx-2} are satisfied if \(\spc{X} = \spc{Y}\) and \(\tens{A}^{\dagger}\) is Hermitian and \(\spc{X}\)-PD.
\end{proposition}

In general, we can state the followings about the kernel and cokernel of \(\tens{M}\).

\begin{proposition} \label{prop:M-singular}
    \(\tens{M} = \tens{A}^{\dagger}\) if and only if \(\tens{A}^{\dagger \ast} \spc{Y} \perp \spc{X}\), or equivalently \(\spc{Y} \perp \tens{A}^{\dagger} \spc{X}\).
\end{proposition}

\begin{proposition} \label{prop:XY-share}
    If \(\tens{A}^{\dagger \ast} \spc{Y} \not\perp \spc{X}\), then \(\hat{\spc{X}} \cap \spc{X} \neq \{\vect{0}\}\) and \(\hat{\spc{Y}} \cap \spc{Y} \neq \{\vect{0}\}\).
\end{proposition}

The above statements imply that \(\tens{M}\) is always rank-deficient, \ie \(\hat{\spc{X}}, \hat{\spc{Y}} \neq \{\vect{0}\}\), except if \(\tens{A}\) is full-rank and \(\tens{A}^{\dagger \ast} \spc{Y} \perp \spc{X}\) (including when \(\tens{X}\) or \(\tens{Y}\) is null).

% ===========================
% A Matrix Inversion Identity
% ===========================

\subsection{Generalized Inverse Identities} \label{sec:identity}

In the followings, we observe \(\tens{M}\) is an outer inverse of \(\tens{A}\) and obtain relations that can be readily implied by the properties of the outer inverse, such as the pseudo-inverse of \(\tens{M}\).

\begin{proposition} \label{prop:MA}
    The matrix \(\tens{M} \in \mathcal{M}_{m, n}(\mathbb{C})\) in \eqref{eq:M} is the outer inverse of \(\tens{A} \in \mathcal{M}_{n, m}(\mathbb{C})\) by
    \begin{equation}
        \tens{M} = \tens{A}^{(2)}_{\hat{\spc{Y}}^{\perp}, \hat{\spc{X}}}. \label{eq:MA2}
    \end{equation}
    Furthermore, \(\dim(\hat{\spc{Y}}) + \rank(\tens{A}) \geq m\).
\end{proposition}

\begin{remark} \label{rem:inv12}
    Since \(\tens{M}\) is a \(\{2\}\)-inverse of \(\tens{A}\), we can mutually infer from \eqref{eq:penrose} that \(\tens{A}\) is a \(\{1\}\)-inverse of \(\tens{M}\). Furthermore, if \(\dim(\hat{\spc{Y}}) + \rank(\tens{A}) = m\), then \(\tens{A}\) and \(\tens{M}\) are the \(\{1, 2\}\)-inverse of each other \citep[p. 73, Corollary 10]{ISRAEL-2003}.
\end{remark}

\begin{remark} \label{rem:A-ypx}
    Recall that \eqref{eq:unique-XY} is the necessary and sufficient condition to uniquely determine a \(\{2\}\)-inverse of a matrix with a prescribed kernel and image. Such condition for \eqref{eq:MA2} becomes
        \begin{equation}
            \tens{A} \hat{\spc{Y}}^{\perp} \oplus \hat{\spc{X}} = \mathbb{C}^n, \label{eq:A-ypx}
        \end{equation}
        or equivalently, \(\hat{\spc{Y}} \oplus \tens{A}^{\ast} \hat{\spc{X}}^{\perp} = \mathbb{C}^m\) \citep[Equation 1.12.b]{CHEN-2000}. Note \(\hat{\spc{X}}\) and \(\hat{\spc{Y}}^{\perp}\) here are not prescribed, rather, defined by \(\tens{M}\), implying they should already satisfy \eqref{eq:A-ypx}.
\end{remark}

\begin{corollary} \label{cor:pinv}
    The Moore-Penrose inverse of \(\tens{M}\) is
    \begin{equation}
        \tens{M}^{\dagger} = \tens{P}_{\hat{\spc{X}}^{\perp}} \tens{A} \tens{P}_{\hat{\spc{Y}}^{\perp}}, \label{eq:M-pinv}
    \end{equation}
    where the orthogonal projection matrices \(\tens{P}_{\hat{\spc{X}}^{\perp}}\) and \(\tens{P}_{\hat{\spc{Y}}^{\perp}}\) are respectively given by
    \begin{subequations}
    \begin{align}
        \tens{P}_{\hat{\spc{X}}^{\perp}} &\coloneqq \tens{I} - \hat{\tens{X}} \left(\hat{\tens{X}}^{\ast} \hat{\tens{X}} \right)^{\dagger} \hat{\tens{X}}^{\ast}, \label{eq:Pxp} \\
        \tens{P}_{\hat{\spc{Y}}^{\perp}} &\coloneqq \tens{I} - \hat{\tens{Y}} \left(\hat{\tens{Y}}^{\ast} \hat{\tens{Y}} \right)^{\dagger} \hat{\tens{Y}}^{\ast}. \label{eq:Pyp}
    \end{align}
    \end{subequations}
\end{corollary}

\Cref{cor:pinv} is closely related to the well-known formula of \citet[Theorem 3]{FILL-1999} for the pseudo-inverse of the sum of two singular matrices, \eg \(\tens{N}\) in \eqref{eq:N-with-B}, when either of \(\tens{A}\) or \(\tens{B}\) therein is singular. However, \Cref{cor:pinv} takes the opposite viewpoint and expresses the pseudo-inverse of \(\tens{M}\) instead, when \(\tens{N}\) in \eqref{eq:N-with-B} is undefined.

The matrix \(\tens{M}\), as a \(\{2\}\)-inverse, can be also expressed in other forms, such as by a full-rank representation using the bases of its image and coimage (see \eg \citep{SHENG-2007}, \citep{STANIMIROVIC-2012} and \citep[Section 5.1]{WANG-2018}). However, when \(\tens{A}\) is square and \(\ind(\tens{M}) = 1\), we can express \(\tens{M}\) directly by its kernel and cokernel, \ie matrices \(\hat{\tens{X}}\) and \(\hat{\tens{Y}}\), as follows.

\begin{lemma} \label{prop:ind-1}
    Suppose \(\tens{A} \in \mathcal{M}_{n, n}(\mathbb{C})\). It holds \(\hat{\spc{Y}}^{\perp} \oplus \hat{\spc{X}} = \mathbb{C}^n\) if and only if \(\ind(\tens{M}) = 1\).
\end{lemma}

\begin{remark} \label{rem:ind-1}
    A practical example for \(\ind(\tens{M}) = 1\) is when \(\tens{X} = \tens{Y}\) and \(\tens{A}\) is Hermitian, as \(\tens{M}\) also becomes Hermitian, which is known to have index one \citep[p. 159]{ISRAEL-2003}.
\end{remark}

If \(\hat{\spc{Y}}^{\perp}\) and \(\hat{\spc{X}}\) are complementary by \Cref{prop:ind-1}, the projection matrix \(\tens{P}_{\hat{\spc{Y}}^{\perp}, \hat{\spc{X}}}\) can be defined, which is required for the following statement.

\begin{theorem} \label{thm:M-rep}
    Suppose \(\tens{A} \in \mathcal{M}_{n,n}(\mathbb{C})\) and \(\ind(\tens{M}) = 1\). Define
    \begin{equation}
        \tens{N} \coloneqq \tens{P}_{\hat{\spc{X}}, \hat{\spc{Y}}^{\perp}} + \tens{A} \tens{P}_{\hat{\spc{Y}}^{\perp}, \hat{\spc{X}}}. \label{eq:def-N}
    \end{equation}
    Then, \(\tens{N}\) is non-singular and
    \begin{equation}
        \tens{M} = \tens{P}_{\hat{\spc{Y}}^{\perp}, \hat{\spc{X}}} \tens{N}^{-1}. \label{eq:M-rep}
    \end{equation}
    Also, \(\tens{P}_{\hat{\spc{Y}}^{\perp}, \hat{\spc{X}}}\) can be expressed by
    \begin{equation}
        \tens{P}_{\hat{\spc{Y}}^{\perp}, \hat{\spc{X}}} = \tens{I} - \hat{\tens{X}} \left( \hat{\tens{Y}}^{\ast} \hat{\tens{X}} \right)^{\dagger} \hat{\tens{Y}}^{\ast}. \label{eq:P_xy}
    \end{equation}
\end{theorem}

\begin{remark} \label{rem:bott-duffin}
    When \(\hat{\spc{X}} = \hat{\spc{Y}}\), we recognize from \eqref{eq:def-bott} and \eqref{eq:outer-bott} that the identity \eqref{eq:M-rep} is the expression for the Bott-Duffin inverse as (see also \cite{denG-2007})
    \begin{equation}
        \tens{M} = \tens{A}^{(2)}_{\hat{\spc{X}}^{\perp}, \hat{\spc{X}}} = \tens{A}^{(-1)}_{(\hat{\spc{X}}^{\perp})} = \tens{P}_{\hat{\spc{X}}^{\perp}} \left( \tens{P}_{\hat{\spc{X}}} + \tens{A} \tens{P}_{\hat{\spc{X}}^{\perp}} \right)^{-1}.
    \end{equation}
\end{remark}

Observe \(\tens{N} \tens{P}_{\hat{\spc{X}}} = \tens{P}_{\hat{\spc{X}}}\) and \(\tens{N} \tens{P}_{\hat{\spc{Y}}^{\perp}} = \tens{A} \tens{P}_{\hat{\spc{Y}}^{\perp}}\), so we can view \(\tens{N}\) by its restrictions on the complementary subspaces \(\hat{\spc{X}}\) and \(\hat{\spc{Y}}^{\perp}\) as
\begin{equation}
    \left. \tens{N} \right|_{\hat{\spc{X}}} = \left. \tens{I} \right|_{\hat{\spc{X}}},
    \quad \text{and} \quad
    \left. \tens{N} \right|_{\hat{\spc{Y}}^{\perp}} = \left. \tens{A} \right|_{\hat{\spc{Y}}^{\perp}}.
            \sublabel{eq:N-rest}{a,b}
\end{equation}
Furthermore, we can interpret the relation between \(\tens{N}\) and \(\tens{M}\) in \Cref{thm:M-rep} by using their \emph{compression} (in the sense of \citet[p. 120]{HALMOS-1982}) on \(\hat{\spc{X}}^{\perp}\), defined by \(\tens{P}_{\hat{\spc{X}}^{\perp}} \left. \tens{N} \right|_{\hat{\spc{X}}^\perp}\) and \(\tens{P}_{\hat{\spc{X}}^{\perp}} \left. \tens{M} \right|_{\hat{\spc{X}}^\perp}\), respectively. We show these two maps on \(\hat{\spc{X}}^{\perp} \to \hat{\spc{X}}^{\perp}\) are inverse of each other.

\begin{corollary} \label{cor:comp}
    Suppose \(\tens{N}\) and \(\tens{M}\) are as in \Cref{thm:M-rep}. Let \(\tens{U}_{\hat{\spc{X}}^{\perp}}\) be a matrix with orthonormal columns forming a basis on \(\hat{\spc{X}}^{\perp}\). Denote the compressions of \(\tens{N}\) and \(\tens{M}\) onto \(\hat{\spc{X}}^{\perp}\) that are represented by their coordinates on the basis \(\tens{U}_{\hat{\spc{X}}^{\perp}}\) respectively by
    \begin{equation}
        \tens{N}_{\hat{\spc{X}}^{\perp}} \coloneqq \tens{U}_{\hat{\spc{X}}^{\perp}}^{\ast} \tens{N} \tens{U}_{\hat{\spc{X}}^{\perp}},
        \quad and  \quad
        \tens{M}_{\hat{\spc{X}}^{\perp}} \coloneqq \tens{U}_{\hat{\spc{X}}^{\perp}}^{\ast} \tens{M} \tens{U}_{\hat{\spc{X}}^{\perp}}.
    \end{equation}
    Then, \(\tens{N}_{\hat{\spc{X}}^{\perp}}\) and \(\tens{M}_{\hat{\spc{X}}^{\perp}}\) are non-singular and
\begin{equation}
    \tens{M}_{\hat{\spc{X}}^{\perp}} = \left(\tens{N}_{\hat{\spc{X}}^{\perp}} \right)^{-1}. \label{eq:MNI}
\end{equation}
\end{corollary}

% Representing the inverse of \(\tens{M}\) on a subspace where it is non-singular, as in \eqref{eq:MNI}, is analogous to the Woodbury identity which finds the inverse of the non-singular form of the matrix \(\tens{M}\).
Representing the inverse of \(\tens{M}\) on a subspace where it is non-singular, as in \eqref{eq:MNI}, is reminiscent of the Woodbury identity which finds the inverse of the non-singular form of the matrix \(\tens{M}\).

% =============================
% A Pseudo-Determinant Identity
% =============================

\subsection{Pseudo-Determinant Identities} \label{sec:det}

Our objective in this section is to find the pseudo-determinant of \(\tens{M}\), denoted by \(\pdet(\tens{M})\) or \(\vert \tens{M} \vert_{\dagger}\). The pseudo-determinant of a square matrix is the product of its non-zero eigenvalues with the convention that the pseudo-determinant of a nilpotent matrix is \(1\). We refer the reader to \citep{KNILL-2014} for the properties and representations of pseudo-determinants and to \citep{ZHANG-2002, SHENG-2007} for the determinantal representation of generalized inverses.

In this section, we require \(\tens{A}\), and hence \(\tens{M}\), to be square matrices for their pseudo-determinant to be defined. It implies from the rank-nullity theorem that
    \begin{equation}
        \dim(\hat{\spc{X}}) = \dim(\hat{\spc{Y}}). \label{eq:dim-xy}
    \end{equation}
We also assume \(\tens{A}\) is EP, which grants us two properties. First, the core-nilpotent decomposition of \(\tens{A}\) becomes (see \citep[Equation 5.11.15]{MEYER-2001} or \citep[Proposition 3]{BAJO-2021})
\begin{equation}
    \tens{A} = \tens{U} \begin{bmatrix} \tilde{\tens{A}} & \tens{0} \\ \tens{0} & \tens{0} \end{bmatrix} \tens{U}^{\ast} = \tens{U}_{\spc{A}} \tilde{\tens{A}} \tens{U}_{\spc{A}}^{\ast}, \label{eq:A-tilde}
\end{equation}
where \(\tilde{\tens{A}} \in \mathcal{M}_{r, r}(\mathbb{C})\) is non-singular with \(r \coloneqq \rank(\tens{A})\), and \(\tens{U} \in \mathcal{M}_{n, n}(\mathbb{C})\) is unitary. Also, \(\spc{A} \coloneqq \image(\tens{A})\), and we set \(\tens{U} = [\tens{U}_{\spc{A}}, \tens{U}_{\spc{A}^{\perp}}]\), where \(\tens{U}_{\spc{A}}\) consists of the first \(r\) columns of \(\tens{U}\) and \(\tens{U}_{\spc{A}^{\perp}}\) consists of the rest of the columns of \(\tens{U}\). We have \(\image(\tens{U}_{\spc{A}}) = \spc{A}\) and \(\image(\tens{U}_{\spc{A}^{\perp}}) = \spc{A}^{\perp}\). %Also, the pseudo-inverse of an EP matrix coincides with its group inverse \cite[p. 77, Theorem 2.2.3]{WANG-2018}, namely
% \begin{equation}
%     \tens{A}^{\dagger} = \tens{A}^{\sharp} = \tens{U} \begin{bmatrix} \tilde{\tens{A}}^{-1} & \tens{0} \\ \tens{0} & \tens{0} \end{bmatrix} \tens{U}^{\ast} = \tens{U}_{\spc{A}} \tilde{\tens{A}}^{-1} \tens{U}_{\spc{A}}^{\ast}. \label{eq:Ainv-tilde}
% \end{equation}

The second consequence of the EP property of \(\tens{A}\) is related to \(\hat{\spc{Y}}\) as follows. Define
\begin{equation}
    \hat{\spc{X}}_{\spc{A}} = \hat{\spc{X}} \cap \spc{A}, \quad \text{and} \quad \hat{\spc{Y}}_{\spc{A}} = \hat{\spc{Y}} \cap \spc{A}. \label{eq:def-XA}
\end{equation}

\begin{lemma} \label{lem:XYAplus}
    It holds that \(\hat{\spc{X}}^{\perp} \subseteq \spc{A}\), and
    % \begin{subequations}
    % \begin{align}
    %     &\hat{\spc{X}} = \hat{\spc{X}}_{\spc{A}} \oplus \spc{A}^{\perp}, \label{eq:XAplus} \\
    %     &\spc{A} = \hat{\spc{X}}_{\spc{A}} \oplus \hat{\spc{X}}^{\perp}. \label{eq:AXplus}
    % \end{align}
    % \end{subequations}
    % Furthermore, if \(\tens{A}\) is an EP matrix, then \(\hat{\spc{Y}}^{\perp} \subseteq \spc{A}\), and
    % \begin{subequations}
    % \begin{align}
    %     &\hat{\spc{Y}} = \hat{\spc{Y}}_{\spc{A}} \oplus \spc{A}^{\perp}, \label{eq:YAplus} \\
    %     &\spc{A} = \hat{\spc{Y}}_{\spc{A}} \oplus \hat{\spc{Y}}^{\perp}. \label{eq:AYplus}
    % \end{align}
    % \end{subequations}
    \begin{equation}
        \hat{\spc{X}} = \hat{\spc{X}}_{\spc{A}} \oplus \spc{A}^{\perp},
        \quad \text{and} \quad
        \spc{A} = \hat{\spc{X}}_{\spc{A}} \oplus \hat{\spc{X}}^{\perp}. 
        \sublabel{eq:XAplus}{a,b}
    \end{equation}
    Furthermore, if \(\tens{A}\) is EP, then \(\hat{\spc{Y}}^{\perp} \subseteq \spc{A}\), and
    \begin{equation}
        \hat{\spc{Y}} = \hat{\spc{Y}}_{\spc{A}} \oplus \spc{A}^{\perp},
        \quad \text{and} \quad
        \spc{A} = \hat{\spc{Y}}_{\spc{A}} \oplus \hat{\spc{Y}}^{\perp}.
        \sublabel{eq:YAplus}{a,b}
    \end{equation}
The above also imply \(\spc{A}^{\perp} \subseteq \hat{\spc{X}} \cap \hat{\spc{Y}}\) and \(\dim(\hat{\spc{X}}_{\spc{A}}) = \dim(\hat{\spc{Y}}_{\spc{A}})\).
\end{lemma}

By the above properties and \Cref{cor:pinv}, we can obtain the pseudo-determinant of \(\tens{M}\) as follows.

\begin{proposition} \label{prop:pMp}
    Suppose \(\tens{A}\) is EP and \(\ind(\tens{M}) = 1\). Let \(\tens{U}_{\hat{\spc{X}}_{\spc{A}}}\), \(\tens{U}_{\hat{\spc{Y}}_{\spc{A}}}\), \(\tens{U}_{\hat{\spc{X}}^{\perp}}\), \(\tens{U}_{\hat{\spc{Y}}^{\perp}}\), and \(\tens{U}_{\spc{A}^{\perp}}\) be matrices with orthonormal columns as bases for \(\hat{\spc{X}}_{\spc{A}}\), \(\hat{\spc{Y}}_{\spc{A}}\), \(\hat{\spc{X}}^{\perp}\), \(\hat{\spc{Y}}^{\perp}\), and \(\spc{A}^{\perp}\), respectively. Then,
    \begin{equation}
        \vert \tens{M} \vert_{\dagger} = \frac{\vert \tens{U}_{\hat{\spc{X}}^{\perp}}^{\ast} \tens{U}_{\hat{\spc{Y}}^{\perp}} \vert}{ \vert \tens{U}_{\hat{\spc{X}}^{\perp}}^{\ast} \tens{A} \tens{U}_{\hat{\spc{Y}}^{\perp}} \vert} =
        \frac{
        \vert \tens{U}_{\hat{\spc{Y}}_{\spc{A}}}^{\ast} \tens{U}_{\hat{\spc{X}}_{\spc{A}}} \vert}
        {\vert \tens{A} \vert_{\dagger} \, \vert \tens{U}_{\hat{\spc{Y}}_{\spc{A}}}^{\ast} \tens{A}^{\dagger} \tens{U}_{\hat{\spc{X}}_{\spc{A}}} \vert}, \label{eq:pinv-M-p}
    \end{equation}
    and the matrices in the above determinants, \(\vert \cdot \vert\), are non-singular.
\end{proposition}

\begin{theorem} \label{thm:pdet}
    Suppose \(\tens{A} \in \mathcal{M}_{n, n}(\mathbb{C})\) is EP, \(\ind(\tens{M}) = 1\), \(\hat{\tens{X}}\) and \(\hat{\tens{Y}}\) have the same number of columns, and
    \begin{equation}
        \ker(\tens{P}_{\spc{A}} \hat{\tens{X}}) \cap \coimage(\tens{P}_{\spc{A}} \hat{\tens{Y}}) = \{\vect{0}\}, \label{eq:ker-PXY}
    \end{equation}
    where \(\tens{P}_{\spc{A}} = \tens{A} \tens{A}^{\dagger}\) is an orthogonal projection matrix onto \(\spc{A}\). Then,
    \begin{equation}
        \vert \tens{M} \vert_{\dagger} = \frac{\vert \hat{\tens{Y}}^{\ast} \tens{P}_{\spc{A}} \hat{\tens{X}} \vert_{\dagger}}{\vert \tens{A}\vert_{\dagger} \, \vert \hat{\tens{Y}}^{\ast} \tens{A}^{\dagger} \hat{\tens{X}} \vert_{\dagger}}. \label{eq:pdet}
    \end{equation}
\end{theorem}

\begin{remark} \label{rem:kerXY}
    If \(\hat{\tens{X}} = \hat{\tens{Y}}\), then \eqref{eq:ker-PXY} is trivially satisfied. In general, it is always possible to choose \(\hat{\tens{X}}\) and \(\hat{\tens{Y}}\) to satisfy \eqref{eq:ker-PXY}; take for instance \(\hat{\tens{X}} = [\tens{U}_{\hat{\spc{X}}_{\spc{A}}}, \tens{U}_{\spc{A}^{\perp}}]\) and \(\hat{\tens{Y}} = [\tens{U}_{\hat{\spc{Y}}_{\spc{A}}}, \tens{U}_{\spc{A}^{\perp}}]\).
\end{remark}

We can also relate the (pseudo-)determinant of \(\tens{M}\) and \(\tens{N}\) based on \Cref{cor:comp}.

\begin{proposition} \label{prop:pdet-M-N}
    Suppose \(\tens{A} \in \mathcal{M}_{n,n}(\mathbb{C})\) and \(\ind(\tens{M}) = 1\). Then,
    \begin{equation}
        \vert \tens{M} \vert_{\dagger} = \vert \tens{M}_{\hat{\spc{X}}^{\perp}} \vert = \vert \tens{N}_{\hat{\spc{X}}^{\perp}} \vert^{-1} = \vert \tens{N} \vert^{-1}. \label{eq:det-N}
    \end{equation}
\end{proposition}

By combining \Cref{thm:pdet} and \Cref{prop:pdet-M-N}, we can eliminate \(\vert \tens{M} \vert_{\dagger}\) and obtain a generic determinant identity for \(\tens{A}\), \(\tens{N}\), \(\hat{\tens{X}}\), and \(\hat{\tens{Y}}\) as given next. To completely remove \(\tens{M}\) from such identity, we replace \(\hat{\tens{X}}\) and \(\hat{\tens{Y}}\) (which are defined by \(\tens{M}\)) with arbitrary matrices \(\tens{X}\) and \(\tens{Y}\) as long as they satisfy the conditions that \(\hat{\tens{X}}\) and \(\hat{\tens{Y}}\) hold. Also, we set \(\tens{P} \coloneqq \tens{P}_{\hat{\spc{Y}}^{\perp}, \hat{\spc{X}}}\) to simplify notations.

\begin{corollary} \label{cor:det-N}
    Suppose \(\tens{A} \in \mathcal{M}_{n, n}(\mathbb{C})\), \(\tens{X}, \tens{Y} \in \mathcal{M}_{n, p}(\mathbb{C})\), and recall \(\spc{X} = \image(\tens{X})\), \(\spc{Y} = \image(\tens{Y})\), and
    \begin{align}
        &\tens{P} \coloneqq \tens{I} - \tens{X} (\tens{Y}^{\ast} \tens{X})^{\dagger} \tens{Y}^{\ast}, \label{eq:def-P2} \\
        &\tens{N} \coloneqq \tens{I} - \tens{P} + \tens{A} \tens{P}. \label{eq:def-N2}
    \end{align}
    If \(\tens{A}\) is EP and
    \begin{subequations} \label{eq:detN-cond}
    \begin{align}
        &\spc{X} \oplus \tens{A}\spc{Y}^{\perp} = \mathbb{C}^n, \label{eq:detN-cond1} \\
        &\spc{X} \oplus \spc{Y}^{\perp} = \mathbb{C}^n, \label{eq:detN-cond2} \\
        &\spc{A}^{\perp} \subseteq \spc{X} \cap \spc{Y}, \label{eq:detN-cond3} \\
        &\ker(\tens{P}_{\spc{A}} \tens{X}) \cap \coimage(\tens{P}_{\spc{A}} \tens{Y}) = \{\vect{0}\}, \label{eq:detN-cond4}
    \end{align}
    \end{subequations}
    Then, \(\tens{N}\) is non-singular and
    \begin{equation}
        \vert \tens{A} \vert_{\dagger} \, \vert \tens{Y}^{\ast} \tens{A}^{\dagger} \tens{X} \vert_{\dagger} = \vert \tens{Y}^{\ast} \tens{A} \tens{A}^{\dagger} \tens{X} \vert_{\dagger} \, \vert \tens{N} \vert. \label{eq:det-MN}
    \end{equation}
\end{corollary}

Note that the conditions \eqref{eq:detN-cond1}, \eqref{eq:detN-cond2}, \eqref{eq:detN-cond3}, and \eqref{eq:detN-cond4} imposed on \(\tens{X}\) and \(\tens{Y}\) are respectively the restatements of \eqref{eq:A-ypx}, \Cref{prop:ind-1}, \Cref{lem:XYAplus}, and \eqref{eq:ker-PXY} that \(\hat{\tens{X}}\) and \(\hat{\tens{Y}}\) satisfy. The conditions of \Cref{cor:det-N} can be met, for instance, by the hypotheses of \Cref{cor:XeqY} and \Cref{prop:A-XPD}. Namely, a Hermitian \(\tens{A}\) is EP; if \(\tens{X} = \tens{Y}\), then \eqref{eq:detN-cond2} and \eqref{eq:detN-cond4} are satisfied; if \(\spc{X} = \spc{Y}\) and \(\tens{A}^{\dagger}\) is \(\spc{X}\)-PD, we can show \eqref{eq:detN-cond1} holds; and, if \(\tens{A}\) is non-singular, then \eqref{eq:detN-cond3} is satisfied. We meet these conditions in \Cref{sec:app}. Also, in \Cref{sec:num}, we compare the numerical complexity of computing \eqref{eq:pinv-M-p} and both sides of \eqref{eq:det-MN}.

% ===========================================
% Applications to Gaussian Process Regression
% ===========================================

\section{Applications to Gaussian Process Regression} \label{sec:app}

We present two applications of the presented identities to the Gaussian process regression, a Bayesian, non-parametric method in supervised learning that provides probabilistic predictions. A brief overview of Gaussian process regression formulations is provided in \Cref{sec:gpr}. \Cref{sec:app-like} presents the utilization of the derived identities to express the likelihood function of a form of Gaussian process as a normal distribution. \Cref{sec:app-asym} derives a power series that is beneficial in the asymptotic analysis of Gaussian process regression.

% ============================================
% A Representation for the Likelihood Function
% ============================================

\subsection{A Representation for the Likelihood Function} \label{sec:app-like}

Let the array of data \(\vect{y} \in \mathbb{R}^n\) be a realization of the stochastic function \(f: \mathcal{D} \to \mathbb{R}\) at \(n\) points with the mean \(\vect{\mu} \in \mathbb{R}^n\) and covariance \(\gtens{\Sigma} \in \mathcal{M}_{n, n}(\mathbb{R})\), which is a symmetric positive-definite (SPD) matrix. We assume \(\gtens{\Sigma}\) depends on the hyperparameters \(\vect{\theta}\). The mean \(\vect{\mu}\) is often modeled by the linear combination of \(p\) basis functions given by the columns of the full rank design matrix \(\tens{X} \in \mathcal{M}_{n, p}(\mathbb{R})\). More details on these variables can be found in \Cref{sec:gpr}. A Gaussian process prior on \(f\) with the above mean and covariance leads to the (marginal) likelihood of the data \(\vect{y}\), given \(\vect{\theta}\), as (see \citep[Equation 2.45]{RASMUSSEN-2006})
\begin{equation}
    p(\vect{y}| \vect{\theta}) = \frac{1}{\sqrt{(2\pi)^{n-m}}} | \gtens{\Sigma} |^{-\sfrac{1}{2}} |\tens{X}^{\ast} \gtens{\Sigma}^{-1} \tens{X} |^{-\sfrac{1}{2}} \exp\left(-\frac{1}{2} \| \vect{y} \|_{\tens{M}}^2 \right), \label{eq:like}
\end{equation}
where, \(\|\vect{y} \|_{\tens{M}}^2 \coloneqq \vect{y}^{\ast} \tens{M} \vect{y}\) and
\begin{equation}
    \tens{M} \coloneqq \gtens{\Sigma}^{-1} - \gtens{\Sigma}^{-1} \tens{X} \left(\tens{X}^{\ast} \gtens{\Sigma}^{-1} \tens{X} \right)^{-1} \tens{X}^{\ast} \gtens{\Sigma}^{-1}, \label{eq:M-like}
\end{equation}
which is the symmetric form of \eqref{eq:M} by setting \(\tens{Y} = \tens{X}\) therein and using \(\gtens{\Sigma}\) instead of \(\tens{A}\). We realize from \Cref{cor:XeqY} and \Cref{prop:A-XPD} that \(\ker(\tens{M}) = \coker(\tens{M}) = \spc{X}\). Note that \eqref{eq:like} is not a density function on \(\mathbb{R}^n\) as it does not have a finite measure since \(\left. p \right|_{\spc{X}}\) is constant. We herein only consider \(\vect{y} \in \spc{X}^{\perp}\).

Equation \eqref{eq:like} represents the limit case of the more general form of the likelihood function of the Gaussian process given in \eqref{eq:like-B}, where the matrix \(\tens{B}^{-1}\) therein vanishes. However, the transition from \eqref{eq:like-B} to \eqref{eq:like} is not straightforward as \(\tens{B}^{-1} \to \tens{0}\). Specifically, \eqref{eq:like} cannot be derived from \eqref{eq:like-B} by the naive substitution of \(\tens{B}^{-1} = \tens{0}\); but rather is derived through projecting \(\vect{y}\) on \(\spc{X}^{\perp}\) (as detailed in \citet[Section 2.7.1]{RASMUSSEN-2006} and references therein).

In contrast to \eqref{eq:like-B}, which can be expressed as the normal distribution \eqref{eq:like-B-2} through the application of the Woodbury identity and determinant lemma, the likelihood function \eqref{eq:like} is not immediately recognized as a normal distribution. Nevertheless, by leveraging the presented identities, we can demonstrate that \eqref{eq:like} is indeed a normal distribution on \(\spc{X}^{\perp}\). This is achieved through a combination of \Cref{rem:ind-1}, \Cref{thm:pdet}, and \Cref{rem:kerXY}. Namely, we have
\begin{equation*}
    | \gtens{\Sigma} | \, | \tens{X}^{\ast} \gtens{\Sigma} ^{-1}\tens{X}| = | \tens{X}^{\ast} \tens{X} | \, | \tens{M} |_{\dagger}^{-1}.
\end{equation*}
Thus, \eqref{eq:like} can be represented by
\begin{equation}
    p(\vect{y}| \vect{\theta}) \propto \frac{1}{\sqrt{(2\pi)^{n-m}}} | \tens{M} |_{\dagger}^{\sfrac{1}{2}} \exp\left(-\frac{1}{2} \| \vect{y} \|_{\tens{M}}^2 \right), \label{eq:like-2}
\end{equation}
where the constant of proportionality is \(|\tens{X}^{\ast} \tens{X}|^{\sfrac{-1}{2}}\).
%that is omitted for simplicity\footnote{The expression \(|\tens{X}^{\ast} \tens{X}|^{\sfrac{1}{2}}\) is the volume of parallelotope spanned by the columns of \(\tens{X}\), adjusting \eqref{eq:like-2} for the non-orthonormality of the basis for \(\mathcal{X}\).}.
We represent the restriction of \eqref{eq:like-2} on \(\spc{X}^{\perp}\) as follows. Define the compression of \(\gtens{\Sigma}\) and \(\tens{M}\) on \(\spc{X}^{\perp}\) by \(\gtens{\Sigma}_{\spc{X}^{\perp}} \coloneqq \tens{U}_{\spc{X}^{\perp}}^{\ast} \gtens{\Sigma} \tens{U}_{\spc{X}^{\perp}}\) and \(\tens{M}_{\spc{X}^{\perp}} \coloneqq \tens{U}_{\spc{X}^{\perp}}^{\ast} \tens{M} \tens{U}_{\spc{X}^{\perp}}\), respectively, where \(\tens{U}_{\spc{X}^{\perp}}\) is defined in \Cref{cor:comp}. Also, define \(\tens{N}\) using \(\gtens{\Sigma}\) and \(\spc{X}\) (instead of \(\tens{A}\) and \(\spc{Y}\)) in \eqref{eq:def-N}. Recall from \Cref{cor:comp} that \(\tens{N}_{\spc{X}^{\perp}} = \tens{U}_{\spc{X}^{\perp}}^{\ast} \tens{N} \tens{U}_{\spc{X}^{\perp}}\). From \subeqref{eq:N-rest}{b}, we can show \(\gtens{\Sigma}_{\spc{X}^{\perp}} = \tens{N}_{\spc{X}^{\perp}}\), so \eqref{eq:MNI} yields \(\tens{M}_{\spc{X}^{\perp}} =(\gtens{\Sigma}_{\spc{X}^{\perp}})^{-1}\), and from \Cref{prop:pdet-M-N}, we obtain
\begin{equation}
    \vert \tens{M} \vert_{\dagger} = \vert \gtens{\Sigma}_{\spc{X}^{\perp}} \vert^{-1}. \label{eq:M-Sigma}
\end{equation}
Let \(\vect{y}_{\spc{X}^{\perp}} \coloneqq \tens{U}_{\spc{X}^{\perp}}^{\ast} \vect{y}\), which is the projection of \(\vect{y}\) onto \(\spc{X}^{\perp}\) represented by its coordinates on the basis of the columns of \(\tens{U}_{\spc{X}^{\perp}}\). Using \(\tens{P}_{\spc{X}^{\perp}} = \tens{I} - \tens{P}_{\spc{X}}\), we can show \(\tens{M} = \tens{P}_{\spc{X}^{\perp}} \tens{M} \tens{P}_{\spc{X}^{\perp}}\), and by setting \(\tens{P}_{\spc{X}^{\perp}} = \tens{U}_{\spc{X}^{\perp}} \tens{U}_{\spc{X}^{\perp}}^{\ast}\) we can obtain
\begin{equation}
    \| \vect{y} \|_{\tens{M}} = \| \vect{y}_{\spc{X}^{\perp}} \|_{\big( \gtens{\Sigma}_{\spc{X}^{\perp}}\big)^{-1}}. \label{eq:M-y}
\end{equation}
Based on \eqref{eq:M-Sigma} and \eqref{eq:M-y}, we realize \eqref{eq:like-2} is the \((n-m)\)-dimensional normal distribution on \(\spc{X}^{\perp}\), \ie
\begin{equation*}
    \vect{y}_{\spc{X}^{\perp}} | \vect{\theta} \sim \mathcal{N} \left( \tens{0}, \gtens{\Sigma}_{\spc{X}^{\perp}}\right).
\end{equation*}
% We also note that the constant of proportionality in \eqref{eq:like-2}, \ie \(|\tens{X}^{\ast} \tens{X}|^{\sfrac{1}{2}}\), is the volume element that scales the probability density function due to the transformation of coordinates on the basis \(\tens{U}_{\spc{X}^{\perp}}\). 
Moreover, \(\tens{M}_{\spc{X}^{\perp}}\) is the precision matrix of the above distribution. Motivated by this, we may regard \(\tens{M}\) as the precision matrix for \eqref{eq:like-2} if we extend the definition of the precision matrix as follows.

\begin{definition} \label{def:precision}
    We define the precision matrix as the Bott-Duffin inverse of the covariance matrix, since by \Cref{rem:bott-duffin}, we have
\begin{equation}
    \tens{M} = \gtens{\Sigma}^{(-1)}_{(\spc{X}^{\perp})}. \label{eq:precision}
\end{equation}
Also, we denote the likelihood function \eqref{eq:like-2} (\ie a normal distribution on \(\mathcal{X}^{\perp}\), but constant along \(\mathcal{X}\)) by
\begin{equation}
    \vect{y} \vert \vect{\theta} \sim \mathcal{N}_{\mathcal{X}^{\perp}}^{-1}(\vect{0}, \tens{M}),
\end{equation}
where \(\mathcal{N}^{-1}\) represents the canonical (or information) parametrization of the normal distribution\footnote{The canonical parametrization of the normal distribution \(\mathcal{N}(\vect{\mu}, \gtens{\Sigma})\) is denoted by \(\mathcal{N}^{-1}(\vect{\eta}, \tens{M})\) where \(\tens{M} \coloneqq \gtens{\Sigma}^{-1}\) is the precision matrix and \(\vect{\eta} \coloneqq \tens{M} \vect{\mu}\) is the potential vector (see \eg \citep[Definition 2.2]{RUE-2005}).} with the potential vector \(\vect{0}\) and the precision matrix \(\tens{M}\). Note that if \(\spc{X}^{\perp} = \mathbb{R}^n\), then \eqref{eq:precision} reduces to the conventional definition of the precision matrix of the normal distribution in \(\mathbb{R}^n\).
\end{definition}

The extended definition of precision matrix preserves the key characteristics of the conventional definition. For instance, it is well established that the zero entries in the precision matrix indicate the conditional independence between variables \citep[p. 21, Threorem 2.2]{RUE-2005}. Specifically, when \(M_{ij} = 0\), it implies that \(y_i\) and \(y_j\) are independent given all other \(y_k\), \(k \neq i, j\). This is particularly useful in Gaussian graphical models. We show that this property also holds for \Cref{def:precision}.

\begin{proposition}[Conditional Independence] \label{prop:cond-ind}
    Suppose \(\tens{M}\) is the precision matrix for \eqref{eq:like-2} in the sense of \Cref{def:precision}. If \(M_{ij} = 0\), then \(y_i \indep y_j \mid y_{\setminus \{i, j\}}\), meaning that \(y_i\) and \(y_j\) are conditionally independent given \(\{y_k \mid k = 1, \dots, n; \; k \neq i, j\}\). 
\end{proposition}

\begin{proof}
    If \(M_{ij} = 0\), we can write \(\Vert \vect{y} \Vert_{\tens{M}}^2 = y_i^2 M_{ii} + a y_i + y_{j}^2 M_{jj} + b y_j + c\) where \(a\), \(b\), and \(c\) are constants involving \(y_{\setminus \{i, j\}}\). Thus, from \eqref{eq:like-2} we have \(p(y_i, y_j \mid y_{\setminus \{i, j\}}, \vect{\theta}) = p(y_i \mid y_{\setminus \{i, j\}}, \vect{\theta}) \, p(y_j \mid y_{\setminus \{i, j\}}, \vect{\theta})\), implying that \(y_i\) and \(y_j\) are independent, given \(y_{\setminus \{i, j\}}\).
\end{proof}

The precision matrix also has the desirable property of allowing for the easy calculation of the conditional precision of a normal distribution. This is achieved by using the diagonal blocks of the precision matrix \citep[p. 26, Theorem 2.5]{RUE-2005}. Namely, if we partition the data by \(\vect{y} = [\vect{y}_1, \vect{y}_2]^{\ast}\) with the corresponding division of the precision matrix as \(\tens{M} = \begin{bmatrix} \tens{M}_{11} & \tens{M}_{12} \\ \tens{M}_{12}^{\ast} & \tens{M}_{22} \end{bmatrix}\), then the resulting conditional precision for \(\vect{y}_1 \vert \vect{y}_2\) is simply \(\tens{M}_{1 \vert 2} = \tens{M}_{11}\). We generalize this property as follows. First, note that the partitions \(\vect{y}_1\) and \(\vect{y}_2\) can be regarded as the components of orthogonal projections of \(\vect{y}\) on two sets of coordinate axes. We generalize this in \Cref{prop:cond-marg} by allowing \(\vect{y}\) to be projected on two complementary subspaces. As a special case of \Cref{prop:cond-marg}, in \Cref{cor:cond-marg}, we present both the conditional and marginal precisions of the partitioned data \(\vect{y} = [\vect{y}_1, \vect{y}_2]^{\ast}\) which applies to \Cref{def:precision} of the precision matrix.

\begin{proposition} \label{prop:cond-marg}
    Suppose \(\tens{M}\) is the precision matrix for \eqref{eq:like-2} in the sense of \Cref{def:precision}. Let \(\mathcal{Y}_1, \mathcal{Y}_2 \in \mathbb{R}^n\) be complementary subspace where \(\mathcal{Y}_1 \cap \mathcal{X} = \varnothing\) and we recall \(\mathcal{X} = \ker(\tens{M})\). Define
    \begin{equation}
        \tens{M}_{\mathcal{Y}_1} \coloneqq \tens{P}_{\mathcal{Y}_1} \tens{M} \tens{P}_{\mathcal{Y}_1},
        \quad
        \tens{M}_{\mathcal{Y}_1,\mathcal{Y}_2} \coloneqq \tens{P}_{\mathcal{Y}_1} \tens{M} \tens{P}_{\mathcal{Y}_2},
        \quad
        \tens{M}_{\mathcal{Y}_2} \coloneqq \tens{P}_{\mathcal{Y}_2} \tens{M} \tens{P}_{\mathcal{Y}_2}.
        \label{eq:M12}
    \end{equation}
    Also, define the projections \(\vect{y}_{\mathcal{Y}_1} \coloneqq \tens{P}_{\mathcal{Y}_1, \mathcal{Y}_2} \vect{y}\) and \(\vect{y}_{\mathcal{Y}_2} \coloneqq \tens{P}_{\mathcal{Y}_2, \mathcal{Y}_1} \vect{y}\). Then, \(\Vert \vect{y} \Vert_{\tens{M}}^2\) can be decomposed by
    \begin{equation}
        \Vert \vect{y} \Vert_{\tens{M}}^2 = \Vert \vect{y}_{\mathcal{Y}_1} - \vect{\mu}_{\mathcal{Y}_1 \vert \mathcal{Y}_2} \Vert_{\tens{M}_{\mathcal{Y}_1}}^2 + \Vert \vect{y}_{\mathcal{Y}_2} \Vert_{\tens{M}'_{\mathcal{Y}_2}}^2, \label{eq:cond-marg}
    \end{equation}
    where \(\vect{\mu}_{\mathcal{Y}_1 \vert \mathcal{Y}_2} = \tens{M}_{\mathcal{Y}_1}^{\dagger} \vect{\eta}_{\mathcal{Y}_1 \vert \mathcal{Y}_2}\) and
    \begin{subequations}
    \begin{align}
        & \vect{\eta}_{\mathcal{Y}_1 \vert \mathcal{Y}_2} \coloneqq -\tens{M}_{\mathcal{Y}_1, \mathcal{Y}_2} \vect{y}_{\mathcal{Y}_2}, \\
        & \tens{M}'_{\mathcal{Y}_2} \coloneqq \tens{M}_{\mathcal{Y}_2} - \tens{M}_{\mathcal{Y}_1, \mathcal{Y}_2}^{\ast} \tens{M}_{\mathcal{Y}_1}^{\dagger} \tens{M}_{\mathcal{Y}_1, \mathcal{Y}_2}.
    \end{align}
    \end{subequations}
\end{proposition}

\begin{proof}
    Since \(\tens{P}_{\mathcal{Y}_1, \mathcal{Y}_2}\) and \(\tens{P}_{\mathcal{Y}_2, \mathcal{Y}_1}\) are complementary projections, we have \(\vect{y} = \vect{y}_{\mathcal{Y}_1} + \vect{y}_{\mathcal{Y}_2}\). By using this and the trivial identities \(\tens{P}_{\mathcal{Y}_1, \mathcal{Y}_2} = \tens{P}_{\mathcal{Y}_1} \tens{P}_{\mathcal{Y}_1, \mathcal{Y}_2}\) and \(\tens{P}_{\mathcal{Y}_2, \mathcal{Y}_1} = \tens{P}_{\mathcal{Y}_2} \tens{P}_{\mathcal{Y}_2, \mathcal{Y}_1}\), and applying the second equation of \eqref{eq:penrose} for the pseudo-inverse of \(\tens{M}_{\mathcal{Y}_1}\), we can write
    \begin{align}
        \Vert \vect{y} \Vert_{\tens{M}}^2 &= 
        \vect{y}_{\mathcal{Y}_1}^{\ast} \tens{M}_{\mathcal{Y}_1} \vect{y}_{\mathcal{Y}_1} + 2 \vect{y}_{\mathcal{Y}_1}^{\ast} \tens{M}_{\mathcal{Y}_1, \mathcal{Y}_2} \vect{y}_{\mathcal{Y}_2} + \vect{y}_{\mathcal{Y}_2}^{\ast} \tens{M}_{\mathcal{Y}_2} \vect{y}_{\mathcal{Y}_2} \notag \\
                                          &= \Vert \vect{y}_{\mathcal{Y}_1} - \vect{\mu}_{\mathcal{Y}_1 \vert \mathcal{Y}_2} \Vert_{\tens{M}_{\mathcal{Y}_1}}^2 + 2 \vect{y}_{\mathcal{Y}_1}^{\ast} \tens{S} \vect{y}_{\mathcal{Y}_2} + \Vert \vect{y}_{\mathcal{Y}_2} \Vert_{\tens{M}'_{\mathcal{Y}_2}}, \label{eq:SM12}
    \end{align}
    where \(\tens{S} \coloneqq \tens{M}_{\mathcal{Y}_1, \mathcal{Y}_2} - \tens{M}_{\mathcal{Y}_1} \tens{M}_{\mathcal{Y}_1}^{\dagger} \tens{M}_{\mathcal{Y}_1, \mathcal{Y}_2}\). We now show \(\tens{S} = \tens{0}\) as follows. Note that \(\tens{M}_{\mathcal{Y}_1} \tens{M}_{\mathcal{Y}_1}^{\dagger} = \tens{P}_{\image(\tens{M}_{\mathcal{Y}_1})}\) and \(\image(\tens{M}_{\mathcal{Y}_1}) = \mathcal{Y}_1\) since \(\mathcal{Y}_1 \cap \mathcal{X} = \varnothing\) by the hypothesis. On the other hand, \(\image(\tens{M}_{\mathcal{Y}_1, \mathcal{Y}_2}) \subset \mathcal{Y}_1\), therefore, \((\tens{M}_{\mathcal{Y}_1} \tens{M}_{\mathcal{Y}_1}^{\dagger}) \tens{M}_{\mathcal{Y}_1, \mathcal{Y}_2} = \tens{M}_{\mathcal{Y}_1, \mathcal{Y}_2}\). Thus, \(\tens{S}\) vanishes and \eqref{eq:SM12} concludes \eqref{eq:cond-marg}.
\end{proof}

\begin{corollary}[Conditional and Marginal Precisions] \label{cor:cond-marg}
    Suppose \(\tens{M}\) is the precision matrix for \eqref{eq:like-2} in the sense of \Cref{def:precision}. Consider the partition of the data \(\vect{y} = [\vect{y}_1, \vect{y}_2]^{\ast} \in \mathbb{R}^n\) where \(\vect{y}_i \in \mathbb{R}^{n_i}\), \(i = 1, 2\), and \(n_1 + n_2 = n\). Let the corresponding partition of the precision matrix be \(\tens{M} = \begin{bmatrix} \tens{M}_{11} & \tens{M}_{12} \\ \tens{M}_{12}^{\ast} & \tens{M}_{22} \end{bmatrix}\). Also define \(\mathcal{Y}_1 \coloneqq \spn(\vect{e}_1, \dots, \vect{e}_{n_1})\) where \(\vect{e}_i\) is the unit vector along the \(i\)-th axis. If \(\mathcal{Y}_1 \cap \mathcal{X} = \varnothing\), then, the distribution of \(\vect{y}_1\) given \(\vect{y}_2\) is
    \begin{equation}
        \vect{y}_1 \vert \vect{y}_2, \vect{\theta} \sim \mathcal{N}^{-1}(\vect{\eta}_{1 \vert 2}, \tens{M}_{1 \vert 2}),
    \end{equation}
    which is the canonical (or information) parametrization of the normal distribution with the potential vector and precision matrix respectively given by
    \begin{subequations} \label{eq:cond-canonical}
    \begin{align}
        & \vect{\eta}_{1 \vert 2} = -\tens{M}_{12} \vect{y}_2, \\
        & \tens{M}_{1 \vert 2} = \tens{M}_{11}.
    \end{align}
    \end{subequations}
    Furthermore, the marginal distribution of \(\vect{y}_2\) is
    \begin{equation}
        \vect{y}_2 \vert \vect{\theta} \sim \mathcal{N}^{-1}(\vect{0}, \tens{M}'_{22}), \label{eq:y2-marg}
    \end{equation}
    where \(\tens{M}'_{22} \coloneqq \tens{M}_{22} - \tens{M}_{12}^{\ast} \tens{M}_{11}^{\dagger} \tens{M}_{12}\) is the marginal precision.
\end{corollary}

\begin{proof}
    Let \(\mathcal{Y}_2 \coloneqq \mathcal{Y}_1^{\perp}\). The projected vectors \(\vect{y}_{\mathcal{Y}_1}\) and \(\vect{y}_{\mathcal{Y}_2}\) in \Cref{prop:cond-marg} become \(\vect{y}_{\mathcal{Y}_1} = [\vect{y}_1, \vect{0}]^{\ast}\) and \(\vect{y}_{\mathcal{Y}_2} = [\vect{0}, \vect{y}_2]^{\ast}\). Also, the matrices in \eqref{eq:M12} become
    \begin{equation*}
        \tens{M}_{\mathcal{Y}_1} = \begin{bmatrix} \tens{M}_{11} & \tens{0} \\ \tens{0} & \tens{0} \end{bmatrix},
        \quad
        \tens{M}_{\mathcal{Y}_1, \mathcal{Y}_2} = \begin{bmatrix} \tens{0} & \tens{M}_{12} \\ \tens{0} & \tens{0} \end{bmatrix},
        \quad
        \tens{M}_{\mathcal{Y}_2} = \begin{bmatrix} \tens{0} & \tens{0} \\ \tens{0} & \tens{M}_{22} \end{bmatrix}.
    \end{equation*}
    By using the above vectors and matrices in \Cref{prop:cond-marg}, the relation \eqref{eq:cond-marg} becomes
    \begin{equation}
        \Vert \vect{y} \Vert_{\tens{M}}^2 = \Vert \vect{y}_1 - \vect{\mu}_{1 \vert 2} \Vert_{\tens{M}_{11}}^2 + \Vert \vect{y}_2 \Vert_{\tens{M}'_{22}}^2, \label{eq:cond-marg-2}
    \end{equation}
    where \(\vect{\mu}_{1 \vert 2} \coloneqq -\tens{M}_{11}^{\dagger} \tens{M}_{12} \vect{y}_2\). By assuming \(\vect{y}_2\) is given in \eqref{eq:cond-marg-2}, the relation \eqref{eq:like-2} becomes
    \begin{equation*}
        p(\vect{y}_1| \vect{y}_2, \vect{\theta}) \propto \exp\left(-\frac{1}{2} \| \vect{y}_1  - \vect{\mu}_{1 \vert 2}\|_{\tens{M}_{11}}^2 \right), \label{eq:like-3}
    \end{equation*}
    which is a normal distribution with the potential vector \(\vect{\eta}_{1 \vert 2} \coloneqq \tens{M}_{11} \vect{\mu}_{1 \vert 2} = -\tens{M}_{12} \vect{y}_{2}\) and the precision matrix \(\tens{M}_{11}\) and concludes \eqref{eq:cond-canonical}. Also, by using \eqref{eq:cond-marg-2} in \eqref{eq:like-2} and marginalizing with respect to \(\vect{y}_1\), we get
    \begin{equation}
        p(\vect{y}_2 \vert \vect{\theta}) \propto \exp\left(-\frac{1}{2} \| \vect{y}_2 \|_{\tens{M}'_{22}}^2 \right),
    \end{equation}
    which concludes \eqref{eq:y2-marg}.
\end{proof}

\begin{remark}[Conditional-Marginal Decomposition]
    The relation \eqref{eq:cond-marg-2} is the decomposition of the Mahalanobis distance \(\Vert \vect{y} \Vert_{\tens{M}}^2\) into the conditional part (the quadratic term involving \(\vect{y}_1\)) and the marginal part (the quadratic term involving \(\vect{y}_2\)). \Cref{prop:cond-marg} extends this decomposition by allowing \(\vect{y}\) to be projected onto subspaces \(\mathcal{Y}_1\) and \(\mathcal{Y}_2\) with \(\tens{M}_{\mathcal{Y}_1}\) and \(\tens{M}_{\mathcal{Y}_2}\) representing the compressions of \(\tens{M}\) on \(\mathcal{Y}_1\) and \(\mathcal{Y}_2\), respectively.  \Cref{cor:cond-marg} is as a specific instance of \Cref{prop:cond-marg} where \(\mathcal{Y}_1 \coloneqq \spn(\vect{e}_1, \dots, \vect{e}_{n_1})\) and \(\mathcal{Y}_2 \coloneqq \spn(\vect{e}_{n_1+1}, \dots, \vect{e}_n)\).
\end{remark}

% =======================================
% Asymptotic Analysis of Gaussian Process
% =======================================

\subsection{Asymptotic Analysis of Gaussian Process} \label{sec:app-asym}

We present another application of \Cref{thm:M-rep} that provides an asymptotic analysis for the Gaussian process regression with a linear mixed model. Consider the covariance matrix
\begin{equation}
    \gtens{\Sigma}(\sigma, \varsigma) = \sigma^2 \tens{K} + \varsigma^2 \tens{I}, \label{eq:mixed-cov}
\end{equation}
where the matrix \(\tens{K} \neq \tens{I}\) is SPD, and the hyperparameters \(\sigma^2\) and \(\varsigma^2\) are the variances of regression misfit and input noise, respectively. %We represent the matrix \(\tens{M}\) of the corresponding mixed covariance by an infinite series.

\begin{proposition} \label{prop:mixed-cov}
    The precision matrix in \eqref{eq:M-like} for the likelihood function \eqref{eq:like} with the covariance matrix \eqref{eq:mixed-cov} can be represented by the series
    \begin{equation}
        \tens{M}(t, \varsigma) = \varsigma^{-2} \tens{P} \sum_{i=0}^{\infty} (-t \tens{K} \tens{P})^i, \label{eq:M-series}
    \end{equation}
    where \(t \coloneqq \sigma^2 / \varsigma^2\), and \(\tens{P} = \tens{I} - \tens{X} (\tens{X}^{\ast} \tens{X})^{-1} \tens{X}^{\ast}\). The above series is convergent if \(t < \lambda_n^{-1}\), where \(\lambda_n\) is the largest eigenvalue of \(\tens{K}\).
\end{proposition}

\begin{proof}
    For simplicity, we write \eqref{eq:mixed-cov} as \(\gtens{\Sigma}(\sigma, \varsigma) = \varsigma^2 \gtens{\Sigma}_t\) where \(\gtens{\Sigma}_t \coloneqq (\tens{I} + t \tens{K})\). We express \eqref{eq:M-like} by
    \begin{equation}
        \varsigma^2 \tens{M}(t, \varsigma) = \gtens{\Sigma}_{t}^{-1} - \gtens{\Sigma}_{t}^{-1} \tens{X} \left( \tens{X}^{\ast} \gtens{\Sigma}_{t}^{-1} \tens{X} \right)^{-1} \tens{X}^{\ast} \gtens{\Sigma}_{t}^{-1}. \label{eq:M-varsigma}
    \end{equation}
    By defining \(\tens{N} = \tens{I} - \tens{P} + \gtens{\Sigma}_t \tens{P}\) and using \Cref{thm:M-rep}, we can express \(\varsigma^2 \tens{M}(t, \varsigma)\) in \eqref{eq:M-varsigma} by
    \begin{equation}
        \varsigma^2 \tens{M}(t, \varsigma) = \tens{P} \left( \tens{I} + t \tens{K} \tens{P} \right)^{-1}. \label{eq:M-varsigma-2}
    \end{equation}
    The Neumann series of \((\tens{I} + t\tens{K}\tens{P})^{-1}\) (see \eg \cite[p. 320, Lemma 1]{DAUTRAY-2000}) enables the representation of  \eqref{eq:M-varsigma-2} as per \eqref{eq:M-series}. Convergence of the Neumann series is guaranteed under the condition that \(\Vert t \tens{K} \tens{P} \Vert < 1\). The norm inequality \(\Vert \tens{K} \tens{P} \Vert \leq \Vert \tens{K} \Vert \Vert \tens{P} \Vert\) and the property that the norm of a projection matrix is equal to \(1\), \ie \(\Vert \tens{P} \Vert = 1\), allow us to impose the stronger condition \(t \Vert \tens{K} \Vert < 1\). By utilizing the 2-norm of the symmetric positive definite (SPD) matrix \(\tens{K}\), given as \(\Vert \tens{K} \Vert = \lambda_n\), the proof is complete.
\end{proof}

When \(t\) is small, \ie \(t \lambda_n \ll 1\), it is often useful to represent \(\tens{M}(t, \varsigma)\) through a truncation of series \eqref{eq:M-series}. This leads to an asymptotic analysis of the likelihood function, as shown by \citet{AMELI-2022-a}. Their study demonstrates that this truncated series approach is an efficient approximation method to train the Gaussian process regression.

% ==================
% Numerical Analysis
% ==================

\section{Numerical Analysis} \label{sec:num}

We provide a numerical analysis of the determinant relations given in \Cref{sec:det} assuming \(\tens{A}\) is non-singular, \(\tens{X} = \tens{Y}\) where \(\tens{X}\) is full rank, and variables are defined on the real domain. We evaluate
\begin{subequations}
\begin{align}
    \logdet(\tens{A}, \tens{X}) \coloneqq& \logdet(\tens{A}) + \logdet(\tens{X}^{\ast} \tens{A}^{-1} \tens{X}) \tag{LD1} \label{eq:LD1} \\
                                =& \logdet(\tens{X}^{\ast} \tens{X}) + \logdet(\tens{N}) \tag{LD2} \label{eq:LD2} \\
                                =& \logdet(\tens{X}^{\ast} \tens{X}) + \logdet(\tens{U}_{\spc{X}^{\perp}}^{\ast} \tens{A} \tens{U}_{\spc{X}^{\perp}}) \tag{LD3} \label{eq:LD3},
\end{align}
\end{subequations}
where \(\logdet(\tens{A}) \coloneqq \log |\det(\tens{A})|\), and recall that \(\tens{N}\) is given in \eqref{eq:def-N2}, and \(\tens{U}_{\spc{X}^{\perp}}\) is defined in \Cref{cor:comp}. The second equality in the above is obtained from \eqref{eq:det-MN} and the third equality in the above can be realized by \eqref{eq:pinv-M-p} and \eqref{eq:pdet}. We have seen the application of \eqref{eq:LD1} in \Cref{sec:app-like} as \(\logdet(\gtens{\Sigma}, \tens{X})\) appears in the logarithm of the likelihood function \eqref{eq:like}. The latter term is the most computationally expensive part of evaluating the aforementioned log-likelihood function (see \eg \cite{AMELI-2022-c}). The relations \eqref{eq:LD2} and \eqref{eq:LD3} provide alternative computational methods for this function.

The goal of our numerical analysis is to compare \eqref{eq:LD1}, \eqref{eq:LD2}, and \eqref{eq:LD3}. Computing with either of these relations have advantages and disadvantages. For instance, \eqref{eq:LD1} requires solving an \(n \times n\) linear system (for inverting \(\tens{A}\)), whereas \eqref{eq:LD2} require solving a smaller \(p \times p\) linear system (for inverting \(\tens{X}^{\ast} \tens{X}\)), assuming \(p < n\). Furthermore, \eqref{eq:LD2} and \eqref{eq:LD3} can take advantage of an orthonormal matrix \(\tens{X}\) to simplify \(\tens{X}^{\ast} \tens{X}\) with the identity matrix. Also, \eqref{eq:LD1} and \eqref{eq:LD3} can take advantage of an SPD matrix \(\tens{A}\) to exploit Cholesky decomposition to either solve a linear system for \(\tens{A}\) in \eqref{eq:LD1}, or to compute \(\logdet\) of \(\tens{A}\) and \(\tens{U}_{\spc{X}^{\perp}}^{\ast} \tens{A} \tens{U}_{\spc{X}^{\perp}}\). Whereas in \eqref{eq:LD2}, \(\tens{N}\) is not necessarily SPD, so we can only employ LU decomposition to compute \(\logdet(\tens{N})\), which costs twice as much as the Cholesky decomposition \citep[Theorem 3.2.1]{GOLUB-1996}.

An efficient implementation of computing \eqref{eq:LD1}, \eqref{eq:LD2}, and \eqref{eq:LD3} are respectively given in \Cref{alg:ld1}, \Cref{alg:ld2}, and \Cref{alg:ld3}. A few remarks on the algorithms are as follows.

% Creates box for columns of algorithm environment, thus, the multi-column comments are aligned.
% The advantage of this approach compared to creating a table inside algorithm is that algorithm2e
% treats a table as one line, so the line-numberings will not be correct. However, with box, each
% row is counted properly, and the line numbers appears as intended.
% Use \algrow below for lines inside if-else condition.
\newcommand{\algrow}[3]{%
    \makebox[0.35\hsize][l]{#1}%
    \makebox[0.51\hsize][l]{#2}%
    \makebox[0.1\hsize][l]{#3}\;}

% Use \algrowtwo for lines that are not inside an if-else condition
\newcommand{\algrowtwo}[3]{%
    \makebox[0.376\hsize][l]{#1}%
    \makebox[0.4885\hsize][l]{#2}%
    \makebox[0.1\hsize][l]{#3}\;}

% -------------
% Algorithm LD1
% -------------

% \setcounter{algocf}{1}
% \setcounter{AlgoLine}{0}
\begin{algorithm}[t!]
    \caption{Computing \(\logdet(\tens{A}, \tens{X})\) using \eqref{eq:LD1}}
    \label{alg:ld1}

    \SetAlgoShortEnd
    \DontPrintSemicolon
    \SetSideCommentRight
    \SetAlFnt{\footnotesize}
    \LinesNumberedHidden

    \small
    % \scriptsize
    % \fontsize{9}{11}\selectfont

    % Input and Output
    \SetKwInOut{Input}{Input}
    \SetKwInOut{Output}{Output}

    \Input{%
        \(\tens{A}\) -- \(n \times n\) non-singular matrix,
        \(\tens{X}\) -- \(n \times p\) full column-rank matrix
    }
    \Output{%
        \(\logdet(\tens{A}, \tens{X})\)
    }

    \BlankLine

    % Function logdet
    \SetKwFunction{ldone}{logdet}
    \SetKwProg{Fn}{Procedure}{}{}
    \Fn{\(\ldone(\tens{A}, \tens{X})\)}{
        \If {\(\tens{A}\) is SPD}{
                
            \algrow{\(\tens{L}_{\tens{A}} \tens{L}_{\tens{A}}^{\ast} \gets \mathbf{A}\)}%
                   {\tri Cholesky decomposition}%
                   {\(\bigO(\frac{1}{3} n^3)\)}

            \algrow{\(\tens{Y} \gets \tens{L}_{\tens{A}}^{-1} \tens{X}\)}%
                   {\tri Solve lower triangular system \(\tens{L}_{\tens{A}} \tens{Y} = \tens{X}\)}%
                   {\(\bigO(\frac{1}{2}n^2 p)\)}

            \ShowLn
            \algrow{\(\tens{W} \gets \tens{Y}^{\ast} \tens{Y}\)}%
                   {\tri Gramian matrix multiplication}%
                   {\(\bigO(\gamma n p^2)\)}
            \label{alg1:l5}

            \algrow{\(\tens{L}_{\tens{W}} \tens{L}_{\tens{W}}^{\ast} \gets \mathbf{W}\)}%
                   {\tri Cholesky decomposition}%
                   {\(\bigO(\frac{1}{3} p^3)\)}

            \ShowLn
            \KwRet \(2 \mathtt{logdet}(\tens{L}_{\tens{A}}) + 2 \mathtt{logdet}(\tens{L}_{\tens{W}})\)\;
            \label{alg1:l7}
        }
        \Else{

            \ShowLn
            \algrow{\(\tens{L}_{\tens{A}}, \tens{U}_{\tens{A}}, \tens{P}_{\tens{A}} \gets \mathtt{lu}(\mathbf{A})\)}%
                   {\tri LU decomposition \(\tens{P}_{\tens{A}}{\tens{A}} = \tens{L}_{\tens{A}} \tens{U}_{\tens{A}}\)}%
                   {\(\bigO(\frac{2}{3} n^3)\)}
            \label{alg1:l9}

            \ShowLn
            \algrow{\(\tens{Z} \gets \tens{L}_{\tens{A}}^{-1} (\tens{P}_{\tens{A}}\tens{X})\)}%
                   {\tri Solve lower triangular system \(\tens{L}_{\tens{A}} \tens{Z} = \tens{P}_{\tens{A}}\tens{X}\)}%
                   {\(\bigO(\frac{1}{2}n^2 p)\)}
            \label{alg1:l10}

            \algrow{\(\tens{Y} \gets \tens{U}_{\tens{A}}^{-1} \tens{Z}\)}%
                   {\tri Solve upper triangular system \(\tens{U}_{\tens{A}} \tens{Y} = \tens{Z}\)}%
                   {\(\bigO(\frac{1}{2}n^2 p)\)}

            \algrow{\(\tens{W} \gets \tens{X}^{\ast} \tens{Y}\)}%
                   {\tri Inner product of columns of \(\tens{X}\) and \(\tens{Y}\)}%
                   {\(\bigO(n p^2)\)}

            \algrow{\(\tens{L}_{\tens{W}}, \tens{U}_{\tens{W}}, \tens{P}_{\tens{W}} \gets \mathtt{lu}(\mathbf{W})\)}%
                   {\tri LU decomposition \(\tens{P}_{\tens{W}} \tens{W} = \tens{L}_{\tens{W}} \tens{U}_{\tens{W}}\)}%
                   {\(\bigO(\frac{2}{3} p^3)\)}

            \KwRet \(\mathtt{logdet}(\tens{U}_{\tens{A}}) + \mathtt{logdet}(\tens{U}_{\tens{W}})\)\;
        }
    }

\end{algorithm}

% -------------
% Algorithm LD2
% -------------

% \setcounter{algocf}{2}
% \setcounter{AlgoLine}{0}
\begin{algorithm}[t!]
    \caption{Computing \(\logdet(\tens{A}, \tens{X})\) using \eqref{eq:LD2}}
    \label{alg:ld2}

    \SetAlgoShortEnd
    \DontPrintSemicolon
    \SetSideCommentRight
    \SetAlFnt{\footnotesize}
    \LinesNumberedHidden

    \small
    % \scriptsize
    % \fontsize{9}{11}\selectfont

    % Input and Output
    \SetKwInOut{Input}{Input}
    \SetKwInOut{Output}{Output}

    \Input{%
        \(\tens{A}\) -- \(n \times n\) non-singular matrix, 
        \(\tens{X}\) -- \(n \times p\) full column-rank matrix
    }
    \Output{%
        \(\logdet(\tens{A}, \tens{X})\)
    }

    \BlankLine

    % Function logpdet
    \SetKwFunction{ldtwo}{logdet}
    \SetKwProg{Fn}{Procedure}{}{}
    \Fn{\(\ldtwo(\tens{A}, \tens{X})\)}{

        \ShowLn
        \(\tens{C} \gets \tens{A} - \tens{I}\)\; \label{alg2:l2}
        \If {\(\tens{X}\) is orthonormal}{

            \algrow{\(\tens{D} \gets \tens{C} \tens{X}\), \(\tens{E} \gets \tens{D} \tens{X}^{\ast}\)}%
                   {\tri Matrix multiplications}%
                   {\(\bigO(2 n^2 p)\)}
        }
        \Else{

            \algrow{\(\tens{V} \gets \tens{X}^{\ast} \tens{X}\)}%
                   {\tri Gramian matrix multiplication}%ss
                   {\(\bigO(\gamma n p^2)\)}

            \algrow{\(\tens{L}_{\tens{V}} \tens{L}_{\tens{V}}^{\ast} \gets \mathbf{V}\)}%
                   {\tri Cholesky decomposition}%
                   {\(\bigO(\frac{1}{3} p^3)\)}

                   \algrow{\(\tens{Y} \gets \tens{X} (\tens{L}_{\tens{V}}^{-1})^{\ast}\)}%
                    {\tri Solve lower triangular system \(\tens{L}_{\tens{V}} \tens{Y}^{\ast} = \tens{X}^{\ast}\)}%
                    {\(\bigO(\frac{1}{2}n p^2)\)}

            \algrow{\(\tens{D} \gets \tens{C} \tens{Y}\), \(\tens{E} \gets \tens{D} \tens{Y}^{\ast}\)}%
                   {\tri Matrix multiplications}%
                   {\(\bigO(2 n^2 p)\)}
        }

        \ShowLn
        \(\tens{N} \gets \tens{A} - \tens{E}\)\;
        \label{alg2:l10}

        \algrowtwo{\(\tens{L}_{\tens{N}}, \tens{U}_{\tens{N}}, \tens{P}_{\tens{N}} \gets \mathtt{lu}(\tens{N})\)}%
                  {\tri LU decomposition \(\tens{P}_{\tens{N}} \tens{N} = \tens{L}_{\tens{N}} \tens{U}_{\tens{N}}\)}%
                  {\(\bigO(\frac{2}{3} n^3)\)}

        \BlankLine

        \lIf {\(\tens{X}\) is orthonormal}{
            \KwRet \(\mathtt{logdet}(\tens{U}_{\tens{N}})\)
        }
        \lElse{ 
            \KwRet \(2\mathtt{logdet}(\tens{L}_{\tens{V}}) + \mathtt{logdet}(\tens{U}_{\tens{N}})\)
        }
    }
\end{algorithm}

% Creates box for columns of algorithm environment, thus, the multi-column comments are aligned.
% The advantage of this approach compared to creating a table inside algorithm is that algorithm2e
% treats a table as one line, so the line-numberings will not be correct. However, with box, each
% row is counted properly, and the line numbers appears as intended.
% Use \algrow below for lines inside if-else condition.
\newcommand{\algrowthree}[3]{%
    \makebox[0.35\hsize][l]{#1}%
    \makebox[0.51\hsize][l]{#2}%
    \makebox[0.1\hsize][l]{#3}\;}

% Use \algrowtwo for lines that are not inside an if-else condition
\newcommand{\algrowfour}[3]{%
    \makebox[0.376\hsize][l]{#1}%
    \makebox[0.4885\hsize][l]{#2}%
    \makebox[0.1\hsize][l]{#3}\;}

% -------------
% Algorithm LD3
% -------------

% \setcounter{algocf}{2}
% \setcounter{AlgoLine}{0}
\begin{algorithm}[ht!]
    \caption{Computing \(\logdet(\tens{A}, \tens{X})\) using \eqref{eq:LD3}}
    \label{alg:ld3}

    \SetAlgoShortEnd
    \DontPrintSemicolon
    \SetSideCommentRight
    \SetAlFnt{\footnotesize}
    \LinesNumberedHidden

    \small
    % \scriptsize
    % \fontsize{9}{11}\selectfont

    % Input and Output
    \SetKwInOut{Input}{Input}
    \SetKwInOut{Output}{Output}

    \Input{%
        \(\tens{A}\) -- \(n \times n\) non-singular matrix, 
        \(\tens{X}\) -- \(n \times p\) full column-rank matrix
    }
    \Output{%
        \(\logdet(\tens{A}, \tens{X})\)
    }

    \BlankLine

    % Function logpdet
    \SetKwFunction{ldtwo}{logdet}
    \SetKwProg{Fn}{Procedure}{}{}
    \Fn{\(\ldtwo(\tens{A}, \tens{X})\)}{

        \lIf {\(\tens{X}\) is orthonormal}{
            \(\tens{X}' \gets \tens{X}\)
        }
        \Else{
            \ShowLn
            \algrowthree{\(\tens{X}' \gets \operatorname{gs}(\tens{X})\)}%
                       {\tri Gram-Schmidt orthonormalization}%
                       {\(\bigO(2 n p^2)\) \label{alg3:gs}}
        }

        \BlankLine

        \algrowfour{\(\tens{U}_{\spc{X}^{\perp}} \gets \operatorname{rand}(n, n-p)\)}%
                   {\tri Initialize a random matrix in \(\mathcal{M}_{n, n-p}(\mathbb{R})\)}%
                   {}

        \ShowLn
        \algrowfour{\(\tens{U}_{\spc{X}^{\perp}} \gets \operatorname{gs}(\tens{U}_{\spc{X}^{\perp}}, \tens{X}')\)}%
                   {\tri Orthonormalize \(\tens{U}_{\spc{X}^{\perp}}\) against \(\spc{X}\)}%
                   {\(\bigO(2 \delta n (n^2-p^2))\) \label{alg3:gs2}}

        \algrowfour{\(\tens{Y} \gets \tens{A} \tens{U}_{\spc{X}^{\perp}}\)}%
                   {\tri Matrix multiplication}%
                   {\(\bigO(n^2 (n-p))\)}

        \algrowfour{\(\tens{Z} \gets \tens{U}_{\spc{X}^{\perp}}^{\ast} \tens{Y}\)}%
                   {\tri \(\gamma=1\) if \(\tens{A}\) is not SPD.}%
                   {\(\bigO(\gamma n (n-p)^2)\)}

        \BlankLine

        \If {\(\tens{A}\) is SPD}{
            \algrowthree{\(\tens{L}_{\tens{Z}} \tens{L}_{\tens{Z}}^{\ast} \gets \mathbf{Z}\)}%
                        {\tri Cholesky decomposition}%
                        {\(\bigO(\frac{1}{3} (n-p)^3)\)}

            \(\alpha \gets 2\logdet(\tens{L}_{\tens{Z}})\)\;
        }
        \Else{
            \algrowthree{\(\tens{L}_{\tens{Z}}, \tens{U}_{\tens{Z}}, \tens{P}_{\tens{Z}} \gets \mathtt{lu}(\tens{Z})\)}%
                        {\tri LU decomposition \(\tens{P}_{\tens{Z}} \tens{Z} = \tens{L}_{\tens{Z}} \tens{U}_{\tens{Z}}\)}%
                        {\(\bigO(\frac{2}{3} (n-p)^3)\)}
            \(\alpha \gets \logdet(\tens{U}_{\tens{Z}})\)\;
        }

        \BlankLine

        \lIf {\(\tens{X}\) is orthonormal}{
            \KwRet \(\alpha\)
        }
        \Else{

            \algrowthree{\(\tens{V} \gets \tens{X}^{\ast} \tens{X}\)}%
                        {\tri Gramian matrix multiplication}%
                        {\(\bigO(\gamma n p^2)\)}

            \algrowthree{\(\tens{L}_{\tens{V}} \tens{L}_{\tens{V}}^{\ast} \gets \mathbf{V}\)}%
                        {\tri Cholesky decomposition}%
                        {\(\bigO(\frac{1}{3} p^3)\)}

            \KwRet \(2\logdet(\tens{L}_{\tens{V}}) + \alpha\)\;
        }
    }
\end{algorithm}

The function \texttt{lu}, such as in \Cref{alg1:l9}, represents the LU decomposition where the matrices \(\tens{L}_{\tens{A}}\) and \(\tens{U}_{\tens{A}}\) therein are lower and upper triangular, respectively. The diagonals of \(\tens{L}_{\tens{A}}\) are normalized to \(1\) such as by the Doolittle algorithm. The LU decomposition is performed with partial pivoting where \(\tens{P}_{\tens{A}}\) is the corresponding permutation matrix, which can be efficiently stored by a one-dimensional array of the indices of row permutations. Also, instead of the matrix multiplication \(\tens{P}_{\tens{A}} \tens{A}\) in \Cref{alg1:l10}, we permute the pointer of the rows of \(\tens{A}\) by \(\mathcal{O}(n)\) operations.

The \(\mathtt{logdet}\) function with only one argument (such as in \Cref{alg1:l7}) performs on the triangular matrices of Cholesky and LU decompositions by \(\mathtt{logdet}(\tens{L}) \gets \sum_{i} \log |L_{ii}|\) where \(L_{ii}\) are the diagonals of \(\tens{L}\). To determine the sign of the determinant, the number of negative diagonals and the parity of the permutation matrix (in the case of LU decomposition) should be considered.

Note that computing \(\tens{N}\) by direct substitution of \(\tens{P}\) in \(\tens{N} = \tens{I} - \tens{P} + \tens{A} \tens{P}\) from \eqref{eq:def-N2} is inefficient since the matrix product \(\tens{A} \tens{P}\) takes \(\mathcal{O}(n^3)\) operations. An efficient computation of \(\tens{N}\) is given from \Cref{alg2:l2} to \Cref{alg2:l10} of \Cref{alg:ld2}, which minimizes the computational cost of matrix multiplication.

The function \(\mathtt{gs}\) in \Cref{alg3:gs} of \Cref{alg:ld3} represents the Gram-Schmidt orthonormalization of \(\tens{X}\). Also, \Cref{alg3:gs2} is the orthonormalization of the randomly generated matrix \(\tens{U}_{\spc{X}^{\perp}}\) against \(\spc{X}\). The complexity of the latter step is obtained by subtracting the complexities of orthonormalizing the \(n \times n\) matrix \([\tens{X}, \tens{U}_{\spc{X}^{\perp}}]\) and the \(n \times p\) matrix \(\tens{X}\) \citep[p. 60]{TREFETHEN-1997}.

The computational complexity of each task is given in the last column of algorithms, which is the leading order of the counts of the multiply-accumulate (MAC) operation consisting of one addition and one multiplication. Note the Gramian matrix multiplications (such as in \Cref{alg1:l5}) may be carried out with half of the full matrix multiplication operation, \ie \(\gamma = \frac{1}{2}\). Efficient implementation of Gramian matrix product is available, for instance, by \texttt{?syrk} routines
% \footnote{By convention in Fortran libraries, \texttt{?syrk} indicates data type specializations \texttt{dsyrk}, \texttt{ssyrk}, \texttt{csyrk}, etc.} 
in LAPACK, BLAS, cuBLAS (CUDA), and clBLAS (OpenCL) libraries. However, in most higher-level numerical packages, Gramian matrix multiplication is not readily available. In such a case, we can adjust the corresponding computational complexity by setting \(\gamma = 1\). Also, the Boolean variable \(\delta\) in \Cref{alg3:gs2} is set to \(0\) if \(\tens{U}_{\spc{X}^{\perp}}\) is already pre-computed, otherwise it is set to \(1\).

\begin{table}[t!]
    \caption{Comparison of complexities of computing \(\logdet(\tens{A}, \tens{X})\) using \Cref{eq:LD1}, \Cref{eq:LD2}, and \Cref{eq:LD3}}
    \label{table:complexity}
    \centering
    \small
    % \fontsize{7.5}{9.5}\selectfont
    \begin{tabular}{c|c}
        \toprule
        Function & Complexity per \(n^3\) \\
        \midrule
        \eqref{eq:LD1} & \(
        \begin{cases}
            \frac{1}{3} + \frac{1}{2}\rho + \gamma \rho^2 + \frac{1}{3}\rho^3, & \text{if }\tens{A}\text{ is SPD} \\
            \frac{2}{3} + \rho + \rho^2 + \frac{2}{3}\rho^3, & \text{otherwise} 
        \end{cases}\) \\

        \midrule

        \eqref{eq:LD2} & \(
        \frac{2}{3} + 2 \rho +
        \begin{cases}
            0, & \text{if } \tens{X} \text{ is orthonormal} \\
            (\gamma + \frac{1}{2}) \rho^2 + \frac{1}{3} \rho^3, & \text{otherwise}
        \end{cases}\) \\

        \midrule

        \eqref{eq:LD3} & \(
            1-\rho +
        \begin{cases}
            \gamma (1-\rho)^2 + \frac{1}{3} (1-\rho)^3, & \text{if } \tens{A} \text{ is SPD}\\
            (1-\rho)^2 + \frac{2}{3}(1-\rho)^3, & \text{otherwise}
        \end{cases}
        +
        \begin{cases}
            2\delta(1-\rho^2), & \text{if } \tens{X} \text{ is orthonormal} \\
            2 + \gamma \rho^2 + \frac{1}{3} \rho^3, & \text{otherwise}
        \end{cases}
        \) \\
        \bottomrule
    \end{tabular}
\end{table}

\Cref{table:complexity} summarizes the computational complexities of the three algorithms per \(n^3\) operations as a function of \(\rho \coloneqq p / n\). Note that the computational complexity of \eqref{eq:LD1} is independent of whether \(\tens{X}\) is orthonormal. Also, the cost of computing \eqref{eq:LD2} is independent of whether \(\tens{A}\) is SPD.

\begin{figure}[t!]
    \centering
    \includegraphics[width=0.638\textwidth]{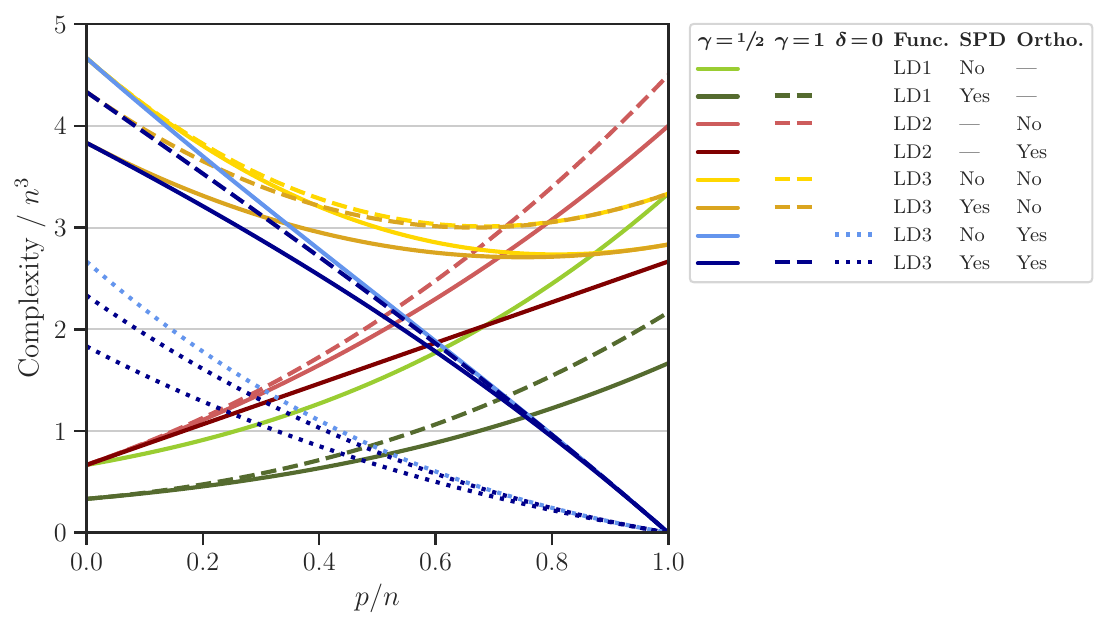}
    \caption{Complexity per \(n^3\) of computing \(\logdet(\tens{A}, \tens{X})\) using \eqref{eq:LD1}, \eqref{eq:LD2}, and \eqref{eq:LD3}. The solid and dashed lines correspond to \(\gamma = \frac{1}{2}\) and \(\gamma=1\), respectively, where \(\delta=1\) in both cases. The case of \(\delta = 0\) (for either of \(\gamma=\frac{1}{2}\) or \(\gamma=1\)) is shown by the dotted curves. The fifth and sixth columns of the plot legend respectively indicate whether \(\tens{A}\) is SPD and \(\tens{X}\) is orthonormal.}
    \label{fig:complexity}
\end{figure}

\Cref{fig:complexity} demonstrates the computational complexities of \Cref{table:complexity}. In general, \eqref{eq:LD3} is the fastest among the relations when \(\tens{X}\) is orthonormal and \(p/n\) is large. Also, when \(\tens{A}\) is SPD, \eqref{eq:LD1} is faster than other relations if \(p/n\) is not large. However, in certain conditions, such as when \(\tens{A}\) is not SPD and \(\tens{X}\) is orthonormal, \eqref{eq:LD2} can be preferred. We investigate this point through a numerical experiment below.

We use an SPD matrix \(\tens{A}\) of the size \(n = 2^9\) obtained from the covariance of an electrocardiogram signal which is further described in \Cref{sec:dataset}. We note, however, that the results of our numerical experiment is not changed if other matrices are used. Also, we create the matrix \(\tens{X}\) of the size \(n \times (n-1)\) by discretizing the trigonometric basis functions of various frequencies in the unit interval. The columns of \(\tens{X}\) are orthonormalized by the Gram-Schmidt process. Throughout the experiment, we used the first \(p\) columns of \(\tens{X}\) and vary \(1 \leq p \leq n-1\). We considered the two cases: \(\tens{A}\) is SPD and \(\tens{A}\) not SPD for the same generated matrix. Also, we considered the case \(\tens{X}\) is orthonormal and the case it is not orthonormal. For each of these four cases, we compute \eqref{eq:LD1}, \eqref{eq:LD2}, and \eqref{eq:LD3} and repeat the numerical experiment multiple times for a better estimation of the measures of the computational cost.

For our numerical experiment, we developed the python package \texttt{detkit} \citep{AMELI-2022-d}, which provides an efficient implementation of \Cref{alg:ld1}, \Cref{alg:ld2}, and \Cref{alg:ld3} in \CC. To properly measure the computational cost, all library components should be compiled with the same configuration. Because of this, our implementation does not depend on external libraries. The computations were carried out on an Intel Xeon E5-2680 v4 processor with a peak performance of \(26.4\) GFLOP/sec per core with level-3 optimization of the GCC compiler. The source code to reproduce the input dataset and output results in this section can be found in the documentation of the software package\footnote{See \url{https://ameli.github.io/detkit} and in particular, the benchmark page therein.}. A minimalistic usage of this package is shown in \Cref{list:detkit}.

\begin{lstlisting}[
    style=mystyle,
    language=Python,
    caption={A minimalistic usage of \texttt{detkit} package. The function \texttt{loggdet} computes \(\logdet(\tens{A}, \tens{X})\).},
    label={list:detkit},
    float=tp]
# Install detkit with ;#"\texttt{pip install detkit}"#;
from detkit import loggdet, ortho_complement, get_instructions_per_task
from detkit.datasets import covariance_matrix, design_matrix
from numpy.random import randn

# Generate a sample SPD matrix ;#$ \tens{A} $#;
n, p = 2**9, 2**8
A = covariance_matrix(n)

# Generate a sample orthonormal matrix ;#$ \tens{X} $#;
X = design_matrix(n, p, ortho=True)

# ;#\texttt{method}#; can be ;#\texttt{'legacy'}#;, ;#'\texttt{proj}'#;, and ;#'\texttt{comp}'#;, respectively for ;#\eqref{eq:LD1}#;, ;#\eqref{eq:LD2}#;, and ;#\eqref{eq:LD3}#;.
ld, sign, inst = loggdet(A, X, method='legacy', sym_pos=True, X_orth=True, flops=True)

# Compute flops from hardware instruction counts
flops = inst / get_instructions_per_task()

# Pre-compute ;#$\tens{U}_{\spc{X}^{\perp}}$#; (here called ;#\texttt{Xp}#;, ;#\ie#;the matrix perpendicular to ;#$\tens{X}$#;)
Xp = randn(n, n-p); ortho_complement(Xp, X, X_orth=True)
ld, sign = loggdet(A, X, Xp=Xp, method='comp', sym_pos=True, X_orth=True)
\end{lstlisting}

\begin{figure}[t!]
    \centering
    \includegraphics[width=\textwidth]{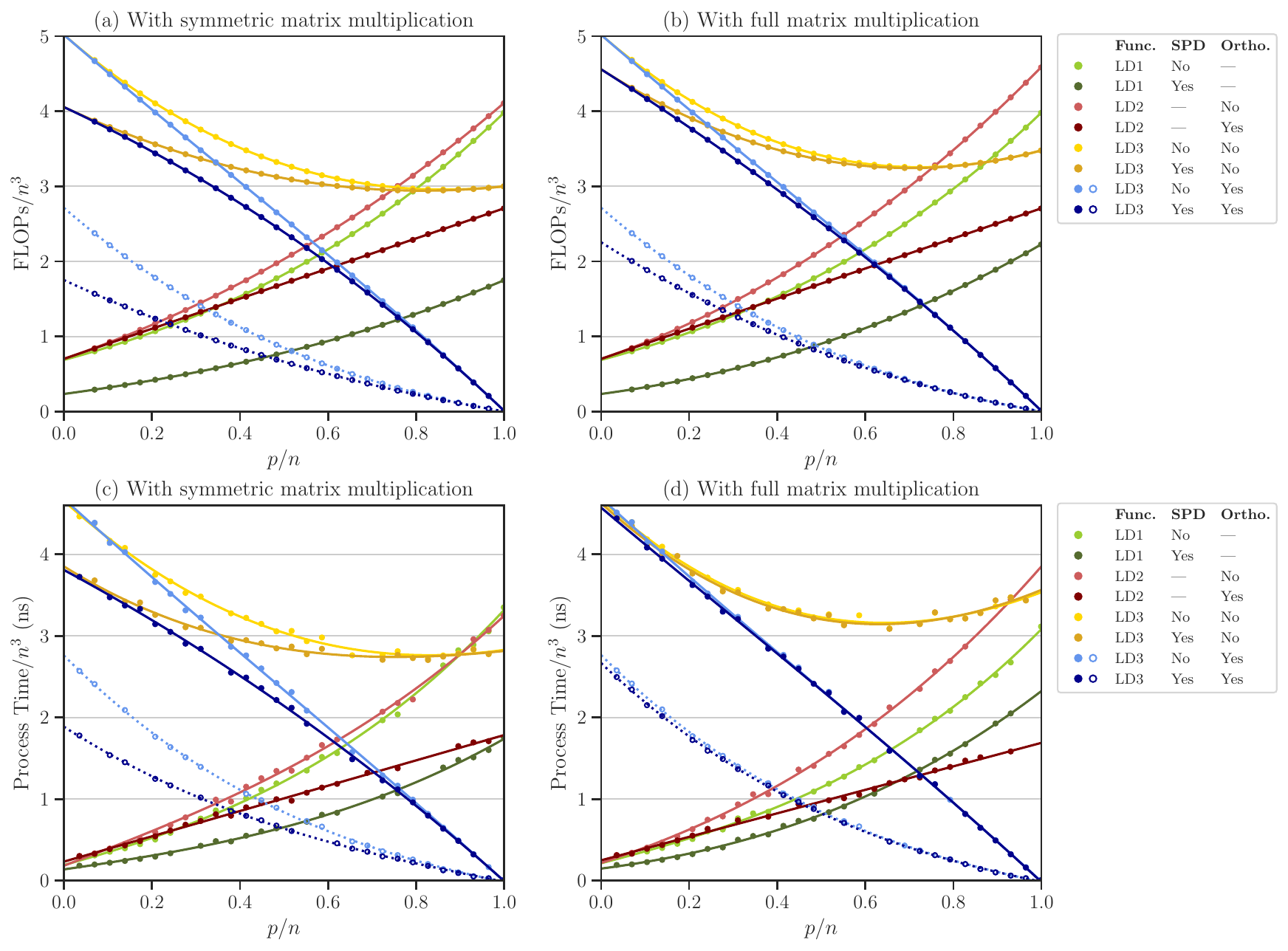}
    \caption{FLOPs (first row) and processing time (second row) of computing \(\logdet(\tens{A}, \tens{X})\) for \(n = 2^9\) using \eqref{eq:LD1}, \eqref{eq:LD2}, and \eqref{eq:LD3}. On the left and right columns of the plot, respectively, the symmetric and full matrix multiplication for computing Gram matrices are used, which are comparable to \(\gamma = \frac{1}{2}\) and \(\gamma=1\) in \Cref{fig:complexity}. The hollow circles correspond to the case when \(\tens{U}_{\spc{X}^{\perp}}\) is already pre-computed and the cost of its computation is excluded, which are comparable to \(\delta=0\) in \Cref{fig:complexity}. The solid and dotted curves are the third order polynomial fit to the data. The fourth and fifth columns of the plot legends respectively indicate whether \(\tens{A}\) is SPD and \(\tens{X}\) is orthonormal.}
    \label{fig:FLOPs}
\end{figure}

For the comparison, we measured the floating-point operations (FLOP) of the computational process. For consistency with MAC operations described earlier, we define one FLOP as a single fused multiply-add (FMA) operation on the processor. Unfortunately, there is no unique way of measuring FLOPs on modern processors. To measure FLOPs, we counted the ``retired hardware instruction events''\footnote{Using the \CC\ API for \texttt{perf} tool, a performance counter for Linux kernel.} during the runtime divided by the same instruction count it takes for a single FMA task on that processor. Instruction counts not only depends on the processor but also the compiler configurations and the matrix size. To obtain an instruction count of a single FMA task that is independent of the matrix size, \(n\), we varied \(n\) to estimate the asymptote of the instruction counts per \(n^3\) as \(n \to \infty\) by fitting this value on a homographic function of \(n\).

The first row of \Cref{fig:FLOPs} shows the FLOPs\(/n^3\) for computing \eqref{eq:LD1}, \eqref{eq:LD2}, and \eqref{eq:LD3}. With a slight adjustment of the scale of these curves in this figure, the experimental FLOPs are similar to the analytical functions in \Cref{fig:complexity}. We observe when \(\tens{X}\) is orthonormal, the curve corresponding to \eqref{eq:LD2} in the figure is linear in \(\rho\) while the other curves demonstrate a third-order polynomial, as expected from \Cref{table:complexity}. Also, we observe that when \(\tens{A}\) is SPD, \eqref{eq:LD1} take fewer FLOPs compared to \eqref{eq:LD2}. On the contrary, when \(\tens{A}\) is not SPD and \(\tens{X}\) is orthonormal, the advantage of \eqref{eq:LD2} over \eqref{eq:LD1} is notable as \(\rho\) increases. 

We note that while FLOPs is a useful measure for individual tasks in the algorithms (such as matrix multiplication or matrix decomposition), the overall FLOPs of the combination of tasks may not reflect the efficiency of the algorithm because tasks with similar complexity might perform differently. For instance, computing \(\tens{X}^{\ast} \tens{X}\) is a memory-bound operation and can be enhanced significantly by the vectorization on the processor if the matrices are stored in row-major ordering. But, this technique may not benefit another task with similar complexity.

As another measure of efficiency, we show the processing time of the numerical experiments in the second row of \Cref{fig:FLOPs}, which is mostly similar to the FLOPs, except when \(\tens{A}\) is SPD and \(\tens{X}\) is orthonormal. In particular, in the second column of the figure, \eqref{eq:LD2} becomes advantageous over \eqref{eq:LD1} at approximately \(p/n > 0.7\). We also observed similar results with smaller and larger \(n\). Also, on both left and right columns of the figure, \eqref{eq:LD3} is preferred at approximately \(p/n > 0.7\). Note that if \(\tens{U}_{\spc{X}^{\perp}}\) is already pre-computed, \eqref{eq:LD3} is even more efficient compared to \eqref{eq:LD1} at approximately \(p/n > 0.5\), which can be observed by the hollow circles in \Cref{fig:FLOPs}.

The above results for SPD matrices are applicable to a practical implementation of Gaussian process regression, particularly when \(p/n\) is large. Namely, using \eqref{eq:LD2} (if \(\gamma=1\)) or \eqref{eq:LD3} to compute \(\logdet(\gtens{\Sigma}, \tens{X})\) is advantageous, provided that \(\tens{X}\) is orthonormalized in advance of training the Gaussian process. Furthermore, during the optimization of the likelihood function, the design matrix \(\tens{X}\) usually remains unchanged and only the hyperparameters of \(\gtens{\Sigma}(\vect{\theta})\) vary. Thus, by pre-computing \(\tens{U}_{\spc{X}^{\perp}}\), \eqref{eq:LD3} can be far more efficient. We note that the one-time cost of the orthonormalization can be greatly outweighed by the gain achieved over iterative evaluations of log-likelihood function during the training of the Gaussian process.

% ==========
% Conclusion
% ==========

\section{Conclusion} \label{sec:conclusion}

We studied a matrix structure that arises in the Woodbury matrix identity when it becomes singular and the Woodbury identity no longer holds. Within the framework of generalized inverses, we presented identities for such matrix formulation that have direct applications to the Gaussian process regression. Namely, we showed a case of likelihood function can be recognized as a normal distribution on a subspace and we defined its precision matrix by the Bott-Duffin inverse of the covariance matrix. Also, these identities enabled us to express the precision matrix of mixed models by a power series that is useful for fast training the Gaussian process \citep{AMELI-2022-a}. We presented numerical analysis and efficient computation of the pseudo-determinant identities that are implemented in the python package \texttt{detkit} \citep{AMELI-2022-d}. Our results show that, under certain conditions, the identities presented in this work can offer an advantage in computing the log-likelihood function of a Gaussian process.

\vspace{0.3cm}
\noindent\textbf{Acknowledgments.}
We acknowledge support from the NSF, Award No.~1520825, and the AHA, Award No. 18EIA33900046. This research used the Savio computational cluster resource provided by the Berkeley Research Computing program at the University of California, Berkeley.

% ========
% Appendix
% ========

\appendix

\begin{appendices}

% ======
% Proofs
% ======

% \section{Proofs of \texorpdfstring{\Cref{sec:main}}{Section 3}} \label{sec:pf}
\section{Proofs} \label{sec:pf}

% =================
% Proofs of Special
% =================

% Equation numbering in appendices becomes A.1, A.2, etc.
\setcounter{equation}{0}
\renewcommand\theequation{A.\arabic{equation}}

\subsection{Proofs of \texorpdfstring{\Cref{sec:special}}{Section 3.1}} \label{sec:pf-special}

\begin{proof}[Proof of \Cref{lem:F}]
    The condition \ref{cond:b} is the orthogonal complement of \ref{cond:a}. Also, \ref{cond:d}\(\Leftrightarrow\)\ref{cond:e}\(\Leftrightarrow\)\ref{cond:f} can be obtained from \ref{cond:a}\(\Leftrightarrow\)\ref{cond:b}\(\Leftrightarrow\)\ref{cond:c} by swapping \(\spc{X}\) and \(\tens{X}\) with \(\spc{Y}\) and \(\tens{Y}\), replacing \(\tens{A}\) and \(\tens{F}\) with \(\tens{A}^{\ast}\) and \(\tens{F}^{\ast}\), and using \(\ker(\tens{F}^{\ast}) = \image(\tens{F})^{\perp}\) and \(\ker(\tens{Y}) = \image(\tens{Y}^{\ast})^{\perp}\). We show \ref{cond:b}\(\Leftrightarrow\)\ref{cond:c} and \ref{cond:c}\(\Leftrightarrow\)\ref{cond:f} if \(\dim(\spc{X}) = \dim(\spc{Y})\). 

    \begin{enumerate}[leftmargin=*,align=left,font=\bfseries]
        \item[\ref{cond:b}\(\Rightarrow\)\ref{cond:c}:] From the definition of \(\tens{F}\), we readily know \(\ker(\tens{X}) \subseteq \ker(\tens{F})\). We show the opposite, \ie \(\ker(\tens{F}) \subseteq \ker(\tens{X})\). Suppose \(\vect{v} \in \ker(\tens{F})\), that is \(\tens{Y}^{\ast} \tens{A}^{\dagger} \tens{X} \vect{v} = \vect{0}\). This implies
    \begin{equation}
        \tens{X} \vect{v} \in \ker(\tens{Y}^{\ast} \tens{A}^{\dagger}) = \image(\tens{A}^{\dagger \ast} \tens{Y})^{\perp} = (\tens{A}^{\dagger \ast} \spc{Y})^{\perp}. \label{eq:Xv-1}
    \end{equation}
    But \(\tens{X} \vect{v} \in \spc{X}\). This together with \eqref{eq:Xv-1} and \ref{cond:b} imply \(\tens{X} \vect{v} = \vect{0}\), meaning \(\vect{v} \in \ker(\tens{X})\). Since \(\vect{v} \in \ker(\tens{F})\) implies \(\vect{v} \in \ker(\tens{X})\), we deduce \(\ker(\tens{F}) \subseteq \ker(\tens{X})\), which concludes \ref{cond:c}.

\item[\ref{cond:c}\(\Rightarrow\)\ref{cond:b}:] Suppose in contrary that \(\mathcal{W} \coloneqq (\tens{A}^{\dagger \ast} \spc{Y})^{\perp} \cap \spc{X} \neq \{\tens{0}\}\). Choose a non-zero \(\vect{w} \in \mathcal{W}\). Since \(\vect{w} \in \spc{X}\), there exists \(\vect{v} \notin \ker(\tens{X})\) such that \(\vect{w} = \tens{X}\vect{v}\). But we also have \(\vect{w} \in (\tens{A}^{\dagger \ast} \spc{Y})^{\perp}\), so \(\tens{F} \vect{v} = \tens{Y}^{\ast} \tens{A}^{\dagger} \vect{w} = \vect{0}\), meaning \(\vect{v} \in \ker(\tens{F})\). This implies \(\ker(\tens{F}) \neq \ker(\tens{X})\) and contradicts with hypothesis \ref{cond:c}. Hence, it must be that \(\mathcal{W} = \{\vect{0}\}\) and concludes \ref{cond:b}.

\item[\ref{cond:c}\(\Leftrightarrow\)\ref{cond:f}] Using \citep[Theorem 1.1.3]{WANG-2018}, the condition \ref{cond:c} is equivalent to \(\rank(\tens{F}) = \rank(\tens{X})\), and \ref{cond:f} is equivalent to \(\rank(\tens{F}) = \rank(\tens{Y})\). Also, the hypothesis \(\dim(\spc{X}) = \dim(\spc{Y})\) can be expressed as \(\rank(\tens{X}) = \rank(\tens{Y})\), which implies \ref{cond:c}\(\Leftrightarrow\)\ref{cond:f} and vice versa. \qedhere
    \end{enumerate}
\end{proof}

\begin{lemma} \label{lem:Q12}
    Suppose \(\tens{F}\) is defined as in \Cref{lem:F}. Define
    % \begin{subequations}
    % \begin{align}
    \begin{equation}
        \tens{Q}_1 \coloneqq \tens{I} - \tens{X} \tens{F}^{\dagger} \tens{Y}^{\ast} \tens{A}^{\dagger},
        \quad \text{and} \quad
        \tens{Q}_2 \coloneqq \tens{I} - \tens{A}^{\dagger} \tens{X} \tens{F}^{\dagger} \tens{Y}^{\ast}.
        \label{eq:def-Q12}
    \end{equation}
    % \end{align}
    % \end{subequations}
    Then, \(\tens{Q}_1\) and \(\tens{Q}_2\) are projection matrices.
\end{lemma}

\begin{proof}
    Using \(\tens{F}^{\dagger} \tens{F} \tens{F}^{\dagger} = \tens{F}^{\dagger}\), we can show \(\tens{Q}_i^2 = \tens{Q}_i\), \(i=1,2\), so \(\tens{Q}_i\) are idempotent.
\end{proof}

\begin{proof}[Proof of \Cref{prop:M-proj}]
    Observe that \(\tens{M} = \tens{A}^{\dagger} \tens{Q}_1 = \tens{Q}_2 \tens{A}^{\dagger}\) where \(\tens{Q}_1\) and \(\tens{Q}_2\) are defined in \Cref{lem:Q12}. We show \(\tens{Q}_1 = \tens{P}_{(\tens{A}^{\dagger \ast} \spc{Y})^{\perp}, \spc{X}}\) if \eqref{eq:A-yx-1} holds, and \(\tens{Q}_2 = \tens{P}_{\spc{Y}^{\perp}, \tens{A}^{\dagger} \spc{X}}\) if \eqref{eq:A-yx-2} holds.

    \begin{enumerate}[leftmargin=*,align=left,label*=\emph{Step (\roman*).},ref=step (\roman*),wide]
        \item\label{step:Q1} We have \(\ker(\tens{Q}_1) = \image(\tens{I} - \tens{Q}_1)\). From the expression \(\tens{I} - \tens{Q}_1 = \tens{X} \tens{F}^{\dagger} \tens{Y}^{\ast} \tens{A}^{\dagger}\), we readily know \(\image(\tens{I} - \tens{Q}_1) \subseteq \image(\tens{X})\), which can also be written as \(\ker(\tens{Q}_1) \subseteq \spc{X}\).
        \item\label{step:Q2} We now show \(\spc{X} \subseteq \ker(\tens{Q}_1)\). To this end, we compute \(\tens{Q}_1 \tens{X}\) as follows. Using \((\tens{F}^{\dagger} \tens{F})^{\ast} = \tens{F}^{\dagger} \tens{F}\) from the fourth condition of \eqref{eq:penrose}, we have \((\tens{Q}_1 \tens{X})^{\ast} = \tens{X}^{\ast} - (\tens{F}^{\dagger} \tens{F}) \tens{X}^{\ast}\). Note that \(\tens{F}^{\dagger} \tens{F} = \tens{P}_{\image(\tens{F}^{\ast})}\) is a projection matrix \citep[Theorem 1.1.3]{WANG-2018}. Also, recall from \Cref{lem:F} that \(\ker(\tens{F}) = \ker(\tens{X})\), which implies \(\image(\tens{F}^{\ast}) = \image(\tens{X}^{\ast})\). Hence, \(\tens{P}_{\image(\tens{F}^{\ast})} \tens{X}^{\ast} = \tens{P}_{\image(\tens{X}^{\ast})} \tens{X}^{\ast} = \tens{X}^{\ast}\), yielding \(\tens{Q}_1 \tens{X} = \tens{0}\), which means \(\spc{X} \subseteq \ker(\tens{Q}_1)\). This together with the results of \ref{step:Q1} implies \(\ker(\tens{Q}_1) = \spc{X}\). 
    \end{enumerate}

    We now find \(\image(\tens{Q}_1) = \ker(\tens{Q}_1^{\ast})^{\perp}\). Similar to \ref{step:Q1} and \ref{step:Q2} in the above, we can show that \(\ker(\tens{Q}_1^{\ast}) = \image(\tens{A}^{\dagger \ast} \tens{Y}) = \tens{A}^{\dagger \ast} \spc{Y}\) (we omit its proof for brevity). This implies \(\image(\tens{Q}_1) = (\tens{A}^{\dagger \ast} \spc{Y})^{\perp}\) and concludes \eqref{eq:M-P1}.

    The image and kernel of \(\tens{Q}_2\) can be readily obtained from those of \(\tens{Q}_1\) as follows. From \(\tens{Q}_1 = \tens{P}_{(\tens{A}^{\dagger \ast} \spc{Y})^{\perp}, \spc{X}}\) we have \(\tens{Q}_1^{\ast} = \tens{P}_{\spc{X}^{\perp}, \tens{A}^{\dagger \ast} \spc{Y}}\). Observe that by swapping \(\tens{X}\) and \(\tens{Y}\) and replacing \(\tens{A}\) with \(\tens{A}^{\ast}\), we obtain \(\tens{Q}_2\) from \(\tens{Q}_1^{\ast}\). Hence, \(\tens{Q}_2 = \tens{P}_{\spc{Y}^{\perp}, \tens{A}^{\dagger} \spc{X}}\), which concludes \eqref{eq:M-P2}.
\end{proof}

\begin{lemma} \label{lem:prod-ker}
    Let \(\tens{G} \coloneqq \tens{P}_2 \tens{P}_1\) where \(\tens{P}_1\) and \(\tens{P}_2\) are projection matrices. Then,
    \begin{subequations}
        \begin{align}
            \ker(\tens{G}) &= \ker(\tens{P}_1) \oplus (\ker(\tens{P}_2) \cap \image(\tens{P}_1)), \label{eq:ker-G}\\
            \coker(\tens{G}) &= \coker(\tens{P}_2) \oplus (\coker(\tens{P}_1) \cap \coimage(\tens{P}_2)). \label{eq:coker-G}
        \end{align}
    \end{subequations}
\end{lemma}

\begin{proof}
    Let \(\vect{v} \in \ker(\tens{G}) =: \mathcal{N}\), so \(\tens{P}_2 \tens{P}_1 \vect{v} = \tens{0}\). This implies \(\tens{P}_1\vect{v} \in \ker(\tens{P}_2) \cap \image(\tens{P}_1) =: \mathcal{N}_2\), which can be written as \(\tens{P}_1 \mathcal{N} = \mathcal{N}_2\). Since \(\tens{P}_1\) is a projection along \(\ker(\tens{P}_1) =: \mathcal{N}_1\), we can construct \(\mathcal{N}\) by \(\mathcal{N}_1 \oplus \mathcal{N}_2\), which proves \eqref{eq:ker-G}. Also, \eqref{eq:coker-G} can be shown by \eqref{eq:ker-G} using \(\coker(\tens{G}) = \ker(\tens{G}^{\ast})\).
\end{proof}

\begin{proof}[Proof of \Cref{thm:M-ker-coker}]
    We first find \(\hat{\spc{X}}\). By using \(\tens{A} = \tens{M}^{(1)}\) from \Cref{rem:inv12} in the identity \(\ker(\tens{M}) = \ker(\tens{M}^{(1)} \tens{M})\) \citep[Equation 10 of Theorem 1.2.4]{WANG-2018}, we have \(\ker(\tens{M}) = \ker(\tens{A} \tens{M})\). Also, \(\tens{A} \tens{M} = \tens{A} \tens{A}^{\dagger} \tens{Q}_1\) where \(\tens{Q}_1\) is defined in \eqref{eq:def-Q12}. Let \(\spc{A} \coloneqq \image(\tens{A})\) and note that \(\tens{A} \tens{A}^{\dagger} = \tens{P}_{\spc{A}}\) is a projection matrix with the kernel \(\spc{A}^{\perp}\). From \eqref{eq:M-P1} we have
    \begin{equation}
        \tens{A} \tens{M} = \tens{P}_{\spc{A}} \tens{P}_{(\tens{A}^{\dagger \ast} \spc{Y})^{\perp}, \spc{X}}, \label{eq:comp-1}
    \end{equation}
    which is the composition of two projection matrices. From \eqref{eq:ker-G}, we have
    \begin{equation}
        \hat{\spc{X}} = \spc{X} \oplus (\spc{A}^{\perp} \cap (\tens{A}^{\dagger \ast} \spc{Y})^{\perp}).
    \end{equation}
    Observing \(\spc{A}^{\perp} = \coker(\tens{A})\) concludes \eqref{eq:hat-X}.

    Finding \(\hat{\spc{Y}}\) is similar. By using \(\tens{A} = \tens{M}^{(1)}\) in the identity \(\image(\tens{M}) = \image(\tens{M} \tens{M}^{(1)})\) \citep[Equation 10 of Theorem 1.2.4]{WANG-2018}, we have \(\image(\tens{M}) = \image(\tens{M} \tens{A})\), which implies \(\coker(\tens{M}) = \coker(\tens{M} \tens{A})\). We have \(\tens{M} \tens{A} = \tens{Q}_2 \tens{A}^{\dagger} \tens{A}\) where \(\tens{Q}_2\) is defined in \eqref{eq:def-Q12}. Let \(\spc{A}^{\ast} \coloneqq \image(\tens{A}^{\ast})\). Note that \(\tens{A}^{\dagger} \tens{A} = \tens{P}_{\spc{A}^{\ast}}\) and its cokernel is \(\spc{A}^{\ast \perp}\). From \eqref{eq:M-P2} we have
    \begin{equation}
        \tens{M} \tens{A} = \tens{P}_{\spc{Y}^{\perp}, \tens{A}^{\dagger}\spc{X}} \tens{P}_{\spc{A}^{\ast}},
    \end{equation}
    which is the composition of two projection matrices. From \eqref{eq:coker-G}, we have
    \begin{equation}
        \hat{\spc{Y}} = \spc{Y} \oplus (\spc{A}^{\ast \perp} \cap (\tens{A}^{\dagger} \spc{X})^{\perp}).
    \end{equation}
    But \(\spc{A}^{\ast \perp} = \ker(\tens{A})\), which concludes \eqref{eq:hat-Y}.
\end{proof}

\begin{proof}[Proof of \Cref{cor:XeqY}]
    If \(\tens{A}\) is non-singular, we have \(\ker(\tens{A}) = \coker(\tens{A}) = \{\vect{0}\}\), so \eqref{eq:hat-X} and \eqref{eq:hat-Y} imply \(\hat{\spc{X}} = \spc{X}\) and \(\hat{\spc{Y}} = \spc{Y}\), respectively.
\end{proof}

% \begin{lemma} \label{lem:Ainv-PD}
%     If the Hermitian matrix \(\tens{A}\) is \(\spc{X}\)-PD, then so is \(\tens{A}^{\dagger}\).
% \end{lemma}
%
% \begin{proof}
%     Let \(\tens{A} = \tens{U} \gtens{\Lambda} \tens{U}^{\ast}\) denote the spectral decomposition of the Hermitian matrix \(\tens{A} \in \mathcal{M}_{n, n}(\mathbb{C})\), where \(\tens{U}\) is unitary and \(\gtens{\Sigma} = \diag(\sigma_1, \dots, \sigma_n)\) is the diagonal matrix of the singular values \(\sigma_i\). Define \(\mathcal{U} \coloneqq \tens{U}^{\ast} \mathcal{X}\). The \(\spc{X}\)-PD property of \(\tens{A}\) implies that \(\vect{x}^{\ast} \tens{A} \vect{x} > 0\) for all \(\vect{x} \in \mathcal{X}\), or equivalently, \(\vect{u}^{\ast} \gtens{\Sigma} \vect{u} > 0\) for all \(\vect{u} \in \mathcal{U}\). The latter condition can be written as \(\sum_{i=1}^{n} \sigma_i u_i^2 > 0\) where \(u_i\) are the components of the vector \(\vect{u}\). Since the singular values are all non-negative, the inequality \(\sum_{i=1}^n \sigma_i^{-1} u_i^2 > 0\) is also valid, which implies \(\vect{u}^{\ast} \gtens{\Sigma}^{\dagger} \vect{u} > 0\), and concludes \(\tens{A}^{\dagger}\) is \(\spc{X}\)-PD.
% \end{proof}

\begin{proof}[Proof of \Cref{prop:A-XPD}]
    % By \Cref{lem:Ainv-PD}, the matrix \(\tens{A}^{\dagger}\) is \(\spc{X}\)-PD, implying that
    If \(\tens{A}^{\dagger}\) is \(\spc{X}\)-PD, it implies \(\vect{x}^{\ast} \tens{A}^{\dagger} \vect{x} > 0\) for all \(\vect{x} \in \spc{X} \setminus \{\vect{0}\}\), which means \(\tens{A}^{\dagger} \spc{X} \cap \spc{X}^{\perp} = \{\tens{0}\}\). Also, since \(\tens{A}^{\dagger} \vect{x} \neq \vect{0}\) in \(\spc{X} \setminus \{\vect{0}\}\), the map \(\tens{A}^{\dagger}\) is a bijection on \(\spc{X}\), so \(\dim(\tens{A}^{\dagger} \spc{X}) = \dim(\spc{X})\), and we conclude \(\tens{A}^{\dagger} \spc{X} \oplus \spc{X}^{\perp} = \mathbb{C}^n\). Therefore, \(\tens{A}\) satisfies \eqref{eq:A-yx-2}, and consequently \eqref{eq:A-yx-1}.
\end{proof}

\begin{proof}[Proof of \Cref{prop:M-singular}]
    The relation \(\tens{M} = \tens{A}^{\dagger}\) can be written as \(\tens{A}^{\dagger} \tens{X} \tens{F}^{\dagger} \tens{Y}^{\ast} \tens{A}^{\dagger} = \tens{0}\). If \(\tens{A}^{\dagger \ast} \spc{Y} \perp \spc{X}\) or \(\spc{Y} \perp \tens{A}^{\dagger} \spc{X}\), then \(\tens{F} = \tens{Y}^{\ast} \tens{A}^{\dagger} \tens{X} = \tens{0}\), which proves the sufficient condition. To show the necessary condition, suppose \(\tens{A}^{\dagger} \tens{X} \tens{F}^{\dagger} \tens{Y}^{\ast} \tens{A}^{\dagger} = \tens{0}\). One possibility to achieve this is if \(\image(\tens{Y}^{\ast} \tens{A}^{\dagger}) \subseteq \ker(\tens{F}^{\dagger})\), or \(\image(\tens{F}^{\dagger}) \subseteq \ker(\tens{A}^{\dagger} \tens{X})\). However,
    \begin{align*}
        &\ker(\tens{F}^{\dagger}) = \ker(\tens{F}^{\ast}) \perp \image(\tens{F}) \subseteq \image(\tens{Y}^{\ast} \tens{A}^{\dagger}), \\
        &\image(\tens{F}^{\dagger}) = \image(\tens{F}^{\ast}) \perp \ker(\tens{F}) \supseteq \ker(\tens{A}^{\dagger} \tens{X}).
    \end{align*}
    Hence, \(\image(\tens{Y}^{\ast} \tens{A}^{\dagger}) \not\subseteq \ker(\tens{F}^{\dagger})\) and \(\image(\tens{F}^{\dagger}) \not\subseteq \ker(\tens{A}^{\dagger} \tens{X})\), which rules out the above possibilities. The remaining feasible case is if \(\tens{A}^{\dagger} \tens{X}\), or \(\tens{F}\), or \(\tens{Y}^{\ast} \tens{A}^{\dagger}\) is null, implying \(\tens{A}^{\dagger \ast} \spc{Y} \perp \spc{X}\) and \(\spc{Y} \perp \tens{A}^{\dagger} \spc{X}\).
\end{proof}

\begin{proof}[Proof of \Cref{prop:XY-share}]
    From the proof of \Cref{prop:M-singular} we know \(\tens{F} \neq \tens{0}\) if \(\tens{A}^{\dagger \ast} \spc{Y} \not\perp \spc{X}\). We calculate \(\tens{M} \tens{X} = \tens{A}^{\dagger} \tens{X} (\tens{I} - \tens{F}^{\dagger} \tens{F})\). If \(\tens{F}\) has full column-rank, then \(\tens{F}^{\dagger} \tens{F} = \tens{I}\), so \(\tens{M} \tens{X} = \tens{0}\), which implies \(\spc{X} \subseteq \hat{\spc{X}}\). If \(\tens{F}\) does not have full column-rank, then \(\tens{F}^{\dagger} \tens{F}\) is a projection matrix, so \(\image(\tens{I} - \tens{F}^{\dagger} \tens{F}) = \ker(\tens{F}^{\dagger} \tens{F}) \subsetneq \mathbb{C}^p\). Also, \(\ker(\tens{X}) \subseteq \ker(\tens{F}^{\dagger} \tens{F})\) by the definition of \(\tens{F}\). Thus, \(\image(\tens{X}(\tens{I} - \tens{F}^{\dagger} \tens{F})) \subsetneq \image(\tens{X})\), namely \(\rank(\tens{X}(\tens{I} - \tens{F}^{\dagger} \tens{F})) < \rank(\tens{X})\). This together with the rank inequality for matrix product, \(\rank(\tens{M} \tens{X}) \leq \rank(\tens{X} (\tens{I} - \tens{F}^{\dagger} \tens{F}))\), implies \(\rank(\tens{M} \tens{X}) < \rank(\tens{X})\), so \(\dim(\tens{M} \spc{X}) < \dim(\spc{X})\). This means \(\spc{X}\) must intersect with \(\ker(\tens{M})\), \ie \(\hat{\spc{X}} \cap \spc{X} \neq \{\vect{0}\}\). Proving \(\hat{\spc{Y}} \cap \spc{Y} \neq \{\vect{0}\}\) is similar by calculating \(\tens{M}^{\ast} \tens{Y}\).
\end{proof}

% ==================
% Proofs of Identity
% ==================

\subsection{Proofs of \texorpdfstring{\Cref{sec:identity}}{Section 3.1}} \label{sec:pf-identity}

\begin{proof}[Proof of \Cref{prop:MA}]
    Recall that \(\tens{M} = \tens{A}^{\dagger} \tens{Q}_1 = \tens{Q}_2 \tens{A}^{\dagger}\) where \(\tens{Q}_1\) and \(\tens{Q}_2\) are defined in \eqref{eq:def-Q12}. By using \(\tens{A}^{\dagger} \tens{A} \tens{A}^{\dagger} = \tens{A}^{\dagger}\), we have
    \begin{equation*}
        \tens{M} \tens{A} \tens{M} = (\tens{Q}_2 \tens{A}^{\dagger}) \tens{A} (\tens{A}^{\dagger} \tens{Q}_1) = \tens{Q}_2 \tens{A}^{\dagger} \tens{Q}_1.
    \end{equation*}
    But \(\tens{Q}_2 \tens{A}^{\dagger} = \tens{A}^{\dagger} \tens{Q}_1\), and \(\tens{Q}_1^2 = \tens{Q}_1\) by \Cref{lem:Q12}. Hence, \(\tens{M} \tens{A} \tens{M} = \tens{A}^{\dagger} \tens{Q}_1 = \tens{M}\), so \(\tens{M}\) satisfies the second condition of \eqref{eq:penrose} and concludes \eqref{eq:MA2}. Also, from the definition of \(\tens{M}\), we have \(\image(\tens{M}) \subseteq \image(\tens{A}^{\dagger})\), so \(\dim(\hat{\spc{Y}}^{\perp}) \leq \rank(\tens{A}^{\dagger})\), which can be written as \(\dim(\hat{\spc{Y}}) + \rank(\tens{A}) \geq m\).
\end{proof}

\begin{proof}[Proof of \Cref{cor:pinv}]
    The relation \eqref{eq:M-pinv} can be recognized by applying \Cref{prop:MA} to the known identity \(\tens{A}^{(2)}_{\hat{\spc{Y}}^{\perp}, \hat{\spc{X}}} = (\tens{P}_{\hat{\spc{X}}^{\perp}} \tens{A} \tens{P}_{\hat{\spc{Y}}^{\perp}})^{\dagger}\) \cite[p. 80]{ISRAEL-2003}. Also, \eqref{eq:Pxp} can be justified as follows. Substitute \eqref{eq:M} in \eqref{eq:M-P1} and set \(\tens{A} = \tens{I}\) and \(\tens{Y} = \tens{X}\). The resulted identity can also be written for \(\hat{\tens{X}}\) and \(\hat{\spc{X}}\) instead of \(\tens{X}\) and \(\spc{X}\) which concludes \eqref{eq:Pxp}. Showing \eqref{eq:Pyp} is similar.
\end{proof}

\begin{proof}[Proof of \Cref{prop:ind-1}]
    If \(\tens{A}\) is square, so is \(\tens{M}\), and \(\dim(\hat{\spc{Y}}) = \dim(\hat{\spc{X}})\) by the rank-nullity theorem. This together with \(\hat{\spc{Y}}^{\perp} \cap \hat{\spc{X}} = \{\vect{0}\}\) is equivalent to \(\hat{\spc{Y}}^{\perp} \oplus \hat{\spc{X}} = \mathbb{C}^n\). Also, \(\ind(\tens{M}) \leq 1\) implies \(\rank(\tens{M}) = \rank(\tens{M}^2)\) and vice versa. Therefore, we need to show \(\rank(\tens{M}) = \rank(\tens{M}^2)\) if and only if \(\hat{\spc{Y}}^{\perp} \cap \hat{\spc{X}} = \{\vect{0}\}\). Define \(\spc{V} \coloneqq \hat{\spc{Y}}^{\perp} \cap \hat{\spc{X}}\), and \(\spc{V}' \coloneqq \hat{\spc{Y}}^{\perp} / \spc{V}\) where \(\hat{\spc{Y}}^{\perp} = \spc{V} \oplus \spc{V}'\). Note that \(\tens{M} \spc{V} = \{\vect{0}\}\) since \(\spc{V} \subseteq \ker(\tens{M})\). Also, \(\dim(\tens{M} \spc{V}') = \dim(\spc{V}')\) since \(\spc{V}' \not\subset \ker(\tens{M})\). Therefore,
    \begin{equation}
        \rank(\tens{M}^2) = \dim(\tens{M} \hat{\spc{Y}}^{\perp}) = \dim(\tens{M} \spc{V} \oplus \tens{M} \spc{V}') = \dim(\spc{V}'). \label{eq:rank-M2}
    \end{equation}
    To show the sufficient condition, let \(\spc{V} = \{\vect{0}\}\), so \(\hat{\spc{Y}}^{\perp} = \spc{V}'\). Hence, from \eqref{eq:rank-M2} we have \(\rank(\tens{M}^2) = \dim(\hat{\spc{Y}}^{\perp}) = \rank(\tens{M})\). To show the necessary condition, let \(\rank(\tens{M}) = \rank(\tens{M}^2)\), but in contrary, suppose \(\spc{V} \neq \{\vect{0}\}\), so \(\spc{V}' \subsetneq \hat{\spc{Y}}^{\perp}\). Thus, \eqref{eq:rank-M2} implies \(\rank(\tens{M}^2) = \dim(\spc{V}') < \dim(\hat{\spc{Y}}^{\perp}) = \rank(\tens{M})\), which is a contradiction. Hence, it must be that \(\spc{V} = \{\vect{0}\}\).
\end{proof}

\begin{proposition} \label{prop:N-invertible}
    Suppose \(\tens{A} \in \mathcal{M}_{n,n}(\mathbb{C})\), and let the spaces \(\hat{\spc{X}}, \hat{\spc{Y}}\) satisfy \(\hat{\spc{Y}}^{\perp} \oplus \hat{\spc{X}} = \mathbb{C}^n\). If
    \begin{equation}
        \tens{A} \hat{\spc{Y}}^{\perp} \cap \hat{\spc{X}} = \{\vect{0}\}, \label{eq:A2-ypx}
    \end{equation}
    then \(\tens{P}_{\hat{\spc{X}}, \hat{\spc{Y}}^{\perp}} + \tens{A} \tens{P}_{\hat{\spc{Y}}^{\perp}, \hat{\spc{X}}}\) is non-singular.
\end{proposition}

\begin{proof}
    Let \(\mathcal{N} \coloneqq \ker(\tens{P}_{\hat{\spc{X}}, \hat{\spc{Y}}^{\perp}} + \tens{A} \tens{P}_{\hat{\spc{Y}}^{\perp}, \hat{\spc{X}}})\). Suppose \(\vect{v} \in \mathcal{N}\), meaning
    \begin{equation}
        \tens{P}_{\hat{\spc{X}}, \hat{\spc{Y}}^{\perp}} \vect{v} + \tens{A} \tens{P}_{\hat{\spc{Y}}^{\perp}, \hat{\spc{X}}} \vect{v} = \vect{0}.
    \end{equation}
    Observe \(\tens{P}_{\hat{\spc{X}}, \hat{\spc{Y}}^{\perp}} \vect{v} \in \hat{\spc{X}}\) and \(\tens{A} \tens{P}_{\hat{\spc{Y}}^{\perp}, \hat{\spc{X}}} \vect{v} \in \tens{A} \hat{\spc{Y}}^{\perp}\). However, because of \eqref{eq:A2-ypx}, it must be that
    \begin{equation}
        \tens{P}_{\hat{\spc{X}}, \hat{\spc{Y}}^{\perp}} \vect{v} = \vect{0},
        \quad \text{and} \quad
        \tens{A} \tens{P}_{\hat{\spc{Y}}^{\perp}, \hat{\spc{X}}} \vect{v} = \vect{0}.
        \sublabel{eq:Nv}{a,b}
    \end{equation}
    The relation \subeqref{eq:Nv}{a} means \(\vect{v} \in \ker(\tens{P}_{\hat{\spc{X}}, \hat{\spc{Y}}^{\perp}}) = \hat{\spc{Y}}^{\perp}\), which implies \(\mathcal{N} \subseteq \hat{\spc{Y}}^{\perp}\) and \(\tens{P}_{\hat{\spc{Y}}^{\perp}, \hat{\spc{X}}} \vect{v} = \vect{v}\). So \subeqref{eq:Nv}{b} yields \(\tens{A} \vect{v} = \vect{0}\), meaning \(\vect{v} \in \ker(\tens{A})\), which implies \(\mathcal{N} \subseteq \ker(\tens{A})\). We concluded
    \begin{equation}
        \mathcal{N} \subseteq \hat{\spc{Y}}^{\perp} \cap \ker(\tens{A}). \label{eq:N-Y-A}
    \end{equation}
    On the other hand, recall from \Cref{lem:Q12} that \(\tens{M} = \tens{A}^{\dagger} \tens{Q}_1\), so
    \begin{equation}
        \hat{\spc{Y}}^{\perp} = \image(\tens{M}) \subseteq \image(\tens{A}^{\dagger}) = \image(\tens{A}^{\ast}) \perp \ker(\tens{A}). \label{eq:im-T2}
    \end{equation}
    Hence \(\hat{\spc{Y}}^{\perp} \cap \ker(\tens{A}) = \{\vect{0}\}\). So from \eqref{eq:N-Y-A}, it must be that \(\mathcal{N} = \{\vect{0}\}\), which concludes the proof.
\end{proof}

\begin{proof}[Proof of \Cref{thm:M-rep}]
    Recall from \Cref{rem:A-ypx} that \(\hat{\spc{X}}\) and \(\hat{\spc{Y}}\) satisfy \eqref{eq:A-ypx}. Thus, \eqref{eq:A2-ypx} holds and \(\tens{N}\) is non-singular by \Cref{prop:N-invertible}. We can right-multiply \eqref{eq:M-rep} by \(\tens{N}\) to instead prove
    \begin{equation}
        \tens{T} \coloneqq \tens{M} (\tens{P}_{\hat{\spc{X}}, \hat{\spc{Y}}^{\perp}} + \tens{A} \tens{P}_{\hat{\spc{Y}}^{\perp}, \hat{\spc{X}}}) - \tens{P}_{\hat{\spc{Y}}^{\perp}, \hat{\spc{X}}} = \tens{0}. \label{eq:def-T}
    \end{equation}
    \begin{enumerate}[leftmargin=*,align=left,label*=\emph{Step (\roman*).},ref=step (\roman*),wide]
        \item We first show \(\tens{A} \tens{T} = \tens{0}\). Using the identity \(\tens{A} \tens{A}^{(2)}_{\hat{\spc{Y}}^{\perp}, \hat{\spc{X}}} = \tens{P}_{\tens{A} \hat{\spc{Y}}^{\perp}, \hat{\spc{X}}}\) \citep[p. 33, Exercise 9, Part 1]{WANG-2018} and applying \(\tens{A}^{(2)}_{\hat{\spc{Y}}^{\perp}, \hat{\spc{X}}} = \tens{M}\) from \Cref{prop:MA}, we obtain
        \begin{align*}
            \tens{A} \tens{T} &= \tens{P}_{\tens{A} \hat{\spc{Y}}^{\perp}, \hat{\spc{X}}} \left( \tens{P}_{\hat{\spc{X}}, \hat{\spc{Y}}^{\perp}} + \tens{A} \tens{P}_{\hat{\spc{Y}}^{\perp}, \hat{\spc{X}}} \right) - \tens{A} \tens{P}_{\hat{\spc{Y}}^{\perp}, \hat{\spc{X}}} \\
            &= \tens{P}_{\tens{A} \hat{\spc{Y}}^{\perp}, \hat{\spc{X}}} \tens{A} \tens{P}_{\hat{\spc{Y}}^{\perp}, \hat{\spc{X}}} - \tens{A} \tens{P}_{\hat{\spc{Y}}^{\perp}, \hat{\spc{X}}} \\
            &= \left( \tens{I} - \tens{P}_{\hat{\spc{X}}, \tens{A} \hat{\spc{Y}}^{\perp}} \right) \tens{A} \tens{P}_{\hat{\spc{Y}}^{\perp}, \hat{\spc{X}}} - \tens{A} \tens{P}_{\hat{\spc{Y}}^{\perp}, \hat{\spc{X}}} \\
            &= \tens{A} \tens{P}_{\hat{\spc{Y}}^{\perp}, \hat{\spc{X}}} - \tens{A} \tens{P}_{\hat{\spc{Y}}^{\perp}, \hat{\spc{X}}} = \tens{0}.
        \end{align*}
        In the above, we used \(\tens{P}_{\tens{A} \hat{\spc{Y}}^{\perp}, \hat{\spc{X}}} \tens{P}_{\hat{\spc{X}}, \hat{\spc{Y}}^{\perp}} = \tens{0}\), also \(\tens{P}_{\hat{\spc{X}}, \tens{A} \hat{\spc{Y}}^{\perp}} \tens{A} \tens{P}_{\hat{\spc{Y}}^{\perp}, \hat{\spc{X}}} = \tens{0}\) since \(\image(\tens{A} \tens{P}_{\hat{\spc{Y}}^{\perp}, \hat{\spc{X}}}) = \tens{A} \hat{\spc{Y}}^{\perp}\).

    \item We now show \(\tens{A} \tens{T} = \tens{0}\) implies \(\tens{T} = \tens{0}\). By the definition of \(\tens{T}\) in \eqref{eq:def-T}, observe that
        \begin{equation}
            \image(\tens{T}) \subseteq \image(\tens{M}) + \image(\tens{P}_{\hat{\spc{Y}}^{\perp}, \hat{\spc{X}}}) = \hat{\spc{Y}}^{\perp}. \label{eq:im-T}
        \end{equation}
        But \eqref{eq:im-T} and \eqref{eq:im-T2} imply \(\image(\tens{T}) \perp \ker(\tens{A})\), hence, the only way \(\tens{A} \tens{T}\) can be null is if either \(\tens{T} = \tens{0}\) or \(\tens{A} = \tens{0}\). The latter case also leads to \(\tens{T} = \tens{0}\) because \(\tens{A} = \tens{0}\) implies \(\tens{M} = \tens{0}\) and by \eqref{eq:A-ypx} we must have \(\hat{\spc{X}} = \mathbb{C}^n\) so \(\tens{P}_{\hat{\spc{Y}}^{\perp}, \hat{\spc{X}}} = \tens{0}\) and by \eqref{eq:def-T}, \(\tens{T}\) becomes null.
    \end{enumerate}
    Lastly, we justify \eqref{eq:P_xy} as follows. Set \(\tens{A} = \tens{I}\), so \eqref{eq:M-rep} becomes \(\tens{M} = \tens{P}_{\hat{\spc{Y}}^{\perp}, \hat{\spc{X}}}\). This together with the expression for \(\tens{M}\) in \eqref{eq:M}, when \(\tens{A} = \tens{I}\), yields \eqref{eq:P_xy}.
\end{proof}

\begin{proof}[Proof of \Cref{cor:comp}]

    By left and right multiplication of \(\tens{P}_{\hat{\spc{X}}^{\perp}}\) to \(\tens{M} \tens{N} = \tens{P}_{\hat{\spc{Y}}^{\perp}, \hat{\spc{X}}}\) in \eqref{eq:M-rep} and using the identities \(\tens{M} = \tens{M} \tens{P}_{\hat{\spc{X}}^{\perp}}\) and \(\tens{P}_{\hat{\spc{X}}^{\perp}} \tens{P}_{\hat{\spc{Y}}^{\perp}, \hat{\spc{X}}} = \tens{P}_{\hat{\spc{X}}^{\perp}}\), we obtain
    \begin{equation}
        \tens{P}_{\hat{\spc{X}}^{\perp}} \tens{M} \tens{P}_{\hat{\spc{X}}^{\perp}} \tens{N} \tens{P}_{\hat{\spc{X}}^{\perp}} = \tens{P}_{\hat{\spc{X}}^{\perp}}. \label{eq:PxMN}
    \end{equation}
    Set \(\tens{P}_{\hat{\spc{X}}^{\perp}} = \tens{U}_{\hat{\spc{X}}^{\perp}} \tens{U}_{\hat{\spc{X}}^{\perp}}^{\ast}\). Left and right multiplication of \eqref{eq:PxMN} respectively by \(\tens{U}_{\hat{\spc{X}}^{\perp}}^{\ast}\) and \(\tens{U}_{\hat{\spc{X}}^{\perp}}\) yields \(\tens{M}_{\hat{\spc{X}}^{\perp}} \tens{N}_{\hat{\spc{X}}^{\perp}} = \tens{I}\). Since \(\tens{N}\) is non-singular by \Cref{thm:M-rep}, so is \(\tens{N}_{\hat{\spc{X}}^{\perp}}\). Also, \(\tens{M}_{\hat{\spc{X}}^{\perp}}\) is non-singular since \(\image(\tens{U}_{\hat{\spc{X}}^{\perp}}) \not\subseteq \ker(\tens{M})\) and from \Cref{prop:ind-1} we know \(\image(\tens{M}) = \hat{\spc{Y}}^{\perp} \not\subseteq \hat{\spc{X}} = \ker(\tens{U}_{\hat{\spc{X}}^{\perp}}^{\ast})\).
\end{proof}

% ============================
% Proofs of Pseudo-Determinant
% ============================

\subsection{Proofs of \texorpdfstring{\Cref{sec:det}}{section 3.3}} \label{sec:pf-det}

\begin{proof}[Proof of \Cref{lem:XYAplus}]
    Recall from \Cref{lem:Q12} that \(\tens{M} = \tens{Q}_2 \tens{A}^{\dagger}\), so \(\ker(\tens{A}^{\dagger}) \subseteq \ker(\tens{M}) = \hat{\spc{X}}\). The orthogonal complement of this relation and using \(\ker(\tens{A}^{\dagger}) = \ker(\tens{A}^{\ast}) = \spc{A}^{\perp}\) yields \(\hat{\spc{X}}^{\perp} \subseteq \spc{A}\). Also, recall from \eqref{eq:im-T2} that \(\hat{\spc{Y}}^{\perp} \subseteq \image(\tens{A}^{\ast})\). Since \(\tens{A}\) is EP, \(\image(\tens{A}^{\ast}) = \spc{A}\), which concludes \(\hat{\spc{Y}}^{\perp} \subseteq \spc{A}\). From the orthogonal complement of both these relations, we obtain \(\spc{A}^{\perp} \subset \hat{\spc{X}} \cap \hat{\spc{Y}}\).

    To show \subeqref{eq:XAplus}{a}, firstly, observe that
    \begin{equation}
        \hat{\spc{X}}_{\spc{A}} \cap \spc{A}^{\perp} = (\hat{\spc{X}} \cap \spc{A}) \cap \spc{A}^{\perp} = \{\vect{0}\}. \label{eq:XA-share}
    \end{equation}
    Secondly, suppose \(\hat{\vect{x}} \in \hat{\spc{X}} \subseteq \mathbb{C}^n\). Since \(\hat{\vect{x}} \in \mathbb{C}^n = \spc{A} \oplus \spc{A}^{\perp}\), there exists \(\vect{a} \in \spc{A}\) and \(\vect{a}^{\perp} \in \spc{A}^{\perp}\) so that \(\hat{\vect{x}} = \vect{a} + \vect{a}^{\perp}\). Recall that earlier we concluded \(\hat{\spc{X}}^{\perp} \subseteq \spc{A}\), which also means \(\spc{A}^{\perp} \subseteq \hat{\spc{X}}\). Thus, \(\vect{a}^{\perp} \in \hat{\spc{X}}\), so \(\vect{a} = \hat{\vect{x}} - \vect{a}^{\perp} \in \hat{\spc{X}} \cap \spc{A} = \hat{\spc{X}}_{\spc{A}}\). Thus, for every \(\hat{\vect{x}} \in \hat{\spc{X}}\), there exists \(\vect{a} \in \hat{\spc{X}}_{\spc{A}}\) and \(\vect{a}^{\perp} \in \spc{A}^{\perp}\) so that \(\hat{\vect{x}} = \vect{a} + \vect{a}^{\perp}\). This result together with \eqref{eq:XA-share} concludes \subeqref{eq:XAplus}{a}. Proving \subeqref{eq:YAplus}{a} is similar. Also, since from \eqref{eq:dim-xy} we have \(\dim(\hat{\spc{X}}) = \dim(\hat{\spc{Y}})\), comparing \subeqref{eq:XAplus}{a} and \subeqref{eq:YAplus}{a} yields \(\dim(\hat{\spc{X}}_{\spc{A}}) = \dim(\hat{\spc{Y}}_{\spc{A}})\).

    Note that \subeqref{eq:XAplus}{a} is a consequence of \(\hat{\spc{X}}^{\perp} \subseteq \spc{A}\). By swapping \(\hat{\spc{X}}\) and \(\spc{A}\), the latter relation remains valid, which accordingly yields \subeqref{eq:XAplus}{b} from \subeqref{eq:XAplus}{a}. Proving \subeqref{eq:YAplus}{b} from \subeqref{eq:YAplus}{a} is similar.
\end{proof}

Let \(\tens{U}_{\hat{\spc{X}}_{\spc{A}}}\), \(\tens{U}_{\hat{\spc{Y}}_{\spc{A}}}\), \(\tens{U}_{\hat{\spc{X}}^{\perp}}\), \(\tens{U}_{\hat{\spc{Y}}^{\perp}}\), and \(\tens{U}_{\spc{A}^{\perp}}\) be as in \Cref{prop:pMp}. For the following proofs, define
\begin{equation}
    \tens{U}_{\hat{\spc{X}}} \coloneqq [\tens{U}_{\hat{\spc{X}}_{\spc{A}}}, \tens{U}_{\spc{A}^{\perp}}],
    \quad \text{and} \quad
    \tens{U}_{\hat{\spc{Y}}} \coloneqq [\tens{U}_{\hat{\spc{Y}}_{\spc{A}}}, \tens{U}_{\spc{A}^{\perp}}],
    \sublabel{eq:UxUyUa}{a,b}
\end{equation}
which their columns form bases for \(\hat{\spc{X}}\) and \(\hat{\spc{Y}}\) due to \subeqref{eq:XAplus}{a} and \subeqref{eq:YAplus}{a}, respectively. Also, define
\begin{equation}
    \tens{U}_{\spc{A}}' \coloneqq [\tens{U}_{\hat{\spc{X}}_{\spc{A}}}, \tens{U}_{\hat{\spc{X}}^{\perp}}],
    \quad \text{and} \quad
    \tens{U}_{\spc{A}}'' \coloneqq [\tens{U}_{\hat{\spc{Y}}_{\spc{A}}}, \tens{U}_{\hat{\spc{Y}}^{\perp}}],
    \sublabel{eq:UAp12}{a,b}
\end{equation}
which their columns form bases for \(\spc{A}\) due to \subeqref{eq:XAplus}{b} and \subeqref{eq:YAplus}{b}, respectively. Accordingly, define
\begin{equation}
    \tens{U}' \coloneqq [\tens{U}_{\hat{\spc{X}}}, \tens{U}_{\hat{\spc{X}}^{\perp}}],
    \quad \text{and} \quad
    \tens{U}'' \coloneqq [\tens{U}_{\hat{\spc{Y}}}, \tens{U}_{\hat{\spc{Y}}^{\perp}}],
    \sublabel{eq:Up12}{a,b}
\end{equation}
which their columns form bases for \(\mathbb{C}^{n}\).

\begin{lemma} \label{lem:UAU}
    The matrix \(\tens{U}_{\spc{A}}'^{\ast} \tens{U}_{\spc{A}}''\) is non-singular. Furthermore, if \(\tens{A}\) is EP, then, \(\tens{U}_{\spc{A}}'^{\ast} \tens{A} \tens{U}_{\spc{A}}''\) is non-singular and 
    \begin{equation}
        \vert \tens{U}_{\spc{A}}'^{\ast} \tens{A} \tens{U}_{\spc{A}}'' \vert = 
        \vert \tens{U}_{\spc{A}}'^{\ast} \tens{U}_{\spc{A}}'' \vert \, \vert \tens{A} \vert_{\dagger}. \label{eq:UAU}
    \end{equation}
\end{lemma}

\begin{proof}
    Since \(\image(\tens{U}_{\spc{A}}'') = \spc{A}\) and \(\ker(\tens{U}_{\spc{A}}'^{\ast}) = \spc{A}^{\perp}\), we have \(\image(\tens{U}_{\spc{A}}'') \not\subseteq \ker(\tens{U}_{\spc{A}}'^{\ast})\). So, the kernel of \(\tens{U}_{\spc{A}}'^{\ast} \tens{U}_{\spc{A}}''\) is the kernel of \(\tens{U}_{\spc{A}}''\), but \(\tens{U}_{\spc{A}}''\) is full column-rank. Hence, the square matrix \(\tens{U}_{\spc{A}}'^{\ast} \tens{U}_{\spc{A}}''\) is non-singular. Also, if \(\tens{A}\) is EP, then \(\ker(\tens{A}) = \spc{A}^{\perp}\), so \(\image(\tens{U}_{\spc{A}}'') \not\subseteq \ker(\tens{A})\). Furthermore, \(\image(\tens{A}) \not\subseteq \ker(\tens{U}_{\spc{A}}'^{\ast})\), so, the kernel of \(\tens{U}_{\spc{A}}'^{\ast} \tens{A} \tens{U}_{\spc{A}}''\) is the kernel of \(\tens{U}_{\spc{A}}''\). Thus, the square matrix \(\tens{U}_{\spc{A}}'^{\ast} \tens{A} \tens{U}_{\spc{A}}''\) is non-singular.

    Define \(\tilde{\tens{U}}_{\spc{A}}' \coloneqq \tens{U}_{\spc{A}}^{\ast} \tens{U}_{\spc{A}}'\) and \(\tilde{\tens{U}}_{\spc{A}}'' \coloneqq \tens{U}_{\spc{A}}^{\ast} \tens{U}_{\spc{A}}''\), which respectively are the coordinates of \(\tens{U}_{\spc{A}}'\) and \(\tens{U}_{\spc{A}}''\)  on the basis of the columns of \(\tens{U}_{\spc{A}}\) defined in \eqref{eq:A-tilde}. Also, if \(\tens{A}\) is EP, recall from the nilpotent decomposition \eqref{eq:A-tilde} that \(\tens{A} = \tens{U}_{\spc{A}} \tilde{\tens{A}} \tens{U}_{\spc{A}}^{\ast}\) where \(\tilde{\tens{A}}\) is non-singular. We have
    \begin{equation}
        \tens{U}_{\spc{A}}'^{\ast} \tens{A} \tens{U}_{\spc{A}}'' = \tilde{\tens{U}}_{\spc{A}}'^{\ast} \tilde{\tens{A}} \tilde{\tens{U}}_{\spc{A}}''.
    \end{equation}
    Note that \(\tilde{\tens{U}}_{\spc{A}}', \tilde{\tens{U}}_{\spc{A}}'' \in \mathcal{M}_{r, r}(\mathbb{C})\) with \(r = \dim(\spc{A})\) are bases in \(\spc{A}\) and are non-singular. Thus,
    \begin{equation}
        \vert \tens{U}_{\spc{A}}'^{\ast} \tens{A} \tens{U}_{\spc{A}}'' \vert = 
        \vert \tilde{\tens{U}}_{\spc{A}}'^{\ast} \tilde{\tens{A}} \tilde{\tens{U}}_{\spc{A}}'' \vert =
        \vert \tilde{\tens{U}}_{\spc{A}}'^{\ast} \vert \, \vert \tilde{\tens{A}} \vert \, \vert \tilde{\tens{U}}_{\spc{A}}'' \vert =
        \vert \tilde{\tens{U}}_{\spc{A}}'^{\ast}  \tilde{\tens{U}}_{\spc{A}}'' \vert \, \vert \tilde{\tens{A}} \vert =
        \vert \tens{U}_{\spc{A}}'^{\ast}  \tens{U}_{\spc{A}}'' \vert \, \vert \tilde{\tens{A}} \vert. \label{eq:UAAUA}
    \end{equation}
    Also, \(\tens{A}\) and \(\diag(\tilde{\tens{A}}, \tens{0})\) in \eqref{eq:A-tilde} are unitarily similar, thus \(\vert \tilde{\tens{A}} \vert = \vert \tens{A} \vert_{\dagger}\), which concludes \eqref{eq:UAU}.
\end{proof}

\begin{lemma} \label{lem:pUxAUy}
    If \(\tens{A}\) is square, then \(\tens{U}_{\hat{\spc{X}}_{\perp}}^{\ast} \tens{A} \tens{U}_{\hat{\spc{Y}}^{\perp}}\) is non-singular. Furthermore, if \(\tens{A}\) is EP, then \(\tens{U}_{\hat{\spc{Y}}_{\spc{A}}}^{\ast} \tens{A}^{\dagger} \tens{U}_{\hat{\spc{X}}_{\spc{A}}}\) is non-singular and
        \begin{equation}
            \vert \tens{U}_{\hat{\spc{X}}^{\perp}}^{\ast} \tens{A} \tens{U}_{\hat{\spc{Y}}^{\perp}} \vert =
            \vert \tens{U}_{\hat{\spc{Y}}_{\spc{A}}}^{\ast} \tens{A}^{\dagger} \tens{U}_{\hat{\spc{X}}_{\spc{A}}} \vert \,
            \vert \tens{U}_{\spc{A}}'^{\ast} \tens{U}_{\spc{A}}'' \vert \,
            \vert \tens{A} \vert_{\dagger}.
            \label{eq:UAU2}
        \end{equation}
\end{lemma}

\begin{proof}
    Let \(\tens{G} \coloneqq \tens{U}_{\hat{\spc{X}}^{\perp}}^{\ast} \tens{A} \tens{U}_{\hat{\spc{Y}}^{\perp}}\). Recall from \eqref{eq:A-ypx} that \(\tens{A} \hat{\spc{Y}}^{\perp} \oplus \hat{\spc{X}} = \mathbb{C}^n\). Similar to the equivalency of \ref{cond:c} and \ref{cond:d} in \Cref{lem:F}, we can accordingly say the conditions \(\tens{A} \hat{\spc{Y}}^{\perp} + \hat{\spc{X}} = \mathbb{C}^n\) and \(\ker(\tens{G}) = \ker(\tens{U}_{\hat{\spc{X}}^{\perp}})\) are equivalent. But because \(\tens{U}_{\hat{\spc{X}}^{\perp}}\) is full-rank, so is \(\tens{G}\). Note \(\tens{G}\) is also square since by \eqref{eq:dim-xy} we have \(\dim(\hat{\spc{X}}^{\perp}) = \dim(\hat{\spc{Y}}^{\perp})\). Thus, \(\tens{G}\) is non-singular.

    We now show \eqref{eq:UAU2}. Recall from \subeqref{eq:UxUyUa}{b} and \subeqref{eq:Up12}{b} that \(\tens{U}'' = [\tens{U}_{\hat{\spc{Y}}_{\spc{A}}}, \tens{U}_{\spc{A}^{\perp}}, \tens{U}_{\hat{\spc{Y}}^{\perp}}]\) is unitary, \ie
    \begin{equation}
        \tens{U}_{\hat{\spc{Y}}_{\spc{A}}} \tens{U}_{\hat{\spc{Y}}_{\spc{A}}}^{\ast} +
        \tens{U}_{\hat{\spc{Y}}^{\perp}} \tens{U}_{\hat{\spc{Y}}^{\perp}}^{\ast}
        = \tens{I} -
        \tens{U}_{\spc{A}^{\perp}} \tens{U}_{\spc{A}^{\perp}}^{\ast}. \label{eq:UpU}
    \end{equation}
    Also, consider \(\tens{P}_{\spc{A}} = \tens{A} \tens{A}^{\dagger}\) and \(\tens{A} \tens{U}_{\spc{A}^{\perp}} = \tens{0}\) since \(\spc{A}^{\perp} = \ker(\tens{A})\) as \(\tens{A}\) is EP. From \eqref{eq:UpU} we obtain
    \begin{equation}
        \tens{A} \left(\tens{U}_{\hat{\spc{Y}}_{\spc{A}}} \tens{U}_{\hat{\spc{Y}}_{\spc{A}}}^{\ast} +
        \tens{U}_{\hat{\spc{Y}}^{\perp}} \tens{U}_{\hat{\spc{Y}}^{\perp}}^{\ast} \right) \tens{A}^{\dagger}
        = \tens{P}_{\spc{A}}. \label{eq:UpU2}
    \end{equation}
    Moreover, consider \(\tens{P}_{\spc{A}} \tens{U}_{\hat{\spc{X}}_{\spc{A}}} = \tens{U}_{\hat{\spc{X}}_{\spc{A}}}\) since \(\image(\tens{U}_{\hat{\spc{X}}_{\spc{A}}}) = \hat{\spc{X}}_{\spc{A}} \subseteq \spc{A}\) by \eqref{eq:def-XA}. So \eqref{eq:UpU2} implies
    \begin{subequations}
    \begin{align}
        &\tens{U}_{\hat{\spc{X}}_{\spc{A}}}^{\ast}
        \tens{A} \left(\tens{U}_{\hat{\spc{Y}}_{\spc{A}}} \tens{U}_{\hat{\spc{Y}}_{\spc{A}}}^{\ast} +
        \tens{U}_{\hat{\spc{Y}}^{\perp}} \tens{U}_{\hat{\spc{Y}}^{\perp}}^{\ast} \right) \tens{A}^{\dagger}
        \tens{U}_{\hat{\spc{X}}_{\spc{A}}} = \tens{I}, \label{eq:UYU1}\\
        &\tens{U}_{\hat{\spc{X}}^{\perp}}^{\ast}
        \tens{A} \left(\tens{U}_{\hat{\spc{Y}}_{\spc{A}}} \tens{U}_{\hat{\spc{Y}}_{\spc{A}}}^{\ast} +
        \tens{U}_{\hat{\spc{Y}}^{\perp}} \tens{U}_{\hat{\spc{Y}}^{\perp}}^{\ast} \right) \tens{A}^{\dagger}
        \tens{U}_{\hat{\spc{X}}_{\spc{A}}} = \tens{0}, \label{eq:UYU2}
    \end{align}
    \end{subequations}
    because \(\tens{U}_{\hat{\spc{X}}_{\spc{A}}}^{\ast} \tens{U}_{\hat{\spc{X}}_{\sp{A}}} = \tens{I}\) and \(\tens{U}_{\hat{\spc{X}}^{\perp}}^{\ast} \tens{U}_{\hat{\spc{X}}_{\spc{A}}} = \tens{0}\). On the other hand, observe from \subeqref{eq:UAp12}{a} and \subeqref{eq:UAp12}{b} that
    \begin{equation}
        \tens{U}_{\spc{A}}'^{\ast} \tens{A} \tens{U}_{\spc{A}}'' =
        \begin{bmatrix}
            \tens{U}_{\hat{\spc{X}}_{\spc{A}}}^{\ast} \tens{A} \tens{U}_{\hat{\spc{Y}}_{\spc{A}}} &
            \tens{U}_{\hat{\spc{X}}_{\spc{A}}}^{\ast} \tens{A} \tens{U}_{\hat{\spc{Y}}^{\perp}} \\
            \tens{U}_{\hat{\spc{X}}^{\perp}}^{\ast} \tens{A} \tens{U}_{\hat{\spc{Y}}_{\spc{A}}} &
            \tens{U}_{\hat{\spc{X}}^{\perp}}^{\ast} \tens{A} \tens{U}_{\hat{\spc{Y}}^{\perp}}
        \end{bmatrix}. \label{eq:UAU3}
    \end{equation}
    Based on \eqref{eq:UYU1}, \eqref{eq:UYU2}, and \eqref{eq:UAU3}, we calculate
    \begin{equation}
        \left(\tens{U}_{\spc{A}}'^{\ast} \tens{A} \tens{U}_{\spc{A}}'' \right) \,
        \begin{bmatrix}
            \tens{U}_{\hat{\spc{Y}}_{\spc{A}}}^{\ast} \tens{A}^{\dagger} \tens{U}_{\hat{\spc{X}}_{\spc{A}}} &
            \tens{0} \\
            \tens{U}_{\hat{\spc{Y}}^{\perp}}^{\ast} \tens{A}^{\dagger} \tens{U}_{\hat{\spc{X}}_{\spc{A}}} &
            \tens{I}
        \end{bmatrix}
        =
        \begin{bmatrix}
            \tens{I} &
            \tens{U}_{\hat{\spc{X}}_{\spc{A}}}^{\ast} \tens{A} \tens{U}_{\hat{\spc{Y}}^{\perp}} \\
            \tens{0} &
            \tens{U}_{\hat{\spc{X}}^{\perp}}^{\ast} \tens{A} \tens{U}_{\hat{\spc{Y}}^{\perp}} \\
        \end{bmatrix}. \label{eq:UAU4}
    \end{equation}
    Recall from \Cref{lem:UAU} that \(\tens{U}_{\spc{A}}'^{\ast} \tens{A} \tens{U}_{\spc{A}}''\) is non-singular if \(\tens{A}\) is EP. We also found that \(\tens{U}_{\hat{\spc{X}}^{\perp}}^{\ast} \tens{A} \tens{U}_{\hat{\spc{Y}}^{\perp}}\) is non-singular, so the right-hand side of \eqref{eq:UAU4}, the block lower-triangular matrix on the left-hand side of \eqref{eq:UAU4}, and accordingly, \(\tens{U}_{\hat{\spc{Y}}_{\spc{A}}}^{\ast} \tens{A}^{\dagger} \tens{U}_{\hat{\spc{X}}_{\spc{A}}}\) are non-singular. Also, the determinant of \eqref{eq:UAU4} yields
    \begin{equation}
        \vert \tens{U}_{\spc{A}}'^{\ast} \tens{A} \tens{U}_{\spc{A}}'' \vert \,
        \vert \tens{U}_{\hat{\spc{Y}}_{\spc{A}}}^{\ast} \tens{A}^{\dagger} \tens{U}_{\hat{\spc{X}}_{\spc{A}}} \vert =
        \vert \tens{U}_{\hat{\spc{X}}^{\perp}}^{\ast} \tens{A} \tens{U}_{\hat{\spc{Y}}^{\perp}} \vert. \label{eq:det-UAU4}
    \end{equation}
    Substituting \eqref{eq:UAU} in \eqref{eq:det-UAU4} concludes \eqref{eq:UAU2}.
\end{proof}

\begin{lemma} \label{lem:pUxUy}
    Suppose \(\tens{A}\) is square and \(\ind(\tens{M}) = 1\). Then, \(\tens{U}_{\hat{\spc{X}}^{\perp}}^{\ast} \tens{U}_{\hat{\spc{Y}}^{\perp}}\) and \(\tens{U}_{\hat{\spc{Y}}_{\spc{A}}}^{\ast} \tens{U}_{\hat{\spc{X}}_{\spc{A}}}\) are non-singular and
    \begin{equation}
       \vert \tens{U}_{\hat{\spc{X}}^{\perp}}^{\ast} \tens{U}_{\hat{\spc{Y}}^{\perp}} \vert =
       \vert \tens{U}_{\hat{\spc{Y}}_{\spc{A}}}^{\ast} \tens{U}_{\hat{\spc{X}}_{\spc{A}}} \vert \,
       \vert \tens{U}_{\spc{A}}'^{\ast} \tens{U}_{\spc{A}}'' \vert. \label{eq:UU2}
    \end{equation}
\end{lemma}

\begin{proof}
    We can obtain \eqref{eq:UU2} by replacing \(\tens{A}\) in \eqref{eq:UAU2} with \(\tens{P}_{\spc{A}}\), since it has all properties of \(\tens{A}\) required in \eqref{eq:UAU2}. Namely, \(\image(\tens{P}_{\spc{A}}) = \image(\tens{A})\), and \(\ker(\tens{P}_{\spc{A}}) = \ker(\tens{A}) = \spc{A}^{\perp}\) since \(\tens{A}\) is EP. Note the eigenvalues of a projection matrix are \(0\) or \(1\), so \(\vert \tens{P}_{\spc{A}} \vert_{\dagger} = 1\). Also, \(\tens{P}_{\spc{A}}^{\dagger} \tens{U}_{\hat{\spc{X}}_{\spc{A}}} = \tens{U}_{\hat{\spc{X}}_{\spc{A}}}\) since \(\tens{P}_{\spc{A}}^{\dagger} = \tens{P}_{\spc{A}}\) and \(\hat{\spc{X}}_{\spc{A}} \subseteq \spc{A}\). Thus, \Cref{lem:pUxAUy} implies \eqref{eq:UU2} and the non-singularity of its terms.
\end{proof}

\begin{proof}[Proof of \Cref{prop:pMp}]
    By using \(\tens{U}'\) from \subeqref{eq:Up12}{a}, we observe that \(\tens{U}'^{\ast} \tens{M} \tens{U}'\) and \(\tens{M}\) are unitarily similar, so \(\vert \tens{U}'^{\ast} \tens{M} \tens{U}' \vert_{\dagger} = \vert \tens{M} \vert_{\dagger}\). With \(\tens{U}' = [\tens{U}_{\hat{\spc{X}}}, \tens{U}_{\hat{\spc{X}}^{\perp}}]\) and \(\tens{M} \tens{U}_{\hat{\spc{X}}} = \tens{0}\) (since \(\hat{\spc{X}} = \ker(\tens{M})\)), we calculate
    \begin{equation}
        \tens{U}'^{\ast} \tens{M} \tens{U}' =
        \begin{bmatrix}
            \tens{0} & \tens{U}_{\hat{\spc{X}}}^{\ast} \tens{M} \tens{U}_{\hat{\spc{X}}^{\perp}} \\
            \tens{0} & \tens{U}_{\hat{\spc{X}}^{\perp}}^{\ast} \tens{M} \tens{U}_{\hat{\spc{X}}^{\perp}}
        \end{bmatrix}.
    \end{equation}
    The eigenvalues of the above block upper-triangular matrix are those of its diagonal blocks \citep[p. 62]{HORN-1990}. Also, from \Cref{cor:comp} we know \(\tens{U}_{\hat{\spc{X}}^{\perp}}^{\ast} \tens{M} \tens{U}_{\hat{\spc{X}}^{\perp}}\) is non-singular. Hence,
    \begin{equation}
        \vert \tens{M} \vert_{\dagger} = \vert \tens{U}_{\hat{\spc{X}}^{\perp}}^{\ast} \tens{M} \tens{U}_{\hat{\spc{X}}^{\perp}} \vert. \label{eq:det-UxMUx}
    \end{equation}
    Moreover, from \Cref{cor:pinv}, we have \(\tens{M} = (\tens{P}_{\hat{\spc{X}}^{\perp}} \tens{A} \tens{P}_{\hat{\spc{Y}}^{\perp}})^{\dagger}\). Set \(\tens{P}_{\hat{\spc{X}}^{\perp}} = \tens{U}_{\hat{\spc{X}}^{\perp}} \tens{U}_{\hat{\spc{X}}^{\perp}}^{\ast}\) and \(\tens{P}_{\hat{\spc{Y}}^{\perp}} = \tens{U}_{\hat{\spc{Y}}^{\perp}} \tens{U}_{\hat{\spc{Y}}^{\perp}}^{\ast}\). Also, recall from \Cref{lem:pUxAUy} that \(\tens{U}_{\hat{\spc{X}}^{\perp}}^{\ast} \tens{A} \tens{U}_{\hat{\spc{Y}}^{\perp}}\) is non-singular. All together, from \eqref{eq:det-UxMUx} we have
    \begin{subequations}
    \begin{align}
        \vert \tens{M} \vert_{\dagger} &= \vert \tens{U}_{\hat{\spc{X}}^{\perp}}^{\ast} \big( \tens{U}_{\hat{\spc{X}}^{\perp}} \tens{U}_{\hat{\spc{X}}^{\perp}}^{\ast} \tens{A} \tens{U}_{\hat{\spc{Y}}^{\perp}} \tens{U}_{\hat{\spc{Y}}^{\perp}}^{\ast} \big)^{\dagger} \tens{U}_{\hat{\spc{X}}^{\perp}} \vert \label{eq:pinv-out0} \\
                                       &= \vert \tens{U}_{\hat{\spc{X}}^{\perp}}^{\ast} \tens{U}_{\hat{\spc{Y}}^{\perp}} \big( \tens{U}_{\hat{\spc{X}}^{\perp}}^{\ast} \tens{A} \tens{U}_{\hat{\spc{Y}}^{\perp}} \big)^{-1} \tens{U}_{\hat{\spc{X}}^{\perp}}^{\ast} \tens{U}_{\hat{\spc{X}}^{\perp}} \vert \label{eq:pinv-out} \\
                                       &= \vert \tens{U}_{\hat{\spc{X}}^{\perp}}^{\ast} \tens{U}_{\hat{\spc{Y}}^{\perp}} \vert \, \vert \tens{U}_{\hat{\spc{X}}^{\perp}}^{\ast} \tens{A} \tens{U}_{\hat{\spc{Y}}^{\perp}} \vert^{-1}, \label{eq:pinv-out2}
    \end{align}
    \end{subequations}
    which concludes the first equality of \eqref{eq:pinv-M-p}. Note that to obtain \eqref{eq:pinv-out}, the reverse-order law for the product of matrices under pseudo-inverse is applied to \(\tens{U}_{\hat{\spc{X}}^{\perp}}\) and \(\tens{U}_{\hat{\spc{Y}}^{\perp}}^{\ast}\) in \eqref{eq:pinv-out0} because they have orthonormal columns and rows, respectively (see \eg \citep[Corollary 1.4.3]{CAMPBELL-2009} or \citep[Equations \(3'\), \(4'\)]{BOULDIN-1973}). Also, to obtain \eqref{eq:pinv-out2}, we used the result of \Cref{lem:pUxUy} that \(\tens{U}_{\hat{\spc{X}}^{\perp}}^{\ast} \tens{U}_{\hat{\spc{Y}}^{\perp}}\) is non-singular. Finally, by combining \Cref{lem:pUxAUy} and \Cref{lem:pUxUy} with the first equality of \eqref{eq:pinv-M-p}, its second equality is obtained.
\end{proof}

\begin{proof}[Proof of \Cref{thm:pdet}]
    We simplify the second equality of \eqref{eq:pinv-M-p} as follows.

    \begin{enumerate}[leftmargin=*,align=left,label*=\emph{Step (\roman*).},ref=step (\roman*),wide]
        \item\label{step:lemM1} From \subeqref{eq:XAplus}{a} we have \(\image(\tens{P}_{\spc{A}} \hat{\tens{X}}) = \tens{P}_{\spc{A}} (\hat{\spc{X}}_{\spc{A}} \oplus \spc{A}^{\perp}) = \hat{\spc{X}}_{\spc{A}}\). Similarly, from \subeqref{eq:YAplus}{a} we can show \(\image(\tens{P}_{\spc{A}} \hat{\tens{Y}}) = \hat{\spc{Y}}_{\spc{A}}\). So, without loss of generality, let \(\tens{U}_{\hat{\spc{X}}_{\spc{A}}}\) and \(\tens{U}_{\hat{\spc{Y}}_{\spc{A}}}\) respectively be the matrices of the left singular vectors of \(\tens{P}_{\spc{A}} \hat{\tens{X}}\) and \(\tens{P}_{\spc{A}} \hat{\tens{Y}}\) given by the singular value decompositions
    \begin{equation}
        \tens{P}_{\spc{A}} \hat{\tens{X}} = \tens{U}_{\hat{\spc{X}}_{\spc{A}}} \gtens{\Sigma}_{\hat{\spc{X}}_{\spc{A}}} \tens{V}_{\hat{\spc{X}}_{\spc{A}}}^{\ast},
        \quad \text{and} \quad
        \tens{P}_{\spc{A}} \hat{\tens{Y}} = \tens{U}_{\hat{\spc{Y}}_{\spc{A}}} \gtens{\Sigma}_{\hat{\spc{Y}}_{\spc{A}}} \tens{V}_{\hat{\spc{Y}}_{\spc{A}}}^{\ast}, \label{eq:PXY-svd}
    \end{equation}
    where \(\gtens{\Sigma}_{\hat{\spc{X}}_{\spc{A}}}\) and \(\gtens{\Sigma}_{\hat{\spc{Y}}_{\spc{A}}}\) are diagonal and non-singular, and \(\tens{V}_{\hat{\spc{X}}_{\spc{A}}}\) and \(\tens{V}_{\hat{\spc{Y}}_{\spc{A}}}\) each have orthonormal columns. Also, \(\tens{P}_{\spc{A}}\) is Hermitian, so \(\tens{P}_{\spc{A}}^{\ast} \tens{P}_{\spc{A}} = \tens{P}_{\spc{A}}\). Furthermore, since \(\tens{A}\) is EP, we have \(\image(\tens{A}^{\dagger}) = \image(\tens{A}^{\ast}) = \spc{A}\), hence, \(\tens{P}_{\spc{A}}^{\ast} \tens{A}^{\dagger} \tens{P}_{\spc{A}} = \tens{A}^{\dagger}\). Based on these, the second equality of \eqref{eq:pinv-M-p} becomes
    \begin{equation}
        \vert \tens{M} \vert_{\dagger} =
        \frac{\vert \tens{V}_{\hat{\spc{Y}}_{\spc{A}}}^{\ast} \hat{\tens{Y}}^{\ast} \tens{P}_{\spc{A}}^{\ast} \tens{P}_{\spc{A}} \hat{\tens{X}} \tens{V}_{\hat{\spc{X}}_{\spc{A}}} \vert}
        {\vert \tens{A} \vert_{\dagger} \, 
        \vert \tens{V}_{\hat{\spc{Y}}_{\spc{A}}}^{\ast} \hat{\tens{Y}}^{\ast} \tens{P}_{\spc{A}}^{\ast} \tens{A}^{\dagger} \tens{P}_{\spc{A}} \hat{\tens{X}} \tens{V}_{\hat{\spc{X}}_{\spc{A}}} \vert} =
        \frac{\vert \tens{V}_{\hat{\spc{Y}}_{\spc{A}}}^{\ast} \tens{G} \tens{V}_{\hat{\spc{X}}_{\spc{A}}} \vert}
        {\vert \tens{A} \vert_{\dagger} \, 
        \vert \tens{V}_{\hat{\spc{Y}}_{\spc{A}}}^{\ast} \tens{H} \tens{V}_{\hat{\spc{X}}_{\spc{A}}} \vert}, \label{eq:UKU}
    \end{equation}
    where \(\tens{G} \coloneqq \hat{\tens{Y}}^{\ast} \tens{P}_{\spc{A}} \hat{\tens{X}}\) and \(\tens{H} \coloneqq \hat{\tens{Y}}^{\ast} \tens{A}^{\dagger} \hat{\tens{X}}\).

\item\label{step:lemM2} From \eqref{eq:PXY-svd}, we have \(\ker(\tens{P}_{\spc{A}} \hat{\tens{X}}) = \ker(\tens{V}_{\hat{\spc{X}}_{\spc{A}}}^{\ast})\) and \((\ker(\tens{P}_{\spc{A}} \hat{\tens{Y}}))^{\perp} = (\ker(\tens{V}_{\hat{\spc{Y}}_{\spc{A}}}^{\ast}))^{\perp} = \image(\tens{V}_{\hat{\spc{Y}}_{\spc{A}}})\). Thus, \eqref{eq:ker-PXY} becomes \(\ker(\tens{V}_{\hat{\spc{X}}_{\spc{A}}}^{\ast}) \cap \image(\tens{V}_{\hat{\spc{Y}}_{\spc{A}}}) = \{\vect{0}\}\), which implies \(\tens{V}_{\hat{\spc{X}}_{\spc{A}}}^{\ast} \tens{V}_{\hat{\spc{Y}}_{\spc{A}}}\) is full-rank. But from \Cref{lem:XYAplus}, we know \(\dim(\hat{\spc{X}}_{\spc{A}}) = \dim(\hat{\spc{Y}}_{\spc{A}})\), so \(\tens{V}_{\hat{\spc{X}}_{\spc{A}}}^{\ast} \tens{V}_{\hat{\spc{Y}}_{\spc{A}}}\) is square, hence non-singular.

    \item\label{step:lemM3} We calculate \(\vert \tens{V}_{\hat{\spc{Y}}_{\spc{A}}}^{\ast} \tens{G} \tens{V}_{\hat{\spc{X}}_{\spc{A}}} \vert\) in \eqref{eq:UKU}. Define the unitary matrix \(\tens{V} \coloneqq [\tens{V}_{\hat{\spc{X}}_{\spc{A}}}, \tens{V}_{\hat{\spc{X}}_{\spc{A}}}']\) where \(\tens{V}_{\hat{\spc{X}}_{\spc{A}}}'\) is the orthonormal complement of \(\tens{V}_{\hat{\spc{X}}_{\spc{A}}}\) in \(\mathbb{C}^n\). The matrices \(\tens{G}\) and \(\tens{V}^{\ast} \tens{G} \tens{V}\) are unitarily similar, so \(\vert \tens{G} \vert_{\dagger} = \vert \tens{V}^{\ast} \tens{G} \tens{V} \vert_{\dagger}\). We calculate
    \begin{equation}
        \tens{V}^{\ast} \tens{G} \tens{V} =
        \begin{bmatrix}
            \tens{V}_{\hat{\spc{X}}_{\spc{A}}}^{\ast} \tens{G} \tens{V}_{\hat{\spc{X}}_{\spc{A}}} & \tens{0} \\
            \tens{V}_{\hat{\spc{X}}_{\spc{A}}}'^{\ast} \tens{G} \tens{V}_{\hat{\spc{X}}_{\spc{A}}} & \tens{0}
        \end{bmatrix}. \label{eq:VGV}
    \end{equation}
    In the above, \(\tens{G} \tens{V}_{\hat{\spc{X}}_{\spc{A}}}' = \tens{0}\) since \(\tens{V}_{\hat{\spc{X}}_{\spc{A}}}^{\ast} \tens{V}_{\hat{\spc{X}}_{\spc{A}}}' = \tens{0}\). From \eqref{eq:VGV} we obtain \(\vert \tens{G} \vert_{\dagger} = \vert \tens{V}_{\hat{\spc{X}}_{\spc{A}}}^{\ast} \tens{G} \tens{V}_{\hat{\spc{X}}_{\spc{A}}} \vert_{\dagger}\). Moreover, from \eqref{eq:PXY-svd} and the definition of \(\tens{G}\), we have \(\tens{V}_{\hat{\spc{X}}_{\spc{A}}}^{\ast} \tens{G} \tens{V}_{\hat{\spc{X}}_{\spc{A}}} = (\tens{V}_{\hat{\spc{X}}_{\spc{A}}}^{\ast} \tens{V}_{\hat{\spc{Y}}_{\spc{A}}}) (\tens{V}_{\hat{\spc{Y}}_{\spc{A}}}^{\ast} \tens{G} \tens{V}_{\hat{\spc{X}}_{\spc{A}}})\). Also, \eqref{eq:PXY-svd} implies \(\tens{V}_{\hat{\spc{Y}}_{\spc{A}}}^{\ast} \tens{G} \tens{V}_{\hat{\spc{X}}_{\spc{A}}} = \gtens{\Sigma}_{\hat{\spc{Y}}_{\spc{A}}} (\tens{U}_{\hat{\spc{Y}}_{\spc{A}}}^{\ast} \tens{U}_{\hat{\spc{X}}_{\spc{A}}}) \gtens{\Sigma}_{\hat{\spc{X}}_{\spc{A}}}\). But from \Cref{lem:pUxUy}, we know \(\tens{U}_{\hat{\spc{Y}}_{\spc{A}}}^{\ast} \tens{U}_{\hat{\spc{X}}_{\spc{A}}}\) is non-singular, so \(\tens{V}_{\hat{\spc{Y}}_{\spc{A}}}^{\ast} \tens{G} \tens{V}_{\hat{\spc{X}}_{\spc{A}}}\) is non-singular. Also, in \ref{step:lemM2}, we showed \(\tens{V}_{\hat{\spc{X}}_{\spc{A}}}^{\ast} \tens{V}_{\hat{\spc{Y}}_{\spc{A}}}\) is non-singular, so \(\tens{V}_{\hat{\spc{X}}_{\spc{A}}}^{\ast} \tens{G} \tens{V}_{\hat{\spc{X}}_{\spc{A}}}\) is non-singular. All together, we obtain
    \begin{equation}
        \vert \tens{V}_{\hat{\spc{Y}}_{\spc{A}}}^{\ast} \tens{G} \tens{V}_{\hat{\spc{X}}_{\spc{A}}} \vert =
        \vert \tens{G} \vert_{\dagger} \, \vert \tens{V}_{\hat{\spc{X}}_{\spc{A}}}^{\ast} \tens{V}_{\hat{\spc{Y}}_{\spc{A}}} \vert^{-1}.
        \label{eq:pG2}
    \end{equation}

    \item\label{step:lemM4} By repeating \ref{step:lemM3} but for \(\tens{H}\) (instead of \(\tens{G}\)), we can similarly show
    \begin{equation}
        \vert \tens{V}_{\hat{\spc{Y}}_{\spc{A}}}^{\ast} \tens{H} \tens{V}_{\hat{\spc{X}}_{\spc{A}}} \vert =
        \vert \tens{H} \vert_{\dagger} \, \vert \tens{V}_{\hat{\spc{X}}_{\spc{A}}}^{\ast} \tens{V}_{\hat{\spc{Y}}_{\spc{A}}} \vert^{-1}.
        \label{eq:pH2}
    \end{equation}
    Substituting \eqref{eq:pG2} and \eqref{eq:pH2} in \eqref{eq:UKU} concludes \eqref{eq:pdet}. \qedhere
    \end{enumerate}
\end{proof}

\begin{proof}[Proof of \Cref{prop:pdet-M-N}]
    By using \(\tens{U}'\) from \subeqref{eq:Up12}{a}, we observe that \(\tens{N}\) and \(\tens{U}'^{\ast} \tens{N} \tens{U}'\) are unitarily similar, so \(\vert \tens{N} \vert = \vert \tens{U}'^{\ast} \tens{N} \tens{U}' \vert\). Also, from \subeqref{eq:N-rest}{a} we have \(\tens{N} \tens{U}_{\hat{\spc{X}}} = \tens{U}_{\hat{\spc{X}}}\). We calculate
    \begin{equation*}
        \tens{U}'^{\ast} \tens{N} \tens{U}' =
        \begin{bmatrix}
            \tens{I} & \tens{U}_{\hat{\spc{X}}}^{\ast} \tens{N} \tens{U}_{\hat{\spc{X}}^{\perp}} \\
            \tens{0} & \tens{U}_{\hat{\spc{X}}^{\perp}}^{\ast} \tens{N} \tens{U}_{\hat{\spc{X}}^{\perp}}
        \end{bmatrix}. \label{eq:N-basis}
    \end{equation*}
    The determinant of the above relation yields \(\vert \tens{N} \vert = \vert \tens{U}_{\hat{\spc{X}}^{\perp}}^{\ast} \tens{N} \tens{U}_{\hat{\spc{X}}^{\perp}} \vert\). By substituting this result together with \eqref{eq:det-UxMUx} in the determinant of \eqref{eq:MNI}, we conclude \eqref{eq:det-N}.
\end{proof}

\begin{proof}[Proof of \Cref{cor:det-N}]
    Combining \eqref{eq:pdet} and \eqref{eq:det-N} and omitting the hat notations yields \eqref{eq:det-MN}. %Note the conditions \eqref{eq:detN-cond1}, \eqref{eq:detN-cond2}, \eqref{eq:detN-cond3}, and \eqref{eq:detN-cond4} for \(\tens{X}\) and \(\tens{Y}\) are the restatements of \eqref{eq:A-ypx}, \Cref{prop:ind-1}, \Cref{lem:XYAplus}, and \eqref{eq:ker-PXY} for \(\hat{\tens{{X}}}\) and \(\hat{\tens{Y}}\), respectively.
\end{proof}

% ===========================
% Gaussian Process Regression
% ===========================

% Equation numbering in appendices becomes B.1, B.2, etc.
\setcounter{equation}{0}
\renewcommand\theequation{B.\arabic{equation}}

\section{Gaussian Process Regression} \label{sec:gpr}

We briefly describe Gaussian process regression in the following, and we refer the interested reader to \cite{NEAL-1998, MACKAY-1998, MINKA-1998, SEEGER-2004, RASMUSSEN-2006} for further details.

Consider the standard regression model \(f(\vect{x}) = \mu(\vect{x}) + \epsilon(\vect{x})\) where \(\vect{x} \in \mathcal{D}\), consisting of the deterministic mean function \(\mu\) and the zero-mean stochastic function \(\epsilon\). A common form of the mean function is the linear model \(\mu(\vect{x}) = \vect{\phi}^{\ast}(\vect{x}) \vect{\beta}\), where \(\vect{\phi}: \mathcal{D} \to \mathbb{R}^p\) is the array of \(p\) basis functions, and \(\vect{\beta} \in \mathbb{R}^p\) are the unknown regression coefficients to be found. The function \(\epsilon\) represents the uncertainty due to either the regression residual or data and is characterized by the covariance function \(\Sigma(\vect{\theta}): \mathcal{D} \times \mathcal{D} \to \mathbb{R}\) where \(\vect{\theta}\) is the array of hyperparameters.

Suppose the array of data \(\vect{y}\) with the components \(y_i \coloneqq f(\vect{x}_i)\) are known on a set of training points \(\vect{x}_i\), \(i = 1, \cdots, n\). We also discretize \(\vect{\phi}\) and \(\Sigma\) on the training points as follows. Define the full-rank matrix \(\tens{X} \in \mathcal{M}_{n, p}(\mathbb{R})\) by the components \(X_{ij} \coloneqq \phi_{j}(\vect{x}_i)\), which is known as the design matrix. The covariance matrix \(\gtens{\Sigma} \in \mathcal{M}_{n, n}(\mathbb{R})\) is defined by the components \(\Sigma_{ij} = \Sigma(\vect{x}_i, \vect{x}_j | \vect{\theta})\).

A Gaussian process prior on \(f\), denoted by \(f \sim \mathcal{GP}(\mu, \Sigma)\), imposes that the joint distribution of \(f\) on any finite set points is normal. On the training point, this implies \(\vect{y} \sim \mathcal{N}(\tens{X} \vect{\beta}, \gtens{\Sigma})\), \ie the likelihood function of the data is the normal distribution
\begin{equation}
    p(\vect{y}| \vect{\beta}, \vect{\theta}) = \frac{1}{\sqrt{(2\pi)^n}} | \gtens{\Sigma} |^{-\sfrac{1}{2}} \exp\left(-\frac{1}{2} \| \vect{y} - \tens{X} \vect{\beta} \|_{\gtens{\Sigma}^{-1}}^2 \right), \label{eq:likelihood}
\end{equation}
where \(\|\vect{y} - \tens{X} \vect{\beta} \|_{\gtens{\Sigma}^{-1}}^2 = (\vect{y} - \tens{X} \vect{\beta})^{\ast} \gtens{\Sigma}^{-1} (\vect{y} - \tens{X} \vect{\beta})\) is the Mahalanobis distance of the data from its mean with respect to the norm induced by \(\gtens{\Sigma}^{-1}\) as a metric tensor, which we assumed to be non-singular (for singular covariance matrix, see \eg \citep{HENKDON-1985} and \citep{HOLBROOK-2018}).

To provide probabilistic predictions on test points, a Gaussian process is trained on the existing data. Training the Gaussian process means to find the parameter \(\vect{\beta}\) and hyperparameters \(\vect{\theta}\), for instance, by maximizing the posterior function of the hyperparameters \(p(\vect{\beta}, \vect{\theta} | \vect{y}) \propto p(\vect{y} | \vect{\beta}, \vect{\theta}) p(\vect{\beta}, \vect{\theta})\). We assume \(p(\vect{\beta},\vect{\theta}) = p(\vect{\beta}) p(\vect{\theta})\). For simplicity, we also assume \(p(\vect{\theta})\) is the improper uniform distribution and we eliminate it from the posterior. A frequently-used prior for \(\vect{\beta}\) is the normal distribution \(\vect{\beta} \sim \mathcal{N}(\vect{b}, \tens{B})\). It is a common practice to marginalize \(\vect{\beta}\) out of the posterior, leading to the marginal posterior \(p(\vect{\theta} | \vect{y})\), or equivalently, the marginal likelihood \(p(\vect{y} | \vect{\theta})\) given by (see \eg \cite[Equation 2.43]{RASMUSSEN-2006})
\begin{equation}
    p(\vect{y} | \vect{\theta}) = \frac{1}{\sqrt{(2\pi)^n}} | \gtens{\Sigma} |^{-\sfrac{1}{2}} |\tens{X}^{\ast} \gtens{\Sigma}^{-1} \tens{X} |^{-\sfrac{1}{2}} | \tens{B} |^{-\sfrac{1}{2}} \exp\left(-\frac{1}{2} \| \vect{y} - \tens{X} \vect{b} \|_{\tens{M}}^2 \right), \label{eq:like-B}
\end{equation}
where
\begin{equation}
    \tens{M} \coloneqq \gtens{\Sigma}^{-1} - \gtens{\Sigma}^{-1} \tens{X} \left(\tens{X}^{\ast} \gtens{\Sigma}^{-1} \tens{X} + \tens{B}^{-1} \right)^{-1} \tens{X}^{\ast} \gtens{\Sigma}^{-1}. \label{eq:M-like-B}
\end{equation}
By applying the Woodbury matrix identity of \eqref{eq:N-with-B} and \eqref{eq:M-with-B} on \(\tens{M}\), we can write \(\tens{M} = \tilde{\gtens{\Sigma}}^{-1}\), where
\begin{equation}
    \tilde{\gtens{\Sigma}} \coloneqq \gtens{\Sigma} + \tens{X} \tens{B} \tens{X}^{\ast}. \label{eq:eqv-cov}
\end{equation}
Moreover, the matrix determinant lemma in \eqref{eq:det-lemma} simplifies \eqref{eq:like-B} to
\begin{equation}
    p(\vect{y} | \vect{\theta}) = \frac{1}{\sqrt{(2\pi)^n}} | \tilde{\gtens{\Sigma}} |^{-\sfrac{1}{2}} \exp\left(-\frac{1}{2} \| \vect{y} - \tens{X} \vect{b} \|_{\tilde{\tens{\Sigma}}^{-1}}^2 \right). \label{eq:like-B-2}
\end{equation}
In other words, when a normal prior is imposed on \(\vect{\beta}\), the marginal likelihood becomes the normal distribution \(\mathcal{N}(\tens{X}\vect{b}, \tilde{\gtens{\Sigma}})\) (see \eg \cite[Equation 2.40]{RASMUSSEN-2006}). We note that \(\tilde{\tens{\Sigma}}\) acts as the equivalent covariance in the presence of the uncertainty \(\tens{B}\) of the parameter \(\vect{\beta}\). The matrix \(\tens{M}\) can be regarded as the precision matrix of \(\tilde{\gtens{\Sigma}}\).

A special case of the above formulations is when the precision matrix of \(\vect{\beta} \sim \mathcal{N}(\vect{b}, \tens{B})\) vanishes, \ie \(\tens{B}^{-1} \to \tens{0}\), which leads to an improper uniform prior on the parameter \(\vect{\beta}\). In this case, \(\tens{M}\) in \eqref{eq:M-like-B} becomes singular, and the Woodbury matrix identity and matrix determinant lemma that are applied in \eqref{eq:eqv-cov} and \eqref{eq:like-B-2} do not hold. We discuss this case in \Cref{sec:app-like}.

% =======
% Dataset
% =======

% Equation numbering in appendices becomes B.1, B.2, etc.
\setcounter{equation}{0}
\renewcommand\theequation{B.\arabic{equation}}

\section{Dataset} \label{sec:dataset}

In our numerical experiment in \Cref{sec:num}, the matrix \(\tens{A}\) is obtained from the covariance of an electrocardiogram (ECG) signal. This signal was taken from MIT-BIH arrhythmia database \citep{MOODY-2001} and is available at PhysioBank \citep{GOLDBERGER-2000}. A short segment of the post-processed ECG signal is displayed in \Cref{fig:ecg-a}. For ease of calculation, the ECG signal is considered wide-sense stationary stochastic process, allowing us to calculate its autocovariance by
\begin{equation*}
    \kappa(\Delta t) = \mathbb{E}[(f(t+\Delta t) - \bar{f})(f(t) - \bar{f})],
\end{equation*}
where \(\Delta t\) is the lag-time in the autocovariance function, \(\mathbb{E}\) is the expectation operator, and \(f\) is the ECG signal with the mean \(\bar{f}\). The covariance matrix \(\tens{A}\) is obtained by the components \(A_{ij} = \kappa(\vert i - j \vert f_s \nu)\) where \(f_s = 360\) Hz is the sampling frequency of the ECG signal and \(\nu = 2\) is the sampling of the autocovariance function. Note that \(\tens{A}\) is a Toeplitz matrix.

\Cref{fig:ecg-b,fig:ecg-c} respectively show the corresponding autocorrelation function \(\tau = \sigma^{-2} \kappa\) and the correlation matrix \(\tens{K} = \sigma^{-2} \tens{A}\) where \(\sigma^2 = \kappa(0)\) is the variance of the signal. Also, \Cref{fig:ecg-d} shows the eigenvalues of the correlation matrix which indicates that \(\tens{K}\) (and hence \(\tens{A}\)) is positive-definite as all the eigenvalues are positive.

\begin{figure}[t!]
    \centering
    \includegraphics[width=\textwidth]{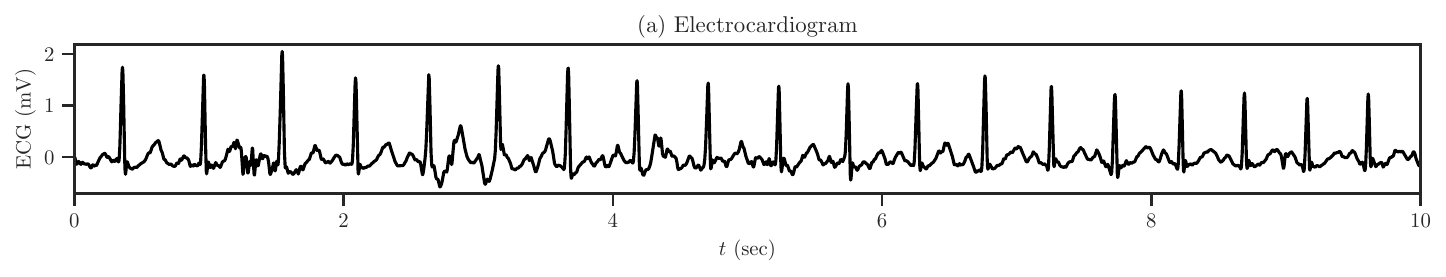}
    \includegraphics[width=\textwidth]{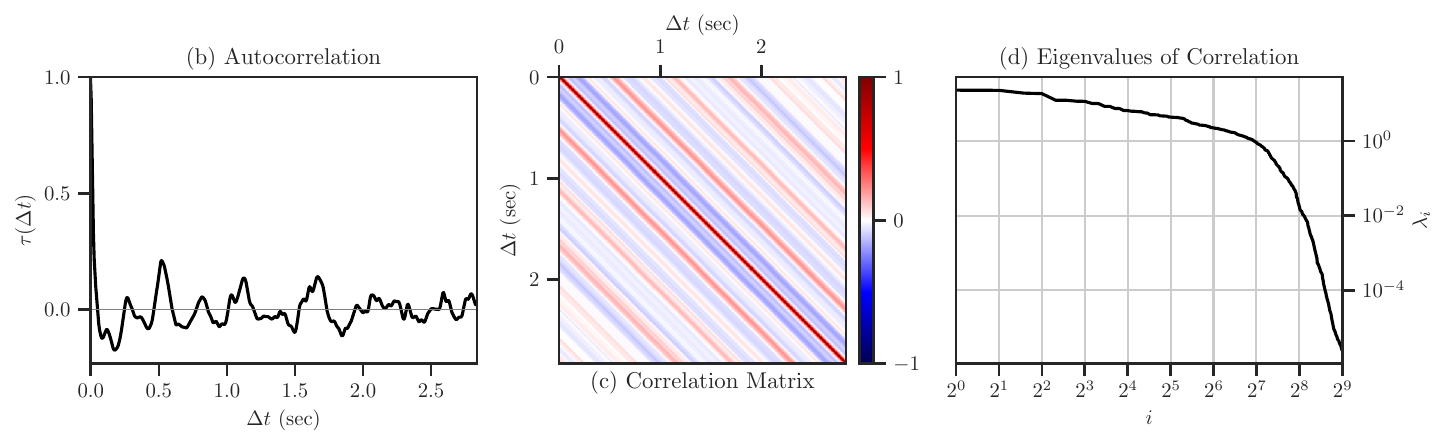}
    \caption{(a) ECG signal. (b) Autocorrelation function and (c) correlation matrix of the ECG signal. (d) Eigenvalues of the correlation matrix.}
    {\phantomsubcaption\label{fig:ecg-a}}
    {\phantomsubcaption\label{fig:ecg-b}}
    {\phantomsubcaption\label{fig:ecg-c}}
    {\phantomsubcaption\label{fig:ecg-d}}
    \label{fig:ecg}
\end{figure}
 
\end{appendices}

% ============
% Bibliography
% ============

% \clearpage
\phantomsection{}
\addcontentsline{toc}{section}{References}

% Smaller References font
% \renewcommand*{\bibfont}{\small}

% \bibliographystyle{plain}
% \bibliographystyle{alpha}
% \bibliographystyle{apalike}
% \bibliographystyle{siam}

% natbib style
% \bibliographystyle{plainnat}
\bibliographystyle{apalike2}

\small
\bibliography{References}

% For biblatex
% \bibliographystyle{apa}
% \printbibliography

\end{document}